\documentclass[11pt,a4paper,reqno]{amsart}
\usepackage{amsfonts}
\usepackage{amsmath,amssymb,amsthm,amsxtra,mathtools}
\usepackage{float}
\usepackage{aliascnt}
\usepackage[colorlinks, linkcolor=blue!72,anchorcolor=orange,
    citecolor=red,urlcolor=Emerald, bookmarksopen,bookmarksdepth=2]{hyperref}
\usepackage[usenames,dvipsnames]{xcolor}
\usepackage{enumitem}
\usepackage{geometry,array}
\usepackage{graphicx}
\usepackage{subfigure}
\usepackage{bookmark}
\usepackage{tikz}
\usepackage{verbatim}
\usetikzlibrary{decorations.pathreplacing}
\usepackage{url}

\usetikzlibrary{matrix,positioning,decorations.markings,arrows,decorations.pathmorphing,	
    backgrounds,fit,positioning,shapes.symbols,chains,shadings,fadings,calc}
\tikzset{->-/.style={decoration={  markings,  mark=at position #1 with
    {\arrow{>}}},postaction={decorate}}}
\tikzset{-<-/.style={decoration={  markings,  mark=at position #1 with
    {\arrow{<}}},postaction={decorate}}}
\usepackage[all]{xy}
\usetikzlibrary{cd}
\usepackage{yhmath}
\usepackage{cleveref}
\usepackage{multirow}

\theoremstyle{plain}
\newtheorem{theorem}{Theorem}[section]

\newaliascnt{proposition}{theorem}
\newtheorem{proposition}[proposition]{Proposition}
\aliascntresetthe{proposition}

\newaliascnt{lemma}{theorem}
\newtheorem{lemma}[lemma]{Lemma}
\aliascntresetthe{lemma}

\newaliascnt{definition}{theorem}
\newtheorem{definition}[definition]{Definition}
\aliascntresetthe{definition}

\newaliascnt{assumption}{theorem}
\newtheorem{assumption}[assumption]{Assumption}
\aliascntresetthe{assumption}

\newaliascnt{remark}{theorem}
\newtheorem{remark}[remark]{Remark}
\aliascntresetthe{remark}

\newaliascnt{construction}{theorem}
\newtheorem{construction}[construction]{Construction}
\aliascntresetthe{construction}
\newaliascnt{corollary}{theorem}            
\newtheorem{corollary}[corollary]{Corollary} 
\aliascntresetthe{corollary}

\newtheorem{propdef}[theorem]{Proposition/Definition}

\numberwithin{equation}{section}

\crefname{theorem}{theorem}{theorems}
\Crefname{theorem}{Theorem}{Theorems}

\crefname{proposition}{proposition}{propositions}
\Crefname{proposition}{Proposition}{Propositions}

\crefname{lemma}{lemma}{lemmas}
\Crefname{lemma}{Lemma}{Lemmas}

\crefname{definition}{definition}{definitions}
\Crefname{definition}{Definition}{Definitions}

\crefname{assumption}{assumption}{assumptions}
\Crefname{assumption}{Assumption}{Assumptions}

\crefname{remark}{remark}{remarks}
\Crefname{remark}{Remark}{Remarks}

\crefname{construction}{construction}{constructions}
\Crefname{Construction}{Construction}{Constructions}
\crefname{corollary}{Corollary}{Corollaries}
\Crefname{corollary}{Corollary}{Corollaries}
\def\hua{\mathcal}
\def\hh{\mathcal}

\def\<{\langle}
\def\>{\rangle}
\def\={\simeq}

\def\to{\rightarrow}

\def\Lten{\stackrel{L}{\otimes}}

\def\ZZ{\mathbb{Z}}

\def\CC{\mathbb{C}}
\def\XX{\mathbb{X}}
\def\on{\operatorname} 
\newcommand\pvd{\operatorname{pvd}}

\newcommand\ind{\operatorname{index}}

\newcommand{\Cone}{\operatorname{Cone}}
\newcommand{\Sim}{\operatorname{Sim}}

\newcommand{\Map}{\operatorname{Map}}

\newcommand{\id}{\operatorname{id}}
\newcommand{\im}{\operatorname{Im}}
\newcommand{\Ho}{\operatorname{Ho}}
\newcommand{\Exit}{\operatorname{Exit}}
\newcommand\Bt[1]{\operatorname{B}_{#1}}
\newcommand{\EG}{\operatorname{EG}}
\newcommand{\EGp}{\EG^\circ}       
\newcommand{\EGb}{\EG^\bullet}       
\newcommand{\CA}{\operatorname{CA}}
\newcommand{\res}{\operatorname{res}}

\newcommand{\indd}{\operatorname{ind}}
\newcommand{\wCA}{\widetilde{\CA}}
\def\uCA{\widetilde{\CA}}
\newcommand*\cocolon{
        \nobreak
        \mskip6mu plus1mu
        \mathpunct{}%
        \nonscript
        \mkern-\thinmuskip
        {:}%
        \mskip2mu
        \relax
}

\newcommand{\cof}{\operatorname{cof}}
\newcommand{\ev}{\operatorname{ev}}

\def\w{\mathbf{w}}

\newcommand\surf{\mathbf{S}}  
\def\surfi{\surf^{\circ}}
\newcommand\sow{\surf_\w}  
\def\surfi{\surf^{\circ}}
\newcommand\subsur{\Sigma}  
\newcommand\colwsur{\overline{\surf}_{\overline{\w}}}  
\newcommand{\Tri}{\Delta}
\newcommand\AS{\mathbb{A}}
\def\D{\hh{D}}
\def\uD{\overline{\hh{D}}}
\def\C{\hh{C}}
\def\h{\hh{H}}
\def\M{\mathbf{M}}
\def\wt{\mathbf{w}}
\newcommand\W{\Delta} 
\def\A{\mathbb{A}}
\def\Ags{\A^{\sharp}_{\wg}}
\def\Agb{\A^{\flat}_{\wg}}
\def\gs{\wg^{\sharp}}
\def\SS{\mathbb{S}}

\newcommand\dAS{\SS} 
\def\lc{\widetilde{c}}
\def\wg{\widetilde{\gamma}}
\def\tg{\widetilde{\gamma}}
\def\we{\widetilde{\eta}}
\def\ws{\widetilde{\sigma}}
\def\wt{\widetilde{\tau}}
\def\wa{\widetilde{\alpha}}
\def\wb{\widetilde{\beta}}
\def\Dw{\D(\sow)}
\def\uA{\overline{\AS}}
\def\uAgs{\uA^{\sharp}_{\gamma}}

\newcommand\hs{\h_S^{\sharp}}
\newcommand\hsb{\h_S^{\flat}}
\newcommand\pss{\psi_S^{\sharp}}
\newcommand\psb{\psi_S^{\flat}}
\def\ue{\overline{\eta}}
\def\cols{\nu}
\def\k{\mathbf{k}}

\def\F{\hh{F}}
\def\ho{\operatorname{h}}

\def\ww{node[white]{$\bullet$}node[red]{$\circ$}}
\def\nn{node{$\bullet$}}

\def\bG{\mathbf{G}}
\def\hF{\hua{F}}
\def\Fs{\hF_{sub}}

\def\pb{\lim\Fs}
\def\glo{\hh{C}}
\newcommand\losec{\hh{L}} 
\def\eSsub{\SS|_{\subsur}}
\def\Ssub{\SS|_{\subsur}}
\def\Fsub{\hF|_{\subsur}}

\def\Squo{\overline{\SS}}
\def\Fquo{\overline{\hF}}

\def\Gs{\bG_{sub}}
\def\Gq{\overline{\bG}}
\def\vq{\overline{v}}
\def\eq{\overline{e}}
\def\Is{I_{sub}}
\def\Iq{\overline{I}}
\def\Es{E_{sub}}
\def\Eq{\overline{E}}
\def\Cat{\operatorname{Cat}}
\def\St{\operatorname{St}}
\def\Prl{\operatorname{Pr}^L_{\St}}
\def\Prr{\operatorname{Pr}^R_{\St}}
\def\Fun{\operatorname{Fun}}
\def\pv{\pvd(\sow)}
\def\pvs{\pvd(\subsur)}
\def\pvq{\pvd(\colwsur)}

\def\Sq{\overline{\surf}}

\def\sl{\prec}
\def\sg{\succ}
\tikzcdset{arrow style=tikz, diagrams={>=stealth}}
\begin{document}
\title{Categorical realization of collapsing subsurfaces and perverse schobers}

\author{Li Fan}
\address{Fl: Department of Mathematical Sciences, Tsinghua University, 100084 Beijing, China.}
\email{fan-l17@tsinghua.org.cn}
\author{Suiqi Lu}
\address{Ls: Department of Mathematical Sciences, Tsinghua University, 100084 Beijing, China.}
\email{lu-sq22@mails.tsinghua.edu.cn}
\maketitle
\begin{abstract}
We study the categorification of collapsed Riemann surfaces with quadratic differentials allowing arbitrary order zeros and poles via the Verdier quotient. We establish an isomorphism between the exchange graph of hearts in the quotient category and the exchange graph of mixed-angulations on the collapsed surface. This extends the work of Barbieri-M\"{o}ller-Qiu-So, who studied Verdier quotients of 3-Calabi-Yau categories and collapsed surfaces without simple poles. We use two methods: a combinatorial approach, and another based on the global sections of a quotient perverse schober. As an application, we describe the Bridgeland stability conditions in terms of quadratic differentials on the collapsed surface.
\end{abstract}
\setcounter{tocdepth}{1}
\tableofcontents
\section{Introduction}
\subsection{Motivation}
\paragraph{\textbf{Stability conditions on triangulated categories}}
Triangulated categories play an important role in representation theory and mathematical physics, most notably through Kontsevich’s homological mirror symmetry conjecture. In particular, the 3-Calabi–Yau categories arising from quivers with potential are closely related to many areas, such as cluster theory, noncommutative differential geometry, and Calabi–Yau geometry, cf. \cite{K}. For 3-Calabi–Yau categories arising from decorated marked surfaces (DMS for short), the symmetry group can be identified with a subgroup of the mapping class group of the associated surface, cf. \cite{Q1}. 

Bridgeland stability conditions on a triangulated category, which were firstly introduced in \cite{B}, are significant in several branches of mathematics and physics, such as algebraic geometry, cluster theory, homological mirror symmetry and string theory, etc. The space of stability conditions on a triangulated category with support property has a complex manifold structure, which can be described by the moduli space of certain type of quadratic differentials on Riemann surfaces, cf. \cite{BS,HKK,KQ2,IQ2,BMQS,CHQ,CHQ2}. Bridgeland-Smith \cite{BS} showed the space of stability conditions on a 3-Calabi-Yau category can be identified with the moduli space of quadratic differentials with simple zeros (Gaiotto-Moore-Neitzke differentials). The space $\on{Stab}(\D)$ is usually non-compact even after projectivization, and several works propose its compactifications, cf. \cite{BCGGM1,BCGGM2,BDL20,Bol20,KKO22,BPPW22}. Barbieri–M\"{o}ller–Qiu–So \cite{BMQS}, followed by \cite{BMS} and \cite{BQ}, generalized the Bridgeland-Smith correspondence to quadratic differentials with arbitrary order zeros and arbitrary higher order poles, which acts as boundary strata in the compactification.
\paragraph{\textbf{Perverse schobers}}
Perverse schobers are conjectural categorification of perverse sheaves introduced by Kapranov-Schechtman \cite{KS}. In \cite{KS2}, the category of perverse sheaves on an oriented surface with boundary is described combinatorially in terms of (co)sheaves on a choice of spanning graph embedded into the surface, which is naturally a ribbon graph. To categorify the perverse sheaves, Merlin Christ \cite{Cginzburg,CSpherical} studied the perverse schobers parametrized by a ribbon graph of some marked surface. When taking the limit of the perverse schober, we get an associated $\infty$-category and hence a triangulated homotopy category. Perverse schobers also play an important role in the areas of (relative) Calabi–Yau structures, topological Fukaya categories and quadratic differential-stability conditions correspondences, cf. \cite{Ccluster, Crelative, CHQ}.

We aim to further study the Bridgeland–Smith correspondence for quadratic differentials with arbitrary order zeros and poles, and generalize the work of Barbieri–M\"{o}ller–Qiu–So, which did not treat quadratic differentials with simple poles. The category we consider arises as a Verdier quotient. Moreover, this work can be viewed as a contribution to the compactification of the space of stability conditions, where the boundary strata correspond to degenerations of quadratic differentials with higher order singularities. Our approach is to establish the isomorphism between two exchange graphs in a purely combinatorial manner, or alternatively to describe the Verdier quotient in terms of a quotient perverse schober.
\subsection{Main results}
We study the categorification of collapsed Riemann surfaces with quadratic differentials allowing arbitrary order zeros and poles via the Verdier quotient. 
We consider the collapsed surface $\colwsur$ of a weighted decorated marked surface $\sow$ with respect to its subsurface $\subsur$ (cf. \Cref{def:collapes}). We assume that the triangulated categories $\D(\sow)$ and $\D(\subsur)$ associated to $\sow$ and $\subsur$ satisfy the three conditions in \Cref{prop:CHQ}. The categorical realization of $\colwsur$ is defined by the following short exact sequence of triangulated categories: 
\begin{equation}\label{eq:intrses}
\begin{tikzcd}
    0 \ar[r] & \D(\subsur) \ar[r] & \Dw \ar[r,"\pi"] &\D(\colwsur) \ar[r] & 0.
\end{tikzcd}
\end{equation}
We show that there is an arc-to-object correspondence on $\colwsur$, that is, for each graded closed arc $\ue$, there is a unique arc object $X_{\ue}$ in $\D(\colwsur)$ (cf. \Cref{prop:objcor}). We particularly focus on the perfectly valued derived category $\pvd(\colwsur)$ and show the following theorem.
\begin{theorem}[\Cref{thm:EGiso}]\label{thm:mainthm}
There is an isomorphism between principal parts of corresponding exchange graphs
\begin{equation}\label{intr:iso}
\EGb(\colwsur)\cong \EGb(\pvq),
\end{equation}
which is compatible with flips and simple tilting.
\end{theorem}
It generalizes Barbieri-M\"{o}ller-Qiu-So's results slightly to the situation where, moreover, simple poles are allowed. The point of the proof is to lift a forward flip in $\colwsur$ to a finite sequence of forward flips in $\sow$ (cf. \Cref{prop:refine_of_flip}), which corresponds to a finite sequence of simple tiltings in $\pv$ and it gives the associated simple tilting in the quotient category $\pvq.$ In particular, $\EGb(\pvq)$ is a union of connected components of $\EG(\pvq)$.

On the other hand, if the category $\D(\sow)$ associated to $\sow$ arises from some perverse schober $\F$ parametrized by some spanning graph $\SS$ of $\sow$, that is, it is given by the triangulated homotopy category of the $\infty$-category $\glo(\SS, \F)$ consisting of global sections of $\F$. Here $\SS$ is the dual graph of a mixed-angulation $\A$ of $\sow$. We restrict $\hh{F}$ to a $\eSsub$-parametrized perverse schober $\Fs$. Then we construct a quotient perverse schober $\Fquo$ parametrized by $\Squo$ (cf. \Cref{prop:sod}) and show the following theorem.
\begin{theorem}[\Cref{thm:exactseq}]\label{thm:mainthm2}
We have a cofiber sequence
\begin{equation}\label{intr:fibco}
\glo({\eSsub},\Fs)\to\glo({\SS},\F)\to\glo({\Squo},\Fquo)
\end{equation}
in the $\infty$-category of stable $\infty$-categories, which induces a short exact sequence of triangulated categories
\begin{equation}\label{intr:fibcotri}
\Ho\big(\glo({\eSsub},\Fs)\big)\to\Ho\big(\glo({\SS},\F)\big)\to\Ho\big(\glo({\Squo},\Fquo)\big)
\end{equation}
by passing to the homotopy categories.
\end{theorem}
The theorem gives a variation of \Cref{thm:mainthm}. As an application, we obtain an isomorphism of complex manifolds
\[
\iota:\on{FQuad}^{\bullet}(\colwsur)\to\on{Stab}^{\bullet}(\pvq)
\]
from the moduli space of framed quadratic differentials associated to the collasped surface to the complex manifold of stability conditions of the associated Verdier quotient (cf. \Cref{thm:quadstab}).
\subsection{Contents}
In \Cref{sec:2}, we introduce some basic knowledge on weighted DMSs, quadratic differentials on a Riemann surface and stability conditions on a triangulated category. In \Cref{sec:3}, we give a categorical description of the collapsed surface and show an isomorphism between two exchange graphs (cf. \Cref{thm:EGiso}). In \Cref{sec:4}, we briefly introduce the perverse schober and its semiorthogonal decomposition, construct the quotient perverse schober for the collapsed surface and provide a variation of \Cref{thm:EGiso}. In \Cref{sec:5}, we give a correspondence between the moduli space of framed quadratic differentials and the complex manifold of stability conditions of the associated category of the collapsed surface as an application.
\subsection{Acknowledgement}
We are grateful to Yu Qiu for suggesting this problem and for many helpful discussions. Part of this work was written during Fl's six-month research visit at Universit\'e Paris Cit\'e under the supervision of Bernhard Keller. Fl thanks to Merlin Christ for thoughtful discussions on perverse schobers on surfaces and Bernhard Keller for precious comments. This work is supported by National Natural Science Foundation of China (Grant No. 12425104).
\section{Backgrounds on surfaces with quadratic differentials}\label{sec:2}
\subsection{Graded weighted decorated marked surfaces}\label{sec:sow}
We follow \cite{HKK,CHQ} to introduce the graded weighted DMS, where some extra data are added. Let $\surf$ be a smooth, compact oriented 2-dimensional manifold with nonempty boundary $\partial\surf$. For simplicity, we may assume that $\surf$ is also connected. An \emph{arc} $c$ in $\surf$ is a curve $c:[0,1]\to \surf$ such that $c(t)\in\surfi$ for any $t\in(0,1)$, where $\surfi:=\surf\backslash\partial\surf$. The \emph{inverse} $\overline{c}$ of $c$ is defined as $\overline{c}(t)=c(1-t)$ for any $t\in[0,1]$. If the endpoints of $c$ coincide, we require that $c$ is not homotopic to a point in $\surf$ relative to its endpoints. In this paper, we always consider arcs up to taking inverse and homotopy relative to endpoints.
\paragraph{\textbf{Graded surfaces and graded curves}}
\begin{definition}
We call $\surf$ a \emph{graded surface} if there is a foliation $\nu$ on it, which is a section of the projectivized tangent bundle $\mathbb{P}(T\surf)$. A morphism between graded surfaces is an orientation preserving local diffeomorphism which is compatible with the foliations. 
\end{definition}
We take more interest in the case when grading structures emerge as horizontal foliations of quadratic differentials. The foliation $\mathrm{hor}(\varphi)$ of $\mathbf{S}$ is defined away from the zeros and poles of a quadratic differential $\varphi$ as 
\[
\mathrm{hor}(\varphi)(p) \coloneqq \{ v \in T_p \mathbf{S} \mid \varphi(v, v) \in \mathbb{R}_{\geq 0} \}  \in \mathbb{P}(T_p \mathbf{S}).
\]  
For an immersed loop \( c \colon S^1 \to \mathbf{S} \) on a graded surface \((\mathbf{S}, \nu)\), the \textit{Maslov index} \(\mathrm{ind}_\nu(c)\) counts the number of times that the grading \(\nu\) becomes tangent to \(c\), normalized as an unsigned integer. At a point \( p \in \surfi\), we define \(\mathrm{ind}_\nu(p) \coloneqq \mathrm{ind}_\nu(c) \), where \(c\) is a small clockwise loop around \( p \). When a surface has a fixed foliation, we often omit \(\nu\) and use \(\indd(p)\) or \(\indd(c)\) instead. The Maslov index of a loop \(c\) coincides with its winding number, given by the angle change integral in the sense of Definition 6.23 in \cite{IQ2}.

For a curve $c:[0,1]\to \surf$, there is a canonical section $s_c: c\to\mathbb{P}(T\surf)$ given by $s_c(z)=T_zc$. A \emph{grading} on $c$ is given by a lift $\lc$ of $s_c$ to $\mathbb{R}(T\surf)$. The pair $(c,\lc)$ is called a \emph{graded curve}, and we usually denote it by $\lc$. Notice that there are $\ZZ$ many lifts of $c$.
\begin{definition}\label{def:gradedcurve}
For two graded curves $\lc_1$ and $\lc_2$ intersecting transversely at a point $p=c_1(t_1)=c_2(t_2)$, the \emph{intersection index} $\ind_p(\lc_1,\lc_2)$ is defined as $\lc_1(t_1)\cdot\kappa\cdot\lc_2(t_2)^{-1}$, representing the class in $\pi_1(\mathbb{P}(T_p\surf))=\mathbb Z$. Here, $\kappa$ is the path from $\dot{c}_1(t_1)$ to $\dot{c}_2(t_2)$, determined by clockwise rotation in $T_p\surf$ by an angle less than $\pi$.
\end{definition}
\begin{lemma}[{\cite[Lemma 3.10 and Proposition 3.12]{FQ}}]\label{lem:3intindec}
The arcs appearing below are all graded arcs in $\surf$ and these intersection points may be either endpoints or points in the interior of the arcs.
\begin{itemize}
\item If $\wa,\wt,\ws$ share the same point $p$ and sitting in clockwise order, see the left and middle one in \Cref{fig:3int1}, then we have
\begin{equation}\label{eq:3intindec}
\ind_{p}(\wa,\ws)=\ind_{p}(\wa,\wt)+\ind_{p}(\wt,\ws).
\end{equation}
\item If $\wa_1,\wa_2,\ldots,\wa_n$ form a contractible triangle, where there are no boundaries and genus, see the right on in \Cref{fig:008}, then we have
\begin{equation}\label{eq:sum1}
\ind_{p_1}(\wa_1,\wa_2)+\ind_{p_2}(\wa_2,\wa_3)+\cdots+\ind_{p_n}(\wa_n,\wa_1)=n-2.
\end{equation}
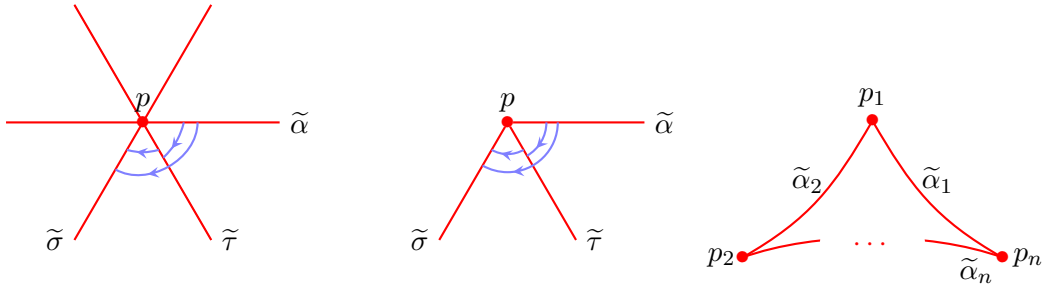
\begin{figure}[h]\centering
	\begin{tikzpicture}[scale=.6]
	\foreach \j in {0,1,2}{
		\draw[red,thick] (180+120*\j:3)to(120*\j:3);}
	\draw[red](120:3)to(-60:3) (0,0)node{$\bullet$};
\draw(-60:3)node[right]{$\wt$}  (3,0)node[right]{$\wa$}(-120:3)node[left]{$\ws$};
	\draw[blue!50,thick,->-=.7,>=stealth]	(0:1.2)to[bend left=60](-120:1.2);
	\draw[blue!50,thick,->-=.7,>=stealth]	(0:.9)to[bend left=15](-60:.9);
	\draw[blue!50,thick,->-=.7,>=stealth]	(-60:.7)to[bend left=15](-120:.7);
	\draw(0,0)node[above]{$p$};
\begin{scope}[shift={(8,0)}]
\foreach \j in {0,1,2}{\draw[red, thick] (240+60*\j:0)to(240+60*\j:3);}
\draw[white] (0,0)node{$\bullet$};
\draw[red] (0,0)\nn;
\draw(300:3)node[right]{$\wt$} (3,0)node[right]{$\wa$}(-120:3)node[left]{$\ws$};
\draw[blue!50,thick,->-=.7,>=stealth]	(0:1.1)arc(0:-120:1.1)(-120:1.1);
\draw[blue!50,thick,->-=.7,>=stealth](0:.85)arc(0:-60:.85)(-60:.85);
\draw[blue!50,thick,->-=.7,>=stealth](-60:.7)arc(-60:-120:.7)(-120:.7);
\draw (0,0)node[above]{$p$};
\end{scope}
\begin{scope}[shift={(16,3)},scale=.57]
	\draw[red,thick] (-90:5.2) .. controls +(-120:3) and +(30:3) .. (-5,-10.5);
	\draw[red,thick] (-90:5.2) .. controls +(-60:3) and +(150:3) .. (5,-10.5);
	\draw[red, thick] (-5,-10.5) .. controls +(20:3) and +(160:3) .. (5,-10.5);
\draw[white,fill = white](0,-10) circle (2);
		\draw[red] (-90:5.2)node{$\bullet$};
	\draw[red] (-5,-10.5)node{$\bullet$};
	\draw[red] (5,-10.5)node{$\bullet$};
    \draw (4,-11) node {$\wa_n$} ;
    \draw (2.5,-7.5) node {$\wa_1$} ;
    \draw(-2.5,-7.5) node {$\wa_2$} ;
    \draw (0,-4.3) node {$p_1$} ;
    \draw (-5.8,-10.5) node {$p_2$} ;
    \draw (6,-10.5) node {$p_n$} ;
\draw[red] (0,-10)node{$\cdots$};
\end{scope}
	\end{tikzpicture}
	\caption{Graded arcs intersect at some points}\label{fig:3int1}
\end{figure}
\end{itemize}
    \end{lemma}
\begin{definition}\label{def:smoothingout}
Let $\wa$ and $\wb$ be two graded arcs. We define the graded arc $\tg=\wa\wedge_{z_1}\wb$ as the smoothing out $\wa\cup\wb[n]$ at $\wa(0)=\wb(0)$ along $z_1$, connecting $\wa(1)$ and $\wb(1)$, cf. \Cref{fig:008}, where $n=\ind_{z_1}(\wa,\wb)$.
\begin{figure}[htbp]
	\begin{tikzpicture}[scale=.35]
	\draw[red, thick] (-90:5.2) .. controls +(-120:3) and +(30:3) .. (-5,-10.5);
	\draw[red,thick] (-90:5.2) .. controls +(-60:3) and +(150:3) .. (5,-10.5);
	\draw[cyan, ultra thick] (-5,-10.5) .. controls +(20:3) and +(160:3) .. (5,-10.5);
		\draw[red] (-90:5.2)\nn;
	\draw[red] (-5,-10.5)\nn;
	\draw[red] (5,-10.5)\nn;
    \draw (-90:10.8) node {$\tg$} ;
    \draw (2.5,-7.5) node {$\wa$} ;
    \draw (-2.5,-7.5) node {$\wb$} ;
    \draw (0,-4.3) node {$z_1$} ;
    \draw (-5.8,-10.5) node {$z_2$} ;
    \draw (5.8,-10.5) node {$z_3$} ;
	\end{tikzpicture}
	\caption{Contractible triangle}\label{fig:008}
    \end{figure}
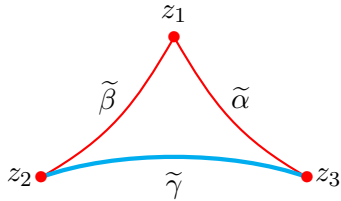
\end{definition}
We notice that in this case, $\wa,\wb$ and $\wg$ form a contractible triangle in $\surf$.
\paragraph{\textbf{Weighted decorated marked surface}}
\begin{definition}\label{def:wdms}
A \emph{weighted decorated marked surface} (weighted DMS for short) is the data $\sow=(\surf,\M,\W,\w)$ where
\begin{itemize}
\item $\surf$ is equipped with a grading $\nu$ on $\surf\setminus (\M\cup \W)$,
\item $\M\subset {\partial\surf}$ is a finite non-empty subset of \textit{marked points},
\item $\W\subset \surf$ is a finite non-empty subset of \textit{decorations},
\item $\w\colon \W\to \ZZ_{\geq -1}$ is the \textit{weight function};
\end{itemize}
such that
\begin{itemize}
\item $\M\cap \W=\emptyset$,
\item each boundary component of $\sow$ has at least one marked point in $\M$,
\item for each $x\in\W\cap \surfi$, we have that $\mathrm{ind}_\nu(x)$ equals to $\w(x)+2$.
\end{itemize}
\end{definition}

We call an arc \emph{open} (resp. \emph{closed}) if its endpoints are in $\M$ (resp. $\Tri$). A closed arc is said to be \emph{simple} if there is no self-intersection except at endpoints. 
\paragraph{\textbf{Mixed-angulation and its dual graph}}
\begin{definition}
A \emph{mixed-angulation} $\A$ of a weighted DMS $\sow$ is a finite set of graded open arcs on $\sow$ such that no two of them (including self-pairs) intersect within $\surfi\backslash\Tri$. These arcs divide $\sow$ into once-decorated polygons, where each decoration $z$ with weight $\w(z)$ is enclosed in a $(\w(z)+2)$-gon. We denote such a $(\w(z)+2)$-gon by $\A(z)$ and refer to as an $\A$-polygon. Furthermore, if $\widetilde{\gamma}_1$ and $\widetilde{\gamma}_2$ are two adjacent edges of an $\A$-polygon, lying in clockwise order and intersecting at some point $p$ in $\M$, then their index satisfies $\ind_p(\widetilde{\gamma}_1,\widetilde{\gamma}_2) = 0$.
\end{definition}
An arc $\wg$ in $\A$ is called a \emph{monogon arc} if the area bounded by $\wg$ is homotopic to a disk. In this case, there are no other arcs cutting the angle in this monogon. For example, the light blue arc in the left one in \Cref{fig:flip2} is a momgon arc while one in the left one in \Cref{fig:flip1} is not.
\begin{proposition}\label{prop:combform}
Suppose that $\sow$ is of genus $g$ with $m$ marked points and $b$ boundary components. If the mixed-angulation $\A$ give a decomposition of $\sow$ into $k$ polygons where the $i$-th one has $s_i$ sides and a decoration with weight $w_i:=s_i-2$. Then we have
\begin{equation}\label{eq:combform}
\sum_{i=1}^k w_i - (m + 2b)=4g-4.
\end{equation}
\end{proposition}
\begin{proof}
Each polygon side is either an internal arc (counted twice) or a boundary segment (between consecutive marked points, counted once), so we have
$$\sum_{i=1}^k s_i = 2I + m,$$
where $I$ is the number of internal arcs. Hence
$$\sum_{i=1}^k w_i=\sum_{i=1}^k (s_i-2)=2I+m-2k.$$
Euler’s formula for the cell decomposition with $V=m$, $E=I+m$, $F=k$ gives
$$V-E+F = m-(I+m)+k = k-I = \chi(\sow)=2-2g-b,$$
where $\chi(\sow)$ is the Euler characteristic of $\sow$. Therefore, we have
$$\sum_{i=1}^k w_i-(m+2b)=\bigl(2I+m-2k\bigr)-(m+2b)=2(I-k-b)=2(2g-2)=4g-4,$$
which completes the proof.
\end{proof}
\begin{definition}\label{def:wgflip}
Let $\AS$ be a mixed-angulation of $\sow$. The \emph{forward flip} $\Ags$ of $\A$ with respect to a graded arc $\wg\in\A$ whose underlying arc is defined as follows:
\begin{itemize}
\item If $\wg$ is not a monogon arc, the new open arc $\gs$ is obtained by moving the endpoints of $\wg$ counterclockwise along the $\A$-gons containing $\wg$, as shown in Figure \ref{fig:flip1}.
\item If $\wg$ is a monogon arc, the new open arc $\gs$ is obtained by moving both the endpoints of $\wg$ together counterclockwise along the $\A$-gons containing $\wg$, as shown in Figure \ref{fig:flip2}.
\end{itemize}
For the grading, we choose some grading of $\gs$ such that the index-zero condition is satisfied. We obtain a new mixed-angulation $\Ags$ by replacing $\wg$ with $\gs$. The \emph{backward flip} $\Agb$ is defined as the inverse operation of a forward flip.
\end{definition}
\begin{figure}[ht]\centering
\makebox[\textwidth][c]{
\begin{tikzpicture}[xscale=1.4,yscale=1.4]
\begin{scope}
\draw[very thick] (0,0) circle (2);
\draw[very thick,fill = gray!10](0,-.5) circle (0.5);
\draw[black] (0:2)\nn (90:2)\nn (150:2)\nn (210:2)\nn (270:2)\nn;
\draw[black] (0:0)\nn (270:1)\nn;
\draw[blue,thick] (90:2)to[out=-150,in=90] (180:1.5)to[out=-90,in=150] (270:2);
\draw[blue,thick] (90:2)to[out=-30,in=90] (0:1.5)to[out=-90,in=30] (270:2);
\draw[blue,thick] (0,-1)--(0,-2);
\draw[blue,thick] (0,0)to[out=120,in=90] (-.95,-.4) to[out=-90,in=120] (0,-2);
\draw[red,thick] (0,-.5) circle (0.7);
\draw[red,thick] (-1.6,0.6) -- (0,1.2) -- (1.6,0.6);
\draw[red,thick] (0,1.2)to[out=-50,in=80] (.7,-.5);
\draw[red,font=\scriptsize] (-1.6,0.6)\ww node[left]{$2$} (1.6,0.6)\ww node[right]{$0$} (0,1.2)\ww node[above]{$1$}(-.7,-.5)\ww node[left]{$1$} (.7,-.5)\ww node[right]{$2$};
\draw[cyan,thick](-90:2) to[out=145,in=-110] (-1,0) to[out=70,in=180] (0,.6) to[out=0,in=110] (1,0) to[out=-70,in=35] (0,-2);
\end{scope}
\draw(3,0)node{\Huge{$\rightsquigarrow$}};
\begin{scope}[shift={(6,0)}]
\draw[very thick] (0,0) circle (2);
\draw[very thick,fill = gray!10](0,-.5) circle (0.5);
\draw[black,thick] (0:2)\nn (90:2)\nn (150:2)\nn (210:2)\nn (270:2)\nn;
\draw[black] (0:0)\nn (270:1)\nn;
\draw[blue,thick] (90:2)to[out=-150,in=90] (180:1.5)to[out=-90,in=150] (270:2);
\draw[blue,thick] (90:2)to[out=-30,in=90] (0:1.5)to[out=-90,in=30] (270:2);
\draw[blue,thick] (0,-1)--(0,-2);
\draw[blue,thick] (0,0)to[out=120,in=90] (-.95,-.4) to[out=-90,in=120] (0,-2);
\draw[red,thick] (-.7,-.5) arc (180:360:.7);
\draw[red,thick] (-1.6,0.6) -- (0,1.2);
\draw[red,thick] (.7,-.5)-- (1.6,0.6);
\draw[red,thick] (0,1.2)to[out=-50,in=80] (.7,-.5);
\draw[red,thick] (0,1.2)to[out=-130,in=90] (-.7,-.5);
\draw[red,font=\scriptsize] (-1.6,0.6)\ww node[left]{$2$} (1.6,0.6)\ww node[right]{$0$} (0,1.2)\ww node[above]{$1$}(-.7,-.5)\ww node[left]{$1$} (.7,-.5)\ww node[right]{$2$};
\draw[cyan,thick](0,0) to[out=70,in=-90] (.3,1) to[out=90,in=-70] (0,2);
\end{scope}
\end{tikzpicture}}
\caption{The forward flip at a usual arc}
\label{fig:flip1}
\end{figure}
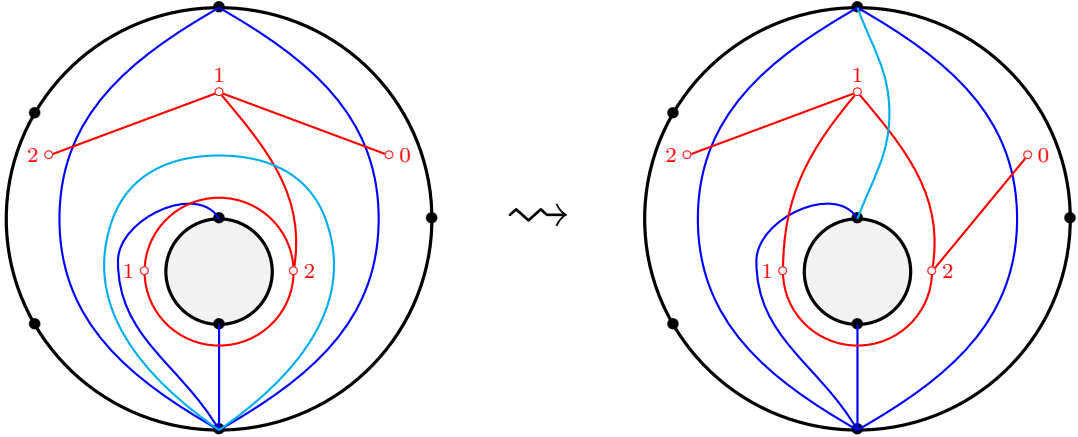

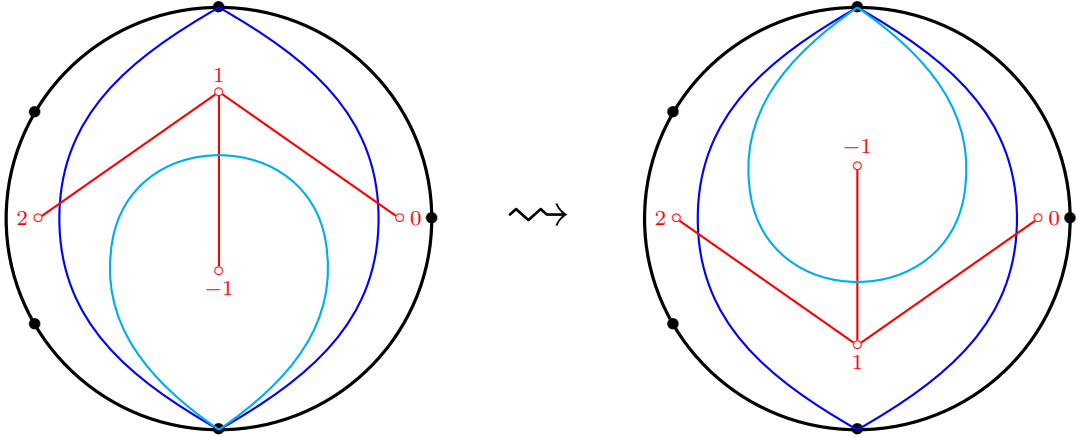
\begin{figure}[hb]\centering
\makebox[\textwidth][c]{
\begin{tikzpicture}[xscale=1.4,yscale=1.4]
\begin{scope}
\draw[very thick] (0,0) circle (2);
\draw[black] (0:2)\nn (90:2)\nn (150:2)\nn (210:2)\nn (270:2)\nn;
\draw[blue,thick] (90:2)to[out=-150,in=90] (180:1.5)to[out=-90,in=150] (270:2);
\draw[blue,thick] (90:2)to[out=-30,in=90] (0:1.5)to[out=-90,in=30] (270:2);
\draw[red,thick] (-1.7,0) -- (0,1.2) -- (1.7,0);
\draw[red,thick] (0,1.2)-- (0,-.5);
\draw[red,font=\scriptsize] (-1.7,0)\ww node[left]{$2$} (1.7,0)\ww node[right]{$0$} (0,1.2)\ww node[above]{$1$} (0,-.5)\ww node[below]{$-1$};
\draw[cyan,thick](-90:2) to[out=145,in=-100] (-1,-.2) to[out=80,in=180] (0,.6) to[out=0,in=100] (1,-.2) to[out=-80,in=35] (0,-2);
\end{scope}
\draw(3,0)node{\Huge{$\rightsquigarrow$}};
\begin{scope}[shift={(6,0)}]
\draw[very thick] (0,0) circle (2);
\draw[black] (0:2)\nn (90:2)\nn (150:2)\nn (210:2)\nn (270:2)\nn;
\draw[blue,thick] (90:2)to[out=-150,in=90] (180:1.5)to[out=-90,in=150] (270:2);
\draw[blue,thick] (90:2)to[out=-30,in=90] (0:1.5)to[out=-90,in=30] (270:2);
\draw[red,thick] (-1.7,0) -- (0,-1.2) -- (1.7,0);
\draw[red,thick] (0,-1.2)-- (0,.5);
\draw[red,font=\scriptsize] (-1.7,0)\ww node[left]{$2$} (1.7,0)\ww node[right]{$0$} (0,-1.2)\ww node[below]{$1$} (0,.5)\ww node[above]{$-1$};
\draw[cyan,thick](0,2) to[out=-145,in=100] (-1,.2) to[out=-80,in=180] (0,-.6) to[out=0,in=-100] (1,.2) to[out=80,in=-35] (0,2);
\end{scope}
\end{tikzpicture}
}
\caption{The forward flip at a monogon arc}
\label{fig:flip2}
\end{figure}
\begin{definition}\label{def:EG}
The \emph{exchange graph} $\EG(\sow)$ of a weighted DMS $\sow$ is a directed graph whose vertices correspond to mixed-angulations, with oriented edges corresponding to forward flips between them. We choose an initial mixed-angulation $\AS_0$ of $\sow$, and denote by $\EGp(\sow)$ the connected component of $\EG(\sow)$ that contains $\AS_0$. 
\end{definition}
\begin{remark}
When we do a flip with respect to some $\wg\in\AS_0$, we let $\gs$ be the smoothing out of $\wg[1], \wa$ and $\wb$ where $\wa$ and $\wb$ are two graded arcss in $\A$-gon containing $\wg$ whose endpoints are the endpoint of $\wg$ and $\gs$ respectively. By \Cref{lem:3intindec}, we have that any (clockwise) index between two adjacent arcs in $\AS_{0,\wg}^{\sharp}$ is zero. Thus, this phenomenon also holds in any $\AS\in\EGp(\sow)$. 
\end{remark}
We introduce some basic terminology on graphs. 
\begin{definition}
Let $\bG$ be a graph. 
\begin{itemize}
\item A \emph{half-edge} in $\bG$ is an oriented edge fragment that has only one endpoint, meaning along with its twin, defines one full edge. An edge is called \emph{internal} if it can be formed by pairing two half-edges and called \emph{external} if it is a remain unpaired half-edge.
\item $\bG$ is called a \emph{ribbon graph} if a cyclic order of half-edges around each vertex of $\bG$ is specified.
\item $\bG$ is called a \emph{spanning graph} of $\surf$ if the inclusions $\bG\subseteq\surf$ is a homotopy equivalence.
\end{itemize}
\end{definition}
Notice that the spanning graph of $\surf$ is automatically a ribbon graph where the cyclic order is given by the orientation of surface $\surf$.
\begin{definition}
The \emph{dual graph} $\SS=\A^*$ of a mixed-angulation $\A$ of $\sow$ is a S-graph \footnote{An S-graph is a graph where at each vertex, there is a cyclic or total order on the set of half-edges meeting at the vertex and a positive integer for each pair of successive half-edges. For more details, we refer to Definition 2.7 in \cite{CHQ}.} consisting of graded closed arcs in $\sow$ with intersections only at endpoints, such that each graded open arc $\wg$ in $\AS$ is \emph{dual} to a unique edge $\we$ of the S-graph in the sense that the two arcs intersect in exactly one point $z$ in the interior with $\ind_z(\wg,\we)=0$.
\end{definition}
\begin{remark}\label{rmk:ind}
Let $\wg$ and $\widetilde{\xi}$ be two graded open arcs in a mixed-angulation $\AS$ of $\sow$ with $\we$ and $\widetilde{\delta}$ the dual arcs of $\wg$ and $\widetilde{\xi}$ respectively. If the index satisfies $\ind(\wg,\widetilde{\xi})=d$, then by the second statement of \Cref{lem:3intindec}, the index of the closed arc satisfies $\ind(\widetilde{\delta},\we)=1-d$.
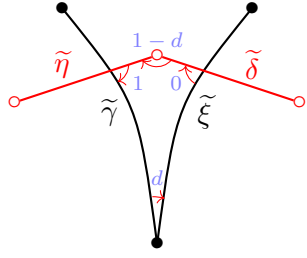
\begin{figure}
\begin{tikzpicture}[scale=1.25]
\draw[thick] (0,-.5) .. controls (-.2,1) .. (-1,2);
\draw[thick] (0,-.5) .. controls (.2,1) .. (1,2);
\draw[thick] (-.5,.9)node{$\wg$} (.5,.9)node{$\widetilde{\xi}$};
\draw[thick,red] (-1,1.4) node{$\we$} (1,1.4) node{$\widetilde{\delta}$};
\draw[thick,red] (0,1.5) -- (-1.5,1);
\draw[thick,red] (0,1.5) -- (1.5,1);
\draw[blue!50,font=\scriptsize] (0,1.65)node{$1-d$} (0,0.2) node{$d$} (-.2,1.2)node{$1$} (.2,1.2) node{$0$};
\coordinate (A) at (0,-.5); 
\coordinate (B) at (0,1.5);      
\coordinate (C) at (-.5,1.5);       
\coordinate (D) at (.5,1.5);  

\draw[->, red] (B) ++(0.15,-0.05) arc (-40:-140:0.2); 
\draw[->, red] (C) ++(0.2,-.1) arc (0:-65:0.2);    
\draw[->, red] (D) ++(-.1,-.3) arc (-120:-185:0.2); 
\draw[->, red] (A) ++(-.05,0.5) arc (90:75:0.5);     
\draw[red](0,1.5)\ww (-1.5,1)\ww (1.5,1)\ww;
\draw(0,-.5)\nn (-1,2)\nn (1,2)\nn;
\end{tikzpicture}
    \caption{Index of open graded arcs induce index of their dual graded arcs}
    \label{fig:2.13}
\end{figure}
\end{remark}
\begin{figure}[ht]\centering
\begin{tikzpicture}[scale=.3]
  \draw[very thick,NavyBlue](0,0)circle(6)node[above,black]{$\we$};
  \draw(-120:5)node{+};
  \draw(-2,0)edge[red, very thick](2,0)  edge[cyan,very thick, dashed](-6,0);
  \draw(2,0)edge[cyan,very thick,dashed](6,0);
  \draw(-2,0)node[white] {$\bullet$} node[red] {$\circ$};
  \draw(2,0)node[white] {$\bullet$} node[red] {$\circ$};
  \draw(0:7.5)edge[very thick,->,>=stealth](0:11);
\end{tikzpicture}\;
\begin{tikzpicture}[scale=.3,yscale=-1]
	\draw[very thick,NavyBlue](0,0)circle(6)node[above,black]{$\we$};
	\draw[red, very thick](-2,0)to(2,0);
	\draw[cyan,very thick, dashed](2,0).. controls +(0:2) and +(0:2) ..(0,-2.5)
	.. controls +(180:1.5) and +(0:1.5) ..(-6,0);
	\draw[cyan,very thick,dashed](-2,0).. controls +(180:2) and +(180:2) ..(0,2.5)
	.. controls +(0:1.5) and +(180:1.5) ..(6,0);
	\draw(-2,0)node[white] {$\bullet$} node[red] {$\circ$};
	\draw(2,0)node[white] {$\bullet$} node[red] {$\circ$};
	\end{tikzpicture}
\caption{The braid twist $\Bt{\we}$}
\label{fig:Braid twist}
\end{figure}
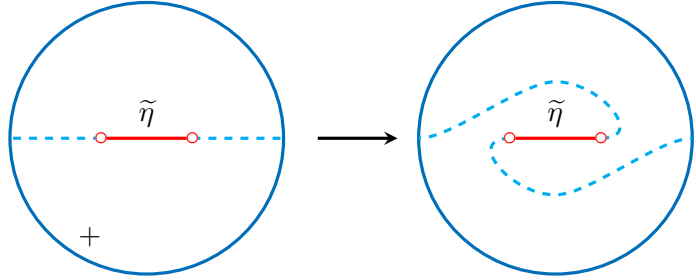
We denote the set of graded closed arcs on $\sow$ by $\uCA(\sow)$. For any closed arc $\we\in\uCA(\sow)$, there is a (positive) \emph{braid twist} $\Bt{\we}$ along $\we$, which is shown in Figure~\ref{fig:Braid twist}. More precisely, for a graded closed arc $\wa$, the braid twist action $B_{\we}(\wa)$ is the graded arc $\wa\wedge\we$.
\begin{definition}\label{rem:f.f.}
Let $\we\in\SS$ be dual to $\wg\in\AS$. We define the \emph{forward flip} of $\SS$ at $\we$ as the new set $\SS^\sharp_{\we}=(\Ags)^*:=\{\wa_{\we}^{\sharp}|\wa\in\SS\}.$
The arc $\we$ is still in $\SS^\sharp_{\we}$ but its grading will be shifted by $1$. For $\wa\neq\we\in\SS$, we flip it as follows.
\begin{itemize}
\item In the usual case, if $\wa$ satisfies $\ind_z(\wa,\we)=1$ for some endpoint $z\in\W$ of $\we$, that is, $\wa$ is the nearest counter-clockwise closed arc from $\we$ around $z$, then $\wa^{\sharp}_{\we}:=\wa\wedge_z\we(=B_{\we}(\wa))$, i.e. move the endpoint $z$ of $\wa$ along $\we$, cf. Figure~\ref{fig:flip1}. 
\item In the monogon case, if $\wa$ satisfies $\ind_z(\wa,\we)=1$ for some endpoint $z\in\W$ of $\we$, then $\wa^{\sharp}_{\we}:=B_{\we}^2(\wa)$, cf. Figure~\ref{fig:flip2}.
\end{itemize}
Similarly, we have the definition of \emph{backward flip}.
\end{definition}
\subsection{Quadratic differentials}
In this subsection, we refer to \cite{S} for the concepts of quadratic differentials.

For a compact Riemann surface $C$, let $K_C$ be the holomorphic cotangent bundle over $C$. A \emph{holomorphic/meromorphic quadratic differential} over $C$ is a holomorphic/meromorphic section of $K_C^{\otimes 2}$. Let $\varphi$ be a meromorphic quadratic differential on $C$. We denote the set of \emph{critical points} (or \emph{singularities}) of $\varphi$ by $\on{Crit}(\varphi)$, consisting of zeros and poles. The zeros and simple poles are called \emph{finite critical points}, while the poles of order $w\geq 2$ are called higher order poles, or \emph{infinite critical points}.

On the complement $C \backslash\on{Crit}(\varphi)$, $\varphi$ is holomorphic and non-vanishing. Moreover, $|\varphi|$ is a flat Riemannian metric on the complement $C \backslash \on{Crit}(\varphi)$, which means that $C \backslash\on{Crit}(\varphi)$ is locally isometric to an Euclidean plane under this metric. But this metric is typically not complete. In order to get its complement, only need to add a conical point with cone angle $(w+2)\pi$ at each finite critical point, cf. \cite{HKK} for more details. The horizontal foliation of $\varphi$, denoted by $\on{hor}(\varphi)$, is obtained by pulling back the horizontal foliation on the Euclidean plane.

A maximal straight arc for the metric $|\varphi|$ in a direction $\theta\in S^1$ is called a \emph{trajectory} in the direction $\theta$. Near each finite critical point, there will be $w+2$ distinguished directions that are limits of a trajectory, while near each infinite critical point, there will be $|w|-2$ distinguished directions that are limits of a trajectory, where $w$ is the order of a zero or pole. There are some interesting types of trajectories:
\begin{itemize}
\item A \emph{saddle connection} is a trajectory which converges to conical points in both directions.
\item A \emph{saddle trajectory} is a saddle connection in the horizontal direction.
\item A \emph{separating trajectory} is a trajectory which converges to a conical point in a direction and approaches to an infinite critical point in the other direction.
\item A \emph{recurrent trajectory} is a trajectory which is recurrent in at least one direction.
\item A \emph{generic trajectory} is a trajectory which approaches to infinite critical points in both directions.
\end{itemize}

Now we fix the directions to be the horizontal directions, and the trajectories are horizontal trajectories. 

The quadratic differential $\varphi$ is called \emph{saddle-free} if there is no saddle trajectory. In this case, there are only separating and generic trajectories. Each separating trajectory approaches to a conical point, and a conical point with cone angle $n\pi$ gives $n$ separating trajectories. These separating trajectories divide the surface $C$ into \emph{horizontal strips}. Each horizontal strip is isometric to $(0,a)\times\mathbb R$, where $a\in(0,+\infty]$ is called the \emph{height}. This gives a \emph{horizontal strip decomposition} of the surface. 

Let $(C,\varphi)$ be a flat surface, where $\varphi$ is a meromorphic quadratic differential on $C$. Then the \emph{weighted decorated real blow-up} of $(C,\varphi)$ is a wDMS $\sow=(\surf, \M, \W, \w)$ constructed as the following (cf. \cite{BMQS}):
\begin{itemize}
\item Each boundary component $\partial_p$ of $\sow$ corresponds to an infinite critical point $p$ of $\varphi$.
\item For each $\partial_p$, the number of marked points on $\partial_p$ equals to $|w|-2$, where $w$ is the order of the infinite critical point $p$.
\item Each decoration in the $\W$ corresponds to a finite critical point of $\varphi$, whose weight equals to the order of the corresponding critical point.
\end{itemize}
Furthermore, the foliation $\nu$ on the wDMS $\sow$ is the horizontal foliation of $\varphi$. In this case, $\sow$ becomes a graded wDMS. Let the weighted decorated real blow-up of $(C,\varphi)$ is denoted by $\sow(C,\varphi)$.

Now we recall the definition of $\sow$-framed quadratic differentials. Let $\sow$ be a wDMS. The \emph{moduli space of $\sow$-framed quadratic differentials} $\on{FQuad}(\sow)$ is defined as follows. A point in it is represented by a flat surface $(C,\varphi)$ with an isomorphism of wDMSs $f:\sow\to\sow(C,\varphi)$. Two points $(C,\varphi,f)$ and $(C',\varphi',f')$ are equivalent if there exists an isomorphism of wDMSs $g:\sow(C,\varphi)\to\sow(C',\varphi)$ such that $(f')^{-1}gf$ is isotopic to identity as morphisms between wDMSs.
\subsection{Hearts and stability conditions}
For foundational material on (bounded) t-structures, hearts, and tilting theory in triangulated categories, we refer to \cite{KQ1}. A triangulated category $\D$ is said to have a \emph{finite heart} $\h$ if $\h$ is of finite length and only has finitely many simple objects, denoted by $\Sim\h$. The \emph{(total) exchange graph} $\EG(\D)$ is an oriented graph where vertices represent hearts in $\D$, and directed edges correspond to simple forward tiltings between them. Here, we consider simple tiltings without rigid assumption.
\begin{propdef}[{\cite[Proposition A.2]{CHQ}}]\label{propdef}
Let $S$ be a simple object in a finite-length heart $\h \in \EG(\D)$. Then $S$ is functorially finite in $\h$. Furthermore, if $\h$ is finite, both forward and backward simple tiltings exist and remain finite, given by:
\begin{equation}\label{eq:forf}
\Sim\hs=\{S[1]\}\cup\{\pss(X) \mid X\in\Sim\h, X\neq S\}
\end{equation}
and
\begin{equation}
\Sim\hsb=\{S[-1]\}\cup\{\psb(X) \mid X\in\Sim\h, X\neq S\}.
\end{equation}
Here, $\pss(X):=\Cone(f)[-1]$ (resp. $\psb(X)=\Cone(g)$), where $f$ and $g$ denote the left and right minimal approximations of $X$ with respect to $\<S[1]\>$ and $\<S[-1]\>$, respectively.
\end{propdef}
Let $\C$ be a full subcategory of $\D$. According to Proposition 2.20 in \cite{AGH}, if $\h_\D$ and $\h_{\C}$ are the respective hearts of the triangulated categories $\D$ and $\C$ satisfying some suitable conditions, then the Verdier quotient $\D/\C$ inherits a bounded t-structure with heart equivalent to $\h_\D/\h_{\C}$. When the heart on $\C$ arises by restriction, we refer to such a heart as being of \emph{quotient type} and \emph{quotient heart} for short.

We refer to \cite{B} for the concept of Bridgeland stability conditions on a triangulated category $\D$. A \emph{stability condition} $\sigma$ on $\D$ is a pair $(Z,\mathcal P)$, where $Z$ is a group homomorphism from the Grothendieck group $K(\mathcal D)$ to $\CC$, called the \emph{central charge}, and $\mathcal P$ consists of full additive subcategories $\mathcal P(\phi)$ for each $\phi\in \mathbb R$, called the \emph{slicings}, satisfying the following following conditions:
\begin{itemize}
\item For any nonzero object $E\in \mathcal P(\phi)$, we have $Z(E)=m(E)\on{e}^{\on{i}\pi\phi}$.
\item For any $\phi\in \mathbb R$, we have $\mathcal P(\phi +1) = \mathcal P(\phi)[1]$.
\item For $\phi_1>\phi_2$ and $A_i\in \mathcal P(\phi_i)\,(i=1,2)$, we have ${\rm Hom}_{\D}(A_1, A_2) = 0$.
\item For any nonzero object $E\in \D$, there is a \emph{Harder-Narasimhan filtration}, see \cite{B} for more details.
\end{itemize}

Moreover, $\sigma$ satisfies the \emph{support property} (cf. \cite{KS}). Let $\on{Stab}(\D)$ be denoted by the set of such stability conditions on $\D$. The set $\on{Stab}(\D)$ is a complex manifold with local coordinate $Z\in \mathrm{Hom}_{\mathbb Z}(K(\mathcal D),\mathbb C)$. An object $E$ is called \emph{semistable} if $E\in \mathcal P(\phi)$ for some $\phi\in \mathbb R$. Moreover, $E$ is called \emph{stable} if it is a simple object in the abelian category $\mathcal P(\phi)$.

There is an alternative definition of a stability condition. A \emph{stability function} on an abelian category $\mathcal C$ is a group homomorphism $Z : K(\mathcal C)\to \mathbb C$, satisfying that for each nonzero object $E\in\mathcal C$, the complex number $Z(E)$ lies in the semi-closed upper half plane
\[
\mathbb H_{-}=\{r\on{e}^{{\rm i}\pi\phi}\,|\,r>0,\,0<\phi\le1\}\subset\mathbb C.
\]
According to Proposition 5.3 in \cite{B}, a stability condition can also be defined on a triangulated category $\D$ in terms of a pair $(\mathcal H,Z)$, where $\mathcal H$ is a heart of $\D$ and $Z$ is a stability function on $\mathcal H$ that satisfies the Harder-Narasimhan property.

The set $\rm{Stab}(\mathcal D)$ has a natural $\mathbb C$-action: for any $s\in\mathbb{C}$, we have
\[s \cdot (Z, \mathcal P) = (\on{e}^{\on{i}\pi s}Z, \mathcal P_{-\on{Re}(s)}),\]
where $\mathcal P_{-\mathrm{Re}(s)}(m) = \mathcal P(m-\mathrm{Re}(s))$ for any $m\in\mathbb R$. 

A connected component $\on{Stab}^{\bullet}(\D)$ of $\on{Stab}(\D)$ is called \emph{generic-finite} if there exists a connected component $\mathrm{EG}^{\bullet}(\D)$ of the exchange graph $\mathrm{EG}(\D)$, whose vertices are all finite hearts, such that
\[\on{Stab}^{\bullet}(\D)=\mathbb C\cdot\bigcup_{\mathcal H\in\mathrm{EG}^{\bullet}(\mathcal D)}\mathrm{U}(\mathcal H),\]
where $\mathrm{U}(\mathcal H)$ is the subspace in $\mathrm{Stab}(\mathcal D)$ consisting of those stability conditions $\sigma=(\mathcal H, Z)$ whose central charge $Z$ takes values in $\mathbb H=\{z\in\mathbb C\ |\ \mathrm{Im}(z)>0\}$.

\section{Isomorphim between two exchange graphs on the collapsed surface}\label{sec:3}
In the study by \cite{BS}, the authors established a correspondence between the spaces of stability conditions and the moduli spaces of framed quadratic differentials with simple zeros. Building upon this foundation, the research in \cite{BMQS} extended this framework to include quadratic differentials with zeros and higher-order poles, as a boundary strata of the compactification of the spaces of stability conditions. In the present work, we generalize the case where surfaces with simple poles and prove the isomorphism between two exchange graphs. To achieve this, we first state the concept of collapsing subsurfaces.
\subsection{Collapse of subsurfaces}\label{sec:collapse}
Let the weighted DMS $\subsur$ be a subsurface of $\sow$, consisting of connected components $\Sigma_i$. We define the boundary of each component as $\partial\Sigma_i = \cup_j c_{ij}$, where each $c_{ij}$ is a (simple closed) curve is a boundary component of $\Sigma_i$.
\begin{assumption}\label{ass1}
We make the following assumptions on $\subsur$:
\begin{itemize}
\item These boundary $c_{ij}$ either lies within the interior of $\sow$ or corresponds to a connected component of $\partial\sow$. 
\item $\subsur$ can be embedded into $\sow$ in the following way: For the boundary component $c_{ij}$ of $\subsur$ which is a simple closed curve in $\sow$, the number of marked points on it equals to its winding number in $\sow$ plus $2$.
\end{itemize}
\end{assumption}
\begin{definition}\label{def:collapes}
The \emph{collapsed surface} $\colwsur$ of $\surf_{\w}$ with respect to $\Sigma$ is defined by filling each boundary $c_{ij}$ in $\surf_{\w} \setminus \Sigma$ by a disc with one decorated point whose weight is $w_{ij}=\indd(c_{ij})-2$.
\end{definition}
By construction, the collapsed surface $\colwsur$ satisfies the conditions of \Cref{def:wdms}, which implies that it is also a weighted DMS. We can put the three surfaces into a \emph{symbolic short exact sequence}
\begin{equation}\label{eq:ses-surf}
\begin{tikzcd}
    \subsur \ar[r,hookrightarrow] & \sow \ar[r,rightsquigarrow] & \colwsur\,,
\end{tikzcd}
\end{equation}
where the first one is the inclusion and the latter one is the collapse. Moreover, there is a canonical grading of $\colwsur$ induced from that of $\sow$ by keeping the foliation outside $\Sigma$ unchanged, and filling the connecting part of $\Sigma$ and $\sow\backslash\Sigma$ with the horizontal foliation.
\begin{figure}[ht]\centering
\makebox[\textwidth][c]{
\begin{tikzpicture}[scale=.3]
\begin{scope}[shift={(-4,0)}]
\draw[thick, fill = green!10](0,0) circle (5);
\node[font=\tiny] at(5,-3) {$C_{11}$};
\draw[thick,fill = gray!10](0,0) circle (1.5);
\draw[red] (0,0) circle (3);
\node[font=\tiny] at(0,0) {$C_{12}$};
\draw (-3,0)\ww node[red,left,font=\scriptsize]{$1$} (3,0)\ww node[red,right,font=\scriptsize]{$2$};
\draw (0,1.5)\nn (0,-1.5)\nn (0,-5)\nn;
\end{scope}

\begin{scope}[shift={(14,0)},scale=4]
\draw[very thick,fill=cyan!10] (0,0) circle (2);
\draw[thick, fill = green!10](0,-.5) circle (1);
\draw[very thick,fill = gray!10](0,-.5) circle (0.5);
\node[font=\tiny] at(.5,-1.6) {$C_{11}$};
\node[font=\tiny] at(0,-.5) {$C_{12}$};
\draw[black] (0:2)\nn (90:2)\nn (150:2)\nn (210:2)\nn (270:2)\nn;
\draw[black] (0:0)\nn (270:1)\nn;
\draw[red] (0,-.5) circle (0.7);
\draw[red] (-1.4,0.6) -- (0,1.2) -- (1.4,0.6);
\draw[red] (0,1.2)to[out=-50,in=80] (.7,-.5);
\draw[red,font=\scriptsize] (-1.4,0.6)\ww node[left]{$2$} (1.4,0.6)\ww node[right]{$0$} (0,1.2)\ww node[above]{$1$}(-.7,-.5)\ww node[left]{$1$} (.7,-.5)\ww node[right]{$2$};
\end{scope}

\begin{scope}[shift={(32,0)},scale=3]
\draw[very thick,fill=cyan!10] (0,0) circle (2);
\draw[black] (0:2)\nn (90:2)\nn (150:2)\nn (210:2)\nn (270:2)\nn;
\draw[red] (-1.4,0.6) -- (0,1.2) -- (1.4,0.6);
\draw[red] (0,-0.5) -- (0,1.2);
\draw[red,font=\scriptsize] (-1.4,0.6)\ww node[left]{$2$} (1.4,0.6)\ww node[right]{$0$} (0,1.2)\ww node[above]{$1$}(0,-.5)\ww node[below]{$-1$};
\end{scope}

\draw(24,0)node{\huge{$\rightsquigarrow$}};
\draw[right hook-latex,>=stealth](2,0)to(4,0);

\draw (-1,-5) node {$\subsur$}  (18,-8)node {$\sow$} (37,-6) node{$\colwsur$};
\end{tikzpicture}}
\label{fig:collapse}\caption{A collapse}
\end{figure}
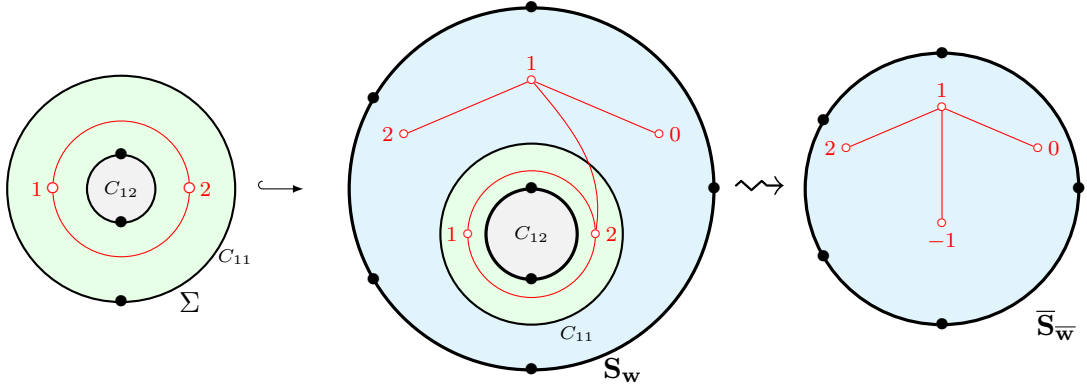
\paragraph{\textbf{Mixed-angulation of $\colwsur$}}
Let $\AS$ and $\uA$ be mixed-angulations of $\sow$ and $\colwsur$ respectively. If $\uA$ is a subset of $\AS$, then we say that $\uA$ is induced from $\AS$ and conversely, $\AS$ is a refinement of $\uA$. The \emph{principal part} $\EGb(\colwsur)$ of~$\EG(\colwsur)$ is defined as the full subgraph of $\EG(\colwsur)$ consisting of the mixed-angulations which are induced from $\EGp(\sow)$.
\begin{proposition}\label{prop:refine_of_flip}
Let $\uA$ be a mixed-angulation in $\EGb(\colwsur)$ and $\uA\xrightarrow{\overline{\gamma}}\uA'$ be a forward/backward flip in $\EGb(\colwsur)$. Then their refinements are related by finite sequences of forward/backward flips in $\EGp(\sow)$:
\begin{equation}\label{eq:liftofflip}
\AS\xrightarrow{\wg}\AS_{\wg}^{\sharp}\to\cdots\to\AS',
\end{equation}
where $\wg$ is the preimage of $\overline{\gamma}$ under the collapse $\sow\rightsquigarrow\colwsur$.
\end{proposition}
\begin{proof}
We only need to show the ``forward'' case because the proof also works if we replace ``forward flip'' by ``backward flip''. The point of the proof is to choose a suitable refinement of $\uA$. We consider the $\uA$-polygons with weighted decorations which arise from collapsing (the connected components of) subsurface and recover the mixed-angulation of $\subsur$, thereby giving a refinement of $\uA$. We particularly consider the $\uA$-polygon in $\colwsur$ containing $\overline{\gamma}$. We notice that there are two angles $a$ and $b$ formed by $\overline{\gamma}$ and its clockwise adjacent edges in the $\uA$-polygon, which are preserved under forward flip, cf. \Cref{fig:3.4.1}.
\begin{figure}
\begin{tikzpicture}[scale=1.1]
\draw[red,thick,dashed] (-1,-.9) .. controls (-2.5,-.8) and (-2.5,0.7) .. (-1.3,1);
\draw (-1.3,0)node{$P$};
  \draw[thick]
    (-1.3,1) .. controls (-0.3,1.2) and (0.3,1.2) .. (1.3,1);

  \draw[thick]
    (-1,-.9) .. controls (-0.3,-0.8) and (0.3,-0.8) .. (1,-.9);

  \draw[thick,cyan]
    (0,1.15) .. controls (-0.2,0.3) and (-0.2,-0.3) .. (0,-.8);

  \draw[thick,blue!50]
    (-.2,1.15) arc[start angle=160, end angle=260, radius=0.2];

  \draw[thick,blue!50]
    (.2,-.85) arc[start angle=-10, end angle=100, radius=0.2];
  \draw[blue!50, thick, dashed] (0,-0.85) -- (-1.3,1);
     \draw (0.1,0.1) node{$\overline{\gamma}$} ;
\draw[thick](0,1.15)\nn (0,-.85)\nn (-1.3,1)\nn;
  \node[below left,blue!50] at (-.1,1.1) {$a$};
  \node[above right,blue!50] at (0,-0.75) {$b$};
  \begin{scope}[shift={(4,0)}]
\draw[thick]
    (-1,-.9) .. controls (-0.3,-0.8) and (0.3,-0.8) .. (1,-.9);
  \draw[thick,cyan]
    (0,-0.85) 
      .. controls (-2,1.7) and (2,1.7) .. (0,-0.85);
  \draw[thick,blue!50]
    (.2,-.85) arc[start angle=-10, end angle=130, radius=0.2];
     \draw (0,1.3) node{$\overline{\gamma}$} (0,.7) node{$P$} ;
\draw[thick] (0,-.85)\nn node[below]{$A$};
\draw[thick] (0,-.65) node[above,blue!50]{$a$};
  \node[above right,blue!50] at (0.15,-0.9) {$b$};
 \draw (0,0)\ww node[red,above,font=\scriptsize]{$-1$}; 
  \end{scope}
\end{tikzpicture}
    \caption{The angles which are preserved when flipping}
    \label{fig:3.4.1}
\end{figure}
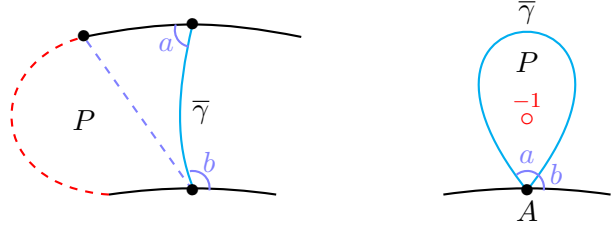
For the polygon arising from collapsing (the connected component of) subsurface, we may, without loss of generality, denote the polygon as $N$-gon $P$ and the connected component as $\subsur_i$. We distinguish the following three cases depending on $N$.

\textbf{Case I:} $N\geq 3$.
Then $\overline{\gamma}$ is not a monogon arc. We claim that there exists a decoration whose weight is at least $1$. If not, by \Cref{prop:combform}, we have that the left hand side of \eqref{eq:combform} is less than $-5$ while the right one is greater than $-4$, which is a contradiction. With this, if $P$ contains the angle $a$ (and $b$ is similar), we connect the two endpoints of its sides that are not the vertex of $a$ by an arc inside $P$, see the dashed arc in the left one of \Cref{fig:3.4.1}. Then we first triangulate $P$ and then remove or add certain arcs so that the remaining arcs cut $P$ into polygons, each of which has a number of sides exactly equal to $w_i+2$ for some weight $w_i$ of the decorations in $\subsur$. Therefore, we get a refinement $\A$ of $\uA$ by adding a mixed-angulation $\subsur$ so that the angle $a$ and $b$ are not effected by new added arcs. We call this kind of refinement of $\uA$ to by \textbf{type I}. Thus we have that $\A\xrightarrow{\wg}\A_{\wg}^{\sharp}$ is a lift of $\uA\xrightarrow{\overline{\gamma}}\uA^\sharp_{\overline{\gamma}}$ in $\sow$.

\textbf{Case II:} $N=2$. If there exists a decoration whose weight is at least $1$, the situation is similar to the above. The only difference lies in the first step, where the arc becomes the dashed arc shown in \Cref{fig:3.4.2}.
\begin{figure}
\begin{tikzpicture}

  \draw[thick]
    (0,1.15) .. controls (0.6,1.1) .. (1.3,1);

  \draw[thick]
    (0,-.85) .. controls (0.5,-0.85) .. (1,-.9);
  \draw[thick,cyan]
    (0,1.15) .. controls (-0.2,0.3) and (-0.2,-0.3) .. (0,-.8);
\draw[thick]
    (0,1.15) .. controls (-2.2,1) and (-2.2,-1) .. (0,-.85);
  \draw[thick,blue!50]
    (-.2,1.15) arc[start angle=170, end angle=260, radius=0.2];

  \draw[thick,blue!50]
    (.2,-.85) arc[start angle=-10, end angle=100, radius=0.2];
\draw[thick,dashed]
    (0,-0.85) 
      .. controls (-1.5,-.8) and (-1,1.7) .. (0,-0.85);
\draw(-.6,-.15)node{$\star$};
     \draw (0.1,0.1) node{$\overline{\gamma}$} ;
\draw[thick](0,1.15)\nn (0,-.85)\nn ;
  \node[below left,blue!50] at (-.1,1.1) {$a$};
  \node[above right,blue!50] at (0,-0.75) {$b$};
\end{tikzpicture}
      \caption{The triangulation when $N=2$}
    \label{fig:3.4.2}
\end{figure}
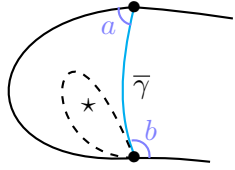
If there is no weighted decoration in $\subsur_i$ with weight bigger than 1, by \Cref{prop:combform}, we have that the left hand side of \eqref{eq:combform} is less than $-4$ while the right one is greater than $-4$, which implies that both sides of the equation must be $-4$. Hence, all decorations on $\subsur_i$ have weight $0$ and we get the following kind of refinement $\A$ of $\uA$ locally dipicted in \Cref{fig:3.4.3}, which is called as \textbf{type II}. 
\begin{figure}
\begin{tikzpicture}[xscale=1.4]
  \draw[thick,cyan] (0,1.2) .. controls (-1.1,0.7) and (-1.1,-0.7) .. (0,-1.2);
\draw[thick,red] (-1.5,0)--(1.5,0);
  \draw[thick] (0,1.2) .. controls (1.1,0.7) and (1.1,-0.7) .. (0,-1.2);
\draw[thick](0,1.2)\nn (0,-1.2)\nn node[below]{$A$};
\draw (-.5,0.5) node{$\overline{\gamma}$} (.5,0.5) node{$\overline{\xi}$};
\draw[red] (-1.2,-0.3) node{$\overline{\eta}$} (1.2,-0.3) node{$\overline{\delta}$};
 \draw (0,0)\ww node[red,above,font=\scriptsize]{$0$};
 \draw[red](-1.5,0)\ww (1.5,0)\ww;
\begin{scope}[shift={(4,0)}]
\draw[thick,cyan] (0,1.2) .. controls (-1.1,0.7) and (-1.1,-0.7) .. (0,-1.2);

  \draw[thick] (0,1.2) .. controls (.4,0.7) and (.4,-0.7) .. (0,-1.2);
\draw[thick] (0,1.2) .. controls (-.4,0.7) and (-.4,-0.7) .. (0,-1.2);
  \draw[thick] (0,1.2) .. controls (1.1,0.7) and (1.1,-0.7) .. (0,-1.2);
\draw[thick](0,1.2)\nn (0,-1.2)\nn node[below]{$A$};
\draw (-1,0) node{$\widetilde{\gamma}$} (1,0) node{$\widetilde{\xi}$};
\draw (-.5,-.3) node[font=\scriptsize]{$\widetilde{\gamma}_1$} (.5,-.3) node[font=\scriptsize]{$\widetilde{\gamma}_m$};
 \draw (-.2,0)\ww node[red,above,font=\scriptsize]{$0$} (.2,0)\ww node[red,above,font=\scriptsize]{$0$} (0,0) node[red,font=\scriptsize]{$\cdots$};
 \draw (-.6,0)\ww node[red,above,font=\scriptsize]{$0$} (0.6,0)\ww node[red,above,font=\scriptsize]{$0$};
\end{scope}
\end{tikzpicture}
    \caption{Refinement of type II}
    \label{fig:3.4.3}
\end{figure}
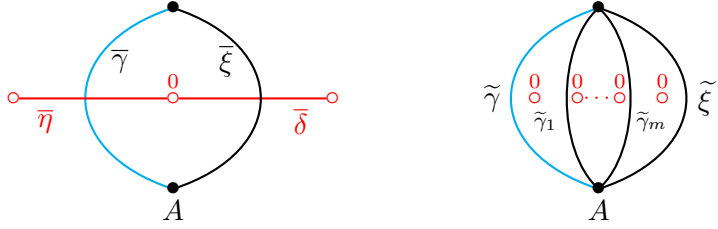
Starting from vertex $A$ and moving clockwise from $\wg$, we label the newly added arcs as $\wg_1,\ldots,\wg_m$. Then we flip $\wg, \wg_1,\ldots,\wg_m$ one by one and get the series in \eqref{eq:liftofflip}.

\textbf{Case III:} $N=1$. As shown in the right one in \Cref{fig:3.4.1}, the outer boundary of $\subsur_i$ contains exactly one marked point, $A$. Below we give the refinement $\uA$, i.e. the mixed-angulation of $\subsur_i$. Since the refinement is performed inside a monogon, it does not cut the angle $b$. We deal with the following two subcases whether there is marked point inside the area bounded by $\wg$.
\begin{enumerate}
\item If there exist a marked point $M$ inside the area bounded by $\wg$, then we draw a monogon $\xi$ with endpoint $M$ such that the region bounded between $\xi$ and $\wg$ is homeomorphic to an annulus, cf. \Cref{fig:3.4.4}. We give a refinement of $\uA$ by using an approach similar to that in the type~I case. We construct two graded open arcs $\gamma_1$ and $\gamma_2$ with endpoints $A$ and $M$ so that $\gamma_1$ and $\gamma_2$ provide a triangulation of the above annular region.
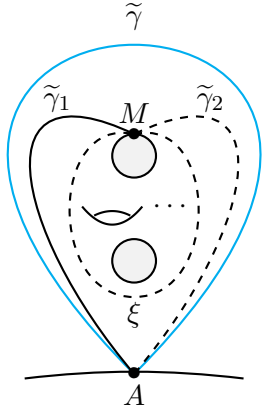
\begin{figure}
\begin{tikzpicture}[xscale=1.2,scale=1.2]
\draw[black,thick,fill=gray!10] (0,1.56) ellipse (0.2 and 0.24);
\draw[black,thick,fill=gray!10] (0,0.4) ellipse (0.2 and 0.24);
\draw[thick]
    (-1,-.9) .. controls (-0.3,-0.8) and (0.3,-0.8) .. (1,-.9);
  \draw[thick,cyan]
    (0,-0.85) 
      .. controls (-4,4) and (4,4) .. (0,-0.85);
     \draw (0,3.1) node{$\wg$} ;
\draw[thick](0,1.8)\nn node[above]{$M$};
\draw (.35,1) node{$\cdots$};
\draw (0,-.2) node{$\xi$};
\draw[thick] (0,-.85)\nn node[below]{$A$};

\begin{scope}[shift={(-0.2,1)},xscale=0.55]
\fill[white]
    plot[domain=-0.5:0.5, smooth] (\x, { -0.2*(1-(2*\x)^2) })  
    -- plot[domain=0.3:-0.3, smooth] (\x, { -0.1 + 0.05*(1-(\x/0.3)^2) }) 
    -- cycle;
\draw[thick]
    plot[domain=-0.5:0.5, smooth] (\x, { -0.2*(1-(2*\x)^2) });  
\draw[thick]
    plot[domain=-0.3:0.3, smooth] (\x, { -0.1 + 0.05*(1-(\x/0.3)^2) }); 
\end{scope}
\draw[thick,dashed] (0,-.85) .. controls +(55:2) and +(30:1.6) .. (0,1.8);
\draw[thick] (0,-.85) .. controls +(125:2) and +(150:1.6) .. (0,1.8);
\draw[thick,dashed] (0,1.8) .. controls +(15:.8) and +(0:.8) .. (0,0) .. controls +(180:.8) and +(165:.8) .. (0,1.8);
\draw (-0.7,2.2) node{$\wg_1$}; 
\draw (0.7,2.2) node{$\wg_2$}; 

\end{tikzpicture}
    \caption{Refinement of type III}
    \label{fig:3.4.4}
\end{figure}
Next, inside $\xi$ we first draw the triangulations and then, according to the weighted decoration of $\Sigma_i$, we add or remove some arcs so as to obtain the mixed-angulation of $\Sigma_i$. Note that, in order to match the weight of $\Sigma_i$, one of $\gamma_1$ or $\gamma_2$ may be removed, but never both, since otherwise the definition of mixed-angulation would not be satisfied. Without loss of generality, we may assume that $\gamma_2$ is probably removed. Finally, by assigning appropriate gradings to the newly added open arcs, we ensure that these graded open arcs satisfy the ``index-zero'' condition required for forming a mixed-angulation. Therefore, we obtain the desired mixed-angulation of $\Sigma_i$ and call this refinement \textbf{type III}. Then we flip $\wg, \wg_1$ (and $\wg_2$) one by one and get the series in \eqref{eq:liftofflip}.
\item If $\subsur_i$ has only one marked point on its outer boundary, then the refinement of $\uA$ can be described locally as the situation where $\wg$ encloses a collection of monogons. We refer to this refinement as \textbf{type IV}. These monogons may have the case where a monogon is nested inside another monogon or involve nontrivial nesting, analogous to linked loops, cf. \Cref{fig:3.4.5}.
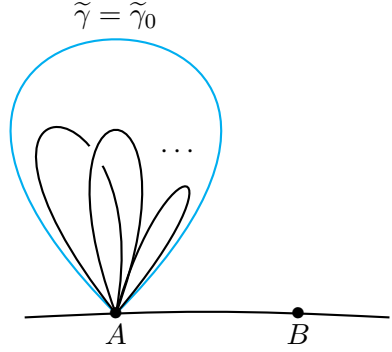
\begin{figure}
\begin{tikzpicture}[xscale=1.2]
\draw[thick]
    (-1,-.9) .. controls (1-0.3,-0.8) and (1.3,-0.8) .. (3,-.9);
  \draw[thick,cyan]
    (0,-0.85) 
      .. controls (-4,4) and (4,4) .. (0,-0.85);
\draw[thick]  (0,-0.85)  .. controls (.55,1.5) and (1.5,1.3) .. (0,-0.85);
\draw[thick]  (0,-0.85)  .. controls (-2.2,2.5) and (.5,2.4) .. (0,-0.85);
 \filldraw[fill=white,draw=none] (-.2,1.25) circle (0.15);
\draw[thick]  (0,-0.85)  .. controls (-1,2.4) and (1,2.4) .. (0,-0.85);
     \draw (0,3.1) node{$\wg=\wg_0$} ;
\draw (.7,1.3) node{$\cdots$};
\draw[thick] (0,-.85)\nn node[below]{$A$};
\draw[thick] (2,-.85)\nn node[below]{$B$};
\end{tikzpicture}
    \caption{Refinement of type IV, an example}
    \label{fig:3.4.5}
\end{figure}
Denote by $\gs$ the new arc obtained from $\wg$ after performing a flip, with endpoints $A$ and $B$. The desired outcome of this sequence of flips is that we arrive at a configuration where there is a large monogon with endpoint $B$, inside of which several monogons are contained (possibly in a complicated nesting pattern).

Before giving the explicit sequence of flips, we first introduce the following notation concerning the local ordering of half-edges at a common vertex. Consider all half-edges of open arcs incident at a fixed vertex, arranged in clockwise order. For half-edges $e_1$ and $e_2$ incident at the same vertex, we say that $e_1<e_2$ if the clockwise angle from $e_1$ to $e_2$ does not intersect the arc $AB$. We write $e_1 \sl e_2$ (resp. $e_1 \sg e_2$) if $e_1<e_2$ ($e_1>e_2$) and there is no half-edge between them in the local ordering. At the vertex $A$, let $e_A$ (resp. $e_B$) denote the half-edge corresponding to the arc $AB$; then by convention we require $e<e_A$ ($e_B<e$) for all other half-edges $e$ incident at $A$ (resp. $B$). For each newly added open arc $\widetilde{\xi}$ with endpoints $A$ and $B$, we denote by $\widetilde{\xi}_L$ (resp. $\widetilde{\xi}_R$) the half-edge of $\widetilde{\xi}$ incident at $A$ (resp. $B$). If both endpoints of $\widetilde{\xi}$ coincide, then the two corresponding half-edges are still denoted by $\widetilde{\xi}_L$ and $\widetilde{\xi}_R$ where they satisfy $\widetilde{\xi}_L \sl \widetilde{\xi}_R$.

We first claim that if $\widetilde{\xi}_L \sl \widetilde{\xi}_R$, then $\widetilde{\xi}$ is a monogon arc. Hence, in this case, the operation of flipping $\widetilde{\xi}$ should be done according to the second one in \Cref{def:wgflip}. We show the claim as follows. Each $\Sigma_i$ of genus $g$ can be regarded as the surface obtained by gluing together the edges labeled $a_i$ in the $4g$-gon shown in \Cref{fig:3.4.6}.
\begin{figure}
\begin{tikzpicture}[scale=1.2]
\def\n{16}   
  \def\r{2}    

  \draw[thick] (0:\r) \foreach \x in {1,...,\n} { -- (360/\n*\x:\r) } -- cycle;
\draw[thick, fill=gray!10] 
    (0,-2) 
    .. controls (-.8,0) and (.8,0) .. (0,-2);
\draw[thick,cyan ] 
    (0,-2) 
    .. controls (-1.5,1) and (1.5,1) .. (0,-2); 
  \foreach \x in {0,...,\numexpr\n-1} {
    \pgfmathsetmacro{\xa}{360/\n*\x}
    \pgfmathsetmacro{\xb}{360/\n*(\x+1)}

    \path (\xa:\r) -- (\xb:\r) coordinate[midway](M);

    \pgfmathsetmacro{\dx}{cos(\xb)-cos(\xa)}
    \pgfmathsetmacro{\dy}{sin(\xb)-sin(\xa)}
    \pgfmathsetmacro{\ang}{atan2(\dy,\dx)} 
    \pgfmathtruncatemacro{\dir}{mod(\x,4)}
    \ifnum\dir<2
      \draw[->] (\xa:\r) -- (M);
    \else
      \draw[->] (\xb:\r) -- (M);
    \fi
    \foreach \x in {0,...,\numexpr\n-1} {
    \draw (360/\n*\x:\r) \nn;
}
  }
\fill[white, draw=none] (1,1.5) circle (1.2);
\draw (1,1.5)node{$\cdots$};
\draw (-1,-.8)node{$\widetilde{\xi}$} (-.3,-1.7)node[font=\scriptsize]{$\widetilde{\xi}_L$} (-1.6,.3)node[font=\scriptsize]{$\widetilde{\xi}_R$};
\draw (0,-2.3)node{$A$};
\draw[red, thick] 
    (0,-2) .. controls (-1,0) .. (157.5:2);
\draw(-101.25:2.2)node{$a_1$} (236.25:2.2)node{$a_2$} (213.75:2.2)node{$a_1$} (191.25:2.2)node{$a_2$} (168.75:2.2)node{$a_3$} (146.25:2.2)node{$a_4$} (123.75:2.2)node{$a_3$} (101.25:2.2)node{$a_4$}  (-11.25:2.5)node{$a_{2g-1}$} (-33.75:2.3)node{$a_{2g}$} (-56.25:2.4)node{$a_{2g-1}$} (-78.75:2.2)node{$a_{2g}$};
\end{tikzpicture}
\caption{An arc on the $\subsur$ can be represented as an arc on the corresponding $4g$-gon}
    \label{fig:3.4.6}
\end{figure}
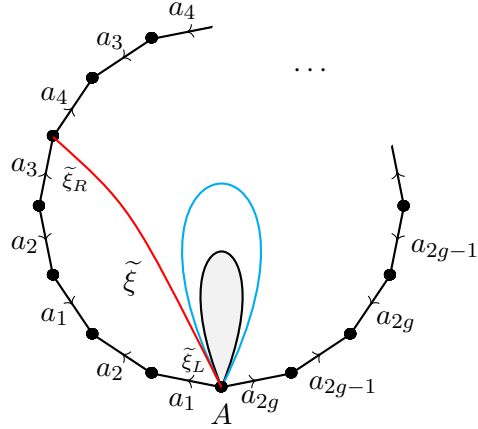
Without loss of generality, we place the boundary at the vertex $A$ and represent $\widetilde{\xi}$ as an arc on the $4g$-gon. Then, traversing once counterclockwise around $A$, if there are no other open arcs lying strictly between $\widetilde{\xi}_L$ and $\widetilde{\xi}_R$, from \Cref{fig:3.4.6}, we see that a mixed-angulation cannot be formed in this situation.

We now describe the sequence of flips iteratively, where we finally get what we expect. First, set $\wg_0=\wg$ and consider $\wg_i$. If $\wg_i$ is not a monogon arc, then we take the graded open arc $\wg_{i+1}$ such that $\wg_{i,L}\sl \wg_{i+1,L}<\wg_{i,R}$. We next flip $\wg_i$, which becomes $\gs_i$. Thus, we have  
\[
\wg_{i+1,R}\sl\gs_{i,L},\quad \text{for any } i.
\]  
Replacing $i$ by $i-1$, we then deduce  
\[
\gs_{i-1,R}\sl\gs_{i,R},\quad \text{for any } i,
\]  
where we set $\gs_{-1,L}=e_A$ and $\gs_{-1,R}=e_B$.
 If $\wg_i$ is a monogon arc, note that in this case we have $i>0$. Then we take $\wg_{i+1}=\gs_{i-1}$, and we naturally have $\wg_{i,L}\sl \wg_{i,R}\sl \wg_{i+1,L}.$  
We flip $\wg_i$, which becomes $\gs_i$, so that  
\[
\gs_{i-1,R}=\wg_{i+1,R}\sl \gs_{i,L}\sl \gs_{i,R}.
\]  
Next we flip $\wg_{i+1}$ to obtain $\gs_{i+1}$ and proceed in this way.  
\begin{figure}\centering
\makebox[\textwidth][c]{
\begin{tikzpicture}[xscale=2]
\begin{scope}[shift={(-5,0)}]
\draw[thick]
    (-.5,-.9) .. controls (1-0.3,-0.8) and (1.3,-0.8) .. (2.5,-.9);
\draw[thick]  (0,-0.85)  .. controls (-1,2.4) and (1,2.4) .. (0,-0.85);
 \filldraw[fill=white,draw=none] (.3,.6) circle (0.15);
 \filldraw[fill=white,draw=none] (-.3,.9) circle (0.15);
\draw[thick]  (0,-0.85)  .. controls (0,1.8) and (1.5,1.2) .. (2,-.85);
\draw[thick]  (0,-0.85)  .. controls (-2.2,1.5) and (0,2.7) .. (0,-0.85);
     \draw (0,2) node{$\wg_{i+1}$} (-.8,1.8) node{$\wg_{i}$} ;
\draw (.9,1.3) node{$\gs_{i-1}$};
\draw[thick] (0,-.85)\nn node[below]{$A$};
\draw[thick] (2,-.85)\nn node[below]{$B$};
\draw(2.2,1)edge[thick,->,>=stealth](2.7,1);
\draw(2.45,1.3)node{$\wg_i$};
\begin{scope}[shift={(3.5,0)}]
\draw[thick]
    (-.5,-.9) .. controls (1-0.3,-0.8) and (1.3,-0.8) .. (2.5,-.9);
\draw[thick]  (0,-0.85)  .. controls (-1,2.4) and (1,2.4) .. (0,-0.85);
 \filldraw[fill=white,draw=none] (.3,.6) circle (0.15);
\draw[thick]  (0,-0.85)  .. controls (0,1.8) and (1.5,1.2) .. (2,-.85);
 \filldraw[fill=white,draw=none] (1.1,.8) circle (0.1);
\draw[thick,red]  (0,-0.85)  .. controls (1,1.8) and (2.5,1.2) .. (2,-.85);
     \draw (0,2) node{$\wg_{i+1}$};
     \draw[red] (1.7,1.3) node{$\gs_{i}$};
\draw (.9,1.3) node{$\gs_{i-1}$};
\draw[thick] (0,-.85)\nn node[below]{$A$};
\draw[thick] (2,-.85)\nn node[below]{$B$};
\end{scope}
\end{scope}
\end{tikzpicture}}
    \caption{A flip at $\wg_i$ in type IV, when $\wg_i$ is not a monogon arc}
    \label{fig:3.4.7}
\end{figure}
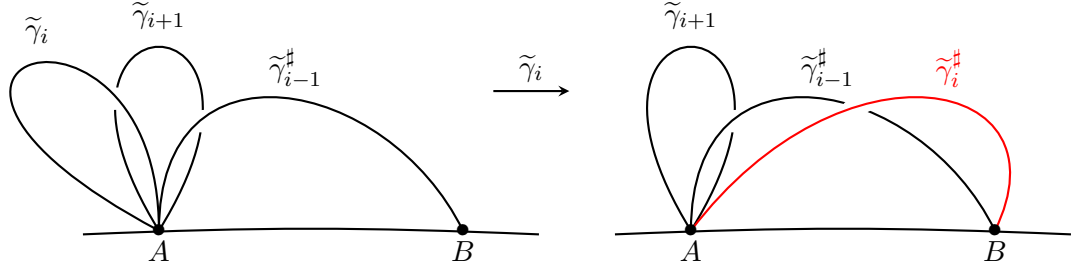
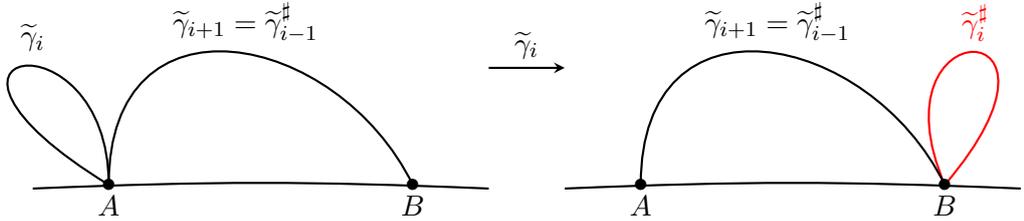
\begin{figure}\centering
\makebox[\textwidth][c]{
\begin{tikzpicture}[xscale=2]
\begin{scope}[shift={(-5,0)}]
\draw[thick]
    (-.5,-.9) .. controls (1-0.3,-0.8) and (1.3,-0.8) .. (2.5,-.9);
\draw[thick]  (0,-0.85)  .. controls (0,1.8) and (1.5,1.2) .. (2,-.85);
\draw[thick]  (0,-0.85)  .. controls (-1.5,1) and (0,1.5) .. (0,-0.85);
     \draw (-.5,1.1) node{$\wg_{i}$} ;
\draw (.9,1.3) node{$\wg_{i+1}=\gs_{i-1}$};
\draw[thick] (0,-.85)\nn node[below]{$A$};
\draw[thick] (2,-.85)\nn node[below]{$B$};
\draw(2.5,.7)edge[thick,->,>=stealth](3,.7);
\draw(2.75,1)node{$\wg_i$};
\begin{scope}[shift={(3.5,0)}]
\draw[thick]
    (-.5,-.9) .. controls (1-0.3,-0.8) and (1.3,-0.8) .. (2.5,-.9);
\draw[thick,red]  (2,-0.85)  .. controls (1.5,1.5) and (3,1.5) .. (2,-0.85);
\draw[thick]  (0,-0.85)  .. controls (0,1.8) and (1.5,1.2) .. (2,-.85);
     \draw[red] (2.2,1.3) node{$\gs_{i}$};
\draw (.9,1.3) node{$\wg_{i+1}=\gs_{i-1}$};
\draw[thick] (0,-.85)\nn node[below]{$A$};
\draw[thick] (2,-.85)\nn node[below]{$B$};
\end{scope}
\end{scope}
\end{tikzpicture}}
    \caption{A flip at $\wg_i$ in type IV, when $\wg_i$ is a monogon arc}
    \label{fig:3.4.8}
\end{figure}
This procedure stops when $\wg_{i+1,L}=e_A$. To see that the process involves only finitely many steps, since if $\wg_i$ has coinciding endpoints, then flipping $\wg_i$ reduces the number of half-edges at $A$ by one or two. If $\wg_i$ has distinct endpoints, then flipping $\wg_{i+1}$ also reduces the number of half-edges at $A$ by one or two. Therefore, the number of half-edges at $A$ decreases and eventually becomes zero. Thus the process stops in finitely many steps. At the final step $N$, we have 
\[e_B \sl \gs_{N,L}
\quad \text{and} \quad
e < \gs_{N,R}\] for every half-edge $e$ at $B$. This implies that $\gs_N$ encloses a collection of monogons. We draw an example of the sequence of forward flips of type IV in \Cref{fig:3} in \Cref{sec:B}. Therefore, the sequence in \eqref{eq:liftofflip} is obtained.
\end{enumerate}
\end{proof}
As a consequence, the graph $\EGb(\colwsur)$ is $(m,m)$-regular, where $m$ is the number of arcs in any mixed-angulation of $\colwsur$. Thus we obtain the following corollary.
\begin{corollary}
the principal part $\EGb(\colwsur)$ is a union of connected components of $\EG(\colwsur)$.
\end{corollary}
\subsection{Categories associated with given surfaces}
Using the correspondence between quadratic differentials and stability conditions in \cite{HKK}, a marked surface may correspond to several categories, but their spaces of stability conditions are isomorphic since they all share the same moduli space of quadratic differentials. We now make the following assumptions on the categories associated to surfaces, and then establish the correspondence between the space of stability conditions and the moduli space of quadratic differentials.

Let $\sow$ be a weighted DMS and $\D(\sow)$ be the corresponding triangulated category satisfying the following assumptions.
\begin{assumption}\label{prop:CHQ}
We make the following assumptions on $\D(\sow)$.
\begin{enumerate}
\item There is an injection from the set of graded closed arc $\wCA(\sow)$ to the set of objects in $\D(\sow)$. We denote the object $X_{\we}$ corresponding to the graded closed arc $\we$.
\item Given two closed arcs $\widetilde{e}$ and $\widetilde{h}$ in $\wCA(\sow)$ intersecting at a decoration $z$ with intersection index $\ind_z(\widetilde{e},\widetilde{h})=d$, there is a unique morphism $X_{\widetilde{e}}\to X_{\widetilde{h}}[d]$ whose cone corresponds to $X_{\we}$, where $\we$ is the graded arc $\widetilde{e}\wedge_z\widetilde{h}$. In other words, we have an equivalence  
$$\Cone(X_{\widetilde{e}}\to X_{\widetilde{h}}[d])\simeq X_{\we}.$$
\item Let $\pv$ be the subcategory of $\D(\sow)$ generated by those $X_{\we}$, where $\we$ is a graded closed arc in $\dAS$ with $\dAS$ the dual set of some mixed-angulation $\AS.$ Moreover, given a mixed-angulation $\AS$ with dual set $\dAS=\{\we_1,\we_2,\ldots,\we_n\}$, there is a finite heart $\h_{\AS}$ in $\pv$ generated by simples $X_{\we_1}, X_{\we_2},\dots, X_{\we_n}$.
\end{enumerate}
\end{assumption}
We consider the canonical heart $\h_0$ of $\pv$ corresponding to the initial mixed-angulation $\AS_0$. Let $\EGp(\pv)$ be the connected component of $\EG(\pv)$ consisting of all hearts obtained from $\h_0$ via repeated simple tilting.
\begin{corollary}\label{cor:ass}
With assumptions above, there is an isomorphism between exchange graphs:
\begin{equation} \label{EGpvdEG}
\EGp(\pv)\,\cong\,\EGp(\sow)\,,
\end{equation}
that is, the forward (or backward) flip of $\AS$ along $\wg$ corresponds to the simple forward (or backward) tilting of $\h_{\AS}$ at $X_{\we}$, where $\AS\in\EGp(\sow)$ and $\we$ is a dual arc of $\wg\in\AS$ with corresponding simple object $X_{\we}$ in the heart $\h_{\AS}$.
\end{corollary}
\subsection{The quotient categories associated to collapsed surfaces}\label{sec:QuotCat}
The collapse operation  
\begin{tikzcd}  
\cols:\sow \ar[r,rightsquigarrow] & \colwsur  
\end{tikzcd}  
induces a surjective map from the set $\uCA(\sow)$ to the set of graded closed arcs on $\colwsur$. Any arc $\we\in\uCA(\sow)$ with endpoints at decorations in $\subsur$ and homotopic to a closed arc in $\subsur$ vanishes under collapse. We obtain the set of closed arc in $\colwsur$ induced from collapsing as:
\[\uCA(\colwsur):=\cols(\uCA(\sow)).\]

Let $\D(\subsur)$ be the thick subcategory of $\D(\sow)$ such that it satisfies the three conditions in \Cref{prop:CHQ}. We define $\D(\colwsur)$ to be the Verdier quotient of $\D(\sow)$ with respect to $\D(\subsur)$, which fits into the following short exact sequence of triangulated categories: 
\begin{equation}\label{eq:ses.cats}
\begin{tikzcd}
    0 \ar[r] & \D(\subsur) \ar[r] & \Dw \ar[r,"\pi"] &\D(\colwsur) \ar[r] & 0.
\end{tikzcd}
\end{equation}
\begin{proposition}\label{prop:objcor}
There is an injection from the set $\uCA(\colwsur)$ of graded closed arcs in $\colwsur$ to the set of objects in $\D(\colwsur)$, that is, for each closed arc $\ue=\cols(\we)\in\uCA(\colwsur)$, the corresponding arc object $X_{\ue}$ is uniquely given by $\pi(X_{\we})$, where $X_{\we}\in\D(\sow)$ is the object corresponding to $\we$.
\end{proposition}  
\begin{proof}
We only need to show that the map $\ue=\cols(\we)\mapsto X_{\we}$ is well-defined, i.e. for any closed arcs $\we,\we'\in\uCA(\sow)$ satisfying $\cols(\we)=\cols(\we')=\ue$, we have $\pi(X_{\we})=\pi(X_{\we'})$, where $\pi:\D(\sow)\to\uD(\colwsur)$ is the projection in \eqref{eq:ses.cats}. By the assumption $\cols(\we)=\cols(\we')=\ue$, the domain bounded by $\we$, $\we'$ and $\partial\subsur$ in $\sow\backslash\subsur$ is contractible. 

\begin{figure}
\begin{tikzpicture}
  \fill[green!10] (0,0) circle[radius=2cm];
\draw[thick,fill = gray!10](-1,0) circle (.2);
  \draw[thick] (225:2cm) arc[start angle=225, end angle=135, radius=2cm];

  \draw[thick,dashed] (225:2cm) arc[start angle=225, end angle=495, radius=2cm];
\fill[white]
    plot[domain=-0.5:0.5, smooth] (\x, { -0.2*(1-(2*\x)^2) })  
    -- plot[domain=0.3:-0.3, smooth] (\x, { -0.1 + 0.05*(1-(\x/0.3)^2) }) 
    -- cycle;

  \draw[thick]
    plot[domain=-0.5:0.5, smooth] (\x, { -0.2*(1-(2*\x)^2) });  

  \draw[thick]
    plot[domain=-0.3:0.3, smooth] (\x, { -0.1 + 0.05*(1-(\x/0.3)^2) }); 
  \coordinate (A) at (0.9,1);
  \coordinate (B) at (150:2);
  \coordinate (C) at (0.9,-1);
  \coordinate (D) at (210:2);

  \draw[thick] (A) .. controls (-.5,1.2) .. (B);

  \draw[thick] (C) .. controls (-.5,-1.2) .. (D);

 \draw[thick,red] (A) .. controls (-2.5,0.7) and (-2.5,-0.7) .. (C);
 \draw (.9,1)\ww node[red,right,font=\scriptsize]{$z$} (.9,-1)\ww node[red,right,font=\scriptsize]{$z'$};
\draw[thick] (-.5,1.4) node{$\we$} (-.5,-1.4) node{$\we'$};
\draw[red] (0,.6) node{$\wg$};
\begin{scope}[shift={(6,0)}]
\fill[green!10] (0,0) circle[radius=2cm];
\draw[thick,fill = gray!10](-1,0) circle (.2);
  \draw[thick] (225:2cm) arc[start angle=225, end angle=135, radius=2cm];

  \draw[thick,dashed] (225:2cm) arc[start angle=225, end angle=495, radius=2cm];
\fill[white]
    plot[domain=-0.5:0.5, smooth] (\x, { -0.2*(1-(2*\x)^2) })  
    -- plot[domain=0.3:-0.3, smooth] (\x, { -0.1 + 0.05*(1-(\x/0.3)^2) }) 
    -- cycle;

  \draw[thick]
    plot[domain=-0.5:0.5, smooth] (\x, { -0.2*(1-(2*\x)^2) });  

  \draw[thick]
    plot[domain=-0.3:0.3, smooth] (\x, { -0.1 + 0.05*(1-(\x/0.3)^2) }); 
  \coordinate (A) at (1.5,0);
  \coordinate (B) at (150:2);
  \coordinate (D) at (210:2);

  \draw[thick] (A) .. controls (.5,.9) .. (B);

  \draw[thick] (A) .. controls (.5,-.9) .. (D);

 \draw[thick,red] (A) .. controls (-2.5,2.1) and (-2.5,-2.1) .. (A);
 \draw (1.5,0)\ww node[red,right,font=\scriptsize]{$z$} ;
\draw[thick] (-.5,1.2) node{$\we$} (-.5,-1.2) node{$\we'$};
\draw[red] (0,.3) node{$\wg$};
\end{scope}
\end{tikzpicture}
    \caption{Locally in $\subsur_i$, where $\we$ and $\we'$ does not intersect in $\sow\backslash\Delta$}
    \label{fig:3.8.1}
\end{figure}
Consider the case that $\we$ and $\we'$ have no intersection in $\sow\backslash\Delta$. Without loss of generality, we assume an endpoint $z$ of $\we$ and an endpoint $z'$ of $\we'$ are contained in some connected component $\subsur_i$ ($z$ and $z'$ may coincide). For the other two endpoints of $\we$ and $\we'$, we assume that they are in $\sow\backslash\subsur$ and coincide (or they are in some connected component $\subsur_j$ and we can deal with this case similarly). Take an arc $\wg\in\uCA(\subsur)$ connecting $z$ and $z'$ such that the triangle formed by $\wg$, $\we$ and $\we'$ is contractible, cf. \Cref{fig:3.8.1}. If the angle from $\wg$ to $\we$ is in the contractible triangle above, we take suitable grading such that $\ind_z(\wg,\we)=0$. By (2) of \Cref{prop:CHQ}, we have a distinguished triangle
\begin{equation}
X_{\wg}\to X_{\we}\to X_{\wg\wedge\we}\to X_{\wg}[1]
\end{equation}
in $\D(\sow)$. Applying $\pi$ to this triangle, we have
\begin{equation}
\pi(X_{\we})=\pi(X_{\wg\wedge\we}).
\end{equation}
Moreover, $\wg\wedge\we$ is homotopic to $\we'$, and thus we have 
\begin{equation}
\pi(X_{\we'})=\pi(X_{\wg\wedge\we})=\pi(X_{\we}).
\end{equation}
If the angle from $\wg$ to $\we$ is not in the contractible triangle above, then the angle from $\wg$ to $\we'$ is in the contractible triangle, and we can prove by changing $\we$ and $\we'$.
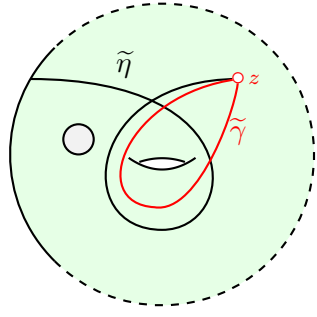
\begin{figure}
\begin{tikzpicture}
  \fill[green!10] (0,0) circle[radius=2cm];
\draw[thick,fill = gray!10](-1.1,0.2) circle (.2);
  \draw[thick] (225:2cm) arc[start angle=225, end angle=135, radius=2cm];

  \draw[thick,dashed] (225:2cm) arc[start angle=225, end angle=495, radius=2cm];
\fill[white]
    plot[domain=-0.5:0.5, smooth] (\x, { -0.2*(1-(2*\x)^2) })  
    -- plot[domain=0.3:-0.3, smooth] (\x, { -0.1 + 0.05*(1-(\x/0.3)^2) }) 
    -- cycle;

  \draw[thick]
    plot[domain=-0.44:0.44, smooth] (\x, { -0.2*(1-(2*\x)^2) });  

  \draw[thick]
    plot[domain=-0.3:0.3, smooth] (\x, { -0.1 + 0.05*(1-(\x/0.3)^2) }); 
\coordinate (P0) at (150:2);
  \coordinate (P1) at (0,-1);
  \coordinate (P2) at (1,1);
  \coordinate (Self) at (0,1); 
  \draw[thick]
    (P0)
    .. controls (1.2,1) and (1,-1) .. (P1)
    .. controls (-1.2,-1) and (-1,1) .. (P2);
  \draw[thick,red]
    (P2)
    .. controls (-.5,.8) and (-1,-.65) .. (-.1,-.7)
    .. controls (.5,-.8) and (1,.7) .. (P2);

 \draw (1,1)\ww node[red,right,font=\scriptsize]{$z$};
\draw[thick] (-.5,1.2) node{$\we$};
\draw[red] (1,.3) node{$\wg$};
\end{tikzpicture}
    \caption{Smoothing out the self-intersections of $\we$ in $\subsur_i$}
    \label{fig:3.8.2}
\end{figure}
Otherwise, $\we$ and $\we'$ intersect in $\surf^{\circ}\backslash\Delta$. 

We may assume that both $\we$ and $\we'$ do not has self-intersections in $\subsur$. If not, we can take some $\wg$ in $\subsur$ such that $\wg\wedge\we$ does not have self-intersections in $\subsur$, cf. \Cref{fig:3.8.2}. Similar as above, we have that $\we$ and $\wg\wedge\we$ induce the isomorphic objects in the quotient category. 

We set that the intersections between $\we$ and $\we'$ in a connected component $\subsur_i$ are denoted by $p_1,\cdots,p_k$ (the order begins from $z$ or $z'$ and is along $\we$ or $\we'$). Denote $\we$ by $\we_0$ and $\we'$ by $\we'_0$. For any $i\geq 1$, we take an arc $\wg_i\in\uCA(\subsur)$ connecting $z$ and $z'$ such that the triangle formed by $\wg_i$, the arc segment $\wideparen{zp_i}$ and $\wideparen{z'p_i}$ is contractible, cf. \Cref{fig:3.8.3}. If the angle from $\wg_i$ to $\we_{i-1}$ is in the contractible triangle above, we take suitable grading such that $\ind_z(\wg_i,\we_{i-1})=0$, and let $\we_i=\wg_i\wedge\we_{i-1}$ and $\we'_i=\we'_{i-1}$. If the angle from $\wg_i$ to $\we_{i-1}$ is not in the contractible triangle, we can change $\we_{i-1}$ and $\we'_{i-1}$. We get
\begin{equation}
\pi(X_{\we})=\pi(X_{\we_1})=\cdots=\pi(X_{\we_k})
\end{equation}
and
\begin{equation}
\pi(X_{\we'})=\pi(X_{\we'_1})=\cdots=\pi(X_{\we'_k}),
\end{equation}
where $\we_k$ and $\we'_k$ have no intersection in $\surf^{\circ}\backslash\Delta$. Then we can prove the well-definedness by the consequence of the first case.
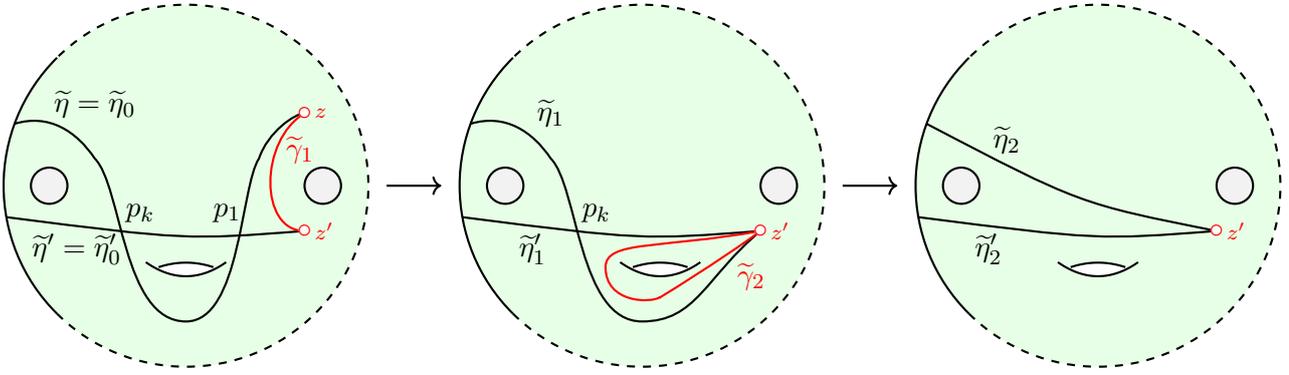
\begin{figure}\centering
\makebox[\textwidth][c]{
\begin{tikzpicture}[scale=1.2]
 \fill[green!10] (0,0) circle[radius=2cm];
\draw[thick,fill = gray!10](-1.5,0) circle (.2);
\draw[thick,fill = gray!10](1.5,0) circle (.2);
  \draw[thick] (225:2cm) arc[start angle=225, end angle=135, radius=2cm];

  \draw[thick,dashed] (225:2cm) arc[start angle=225, end angle=495, radius=2cm];

  \fill[white]
    plot[domain=-0.5:0.5, smooth] (\x, { -0.2*(1-(2*\x)^2) -.8 })  
    -- plot[domain=0.3:-0.3, smooth] (\x, { -0.1 + 0.05*(1-(\x/0.3)^2) -.8 }) 
    -- cycle;

  \draw[thick]
    plot[domain=-0.44:0.44, smooth] (\x, { -0.2*(1-(2*\x)^2) -.8 });  

  \draw[thick]
    plot[domain=-0.3:0.3, smooth] (\x, { -0.1 + 0.05*(1-(\x/0.3)^2) -.8 }); 
  \coordinate (A) at (1.3,.8);
  \coordinate (B) at (160:2);
  \coordinate (C) at (1.3,-.5);
  \coordinate (D) at (190:2);

  \draw[thick]  (B)
    .. controls (-1.5,.8) and (-1.2,.6) .. (-1,.3)
    .. controls (-0.7,0) and (-0.7,-1.5) .. (0,-1.5)
    .. controls (0.6,-1.5) and (0.6,0) .. (.8,.3)
    .. controls (.9,.6) and (1.2,.8) .. (A);

  \draw[thick] (C) .. controls (0,-.6) .. (D);

 \draw[thick,red] (A) .. controls (.8,0.5) and (.8,-0.5) .. (C);
 \draw (A)\ww node[red,right,font=\scriptsize]{$z$} (C)\ww node[red,right,font=\scriptsize]{$z'$};
\draw[thick] (-1,.9) node{$\we=\we_0$} (-1.2,-.7) node{$\we'=\we'_0$};
\draw[red] (1.25,.4) node{$\wg_1$};
\draw[thick] (.45,-.3)node{$p_1$}(-.5,-.3)node{$p_k$};
\begin{scope}[shift={(5,0)}]
\fill[green!10] (0,0) circle[radius=2cm];
\draw[thick,fill = gray!10](-1.5,0) circle (.2);
\draw[thick,fill = gray!10](1.5,0) circle (.2);
  \draw[thick] (225:2cm) arc[start angle=225, end angle=135, radius=2cm];

  \draw[thick,dashed] (225:2cm) arc[start angle=225, end angle=495, radius=2cm];

  \fill[white]
    plot[domain=-0.5:0.5, smooth] (\x+.2, { -0.2*(1-(2*\x)^2) -.8 })  
    -- plot[domain=0.3:-0.3, smooth] (\x+.2, { -0.1 + 0.05*(1-(\x/0.3)^2) -.8 }) 
    -- cycle;

  \draw[thick]
    plot[domain=-0.44:0.44, smooth] (\x+.2, { -0.2*(1-(2*\x)^2) -.8 });  

  \draw[thick]
    plot[domain=-0.3:0.3, smooth] (\x+.2, { -0.1 + 0.05*(1-(\x/0.3)^2) -.8 }); 
  \coordinate (A) at (1.3,.8);
  \coordinate (B) at (160:2);
  \coordinate (C) at (1.3,-.5);
  \coordinate (D) at (190:2);

  \draw[thick]  (B)
    .. controls (-1.5,.8) and (-1.2,.6) .. (-1,.3)
    .. controls (-0.7,0) and (-0.7,-1.5) .. (0,-1.5)
    .. controls (0.6,-1.5)  and (0.7,-1) .. (C);

  \draw[thick] (C) .. controls (0,-.6) .. (D);


\draw[thick,red] (C)    .. controls (0.2,-0.7) and (-0.4,-0.6) ..
  (-0.4,-0.9)
    .. controls (-0.4,-1.3) and (0.1,-1.3) ..
  (.2,-1.23)
    .. controls (0.1,-1.3) and (0.9,-.8) ..(C);
 \draw (C)\ww node[red,right,font=\scriptsize]{$z'$};
\draw[thick] (-1,.8) node{$\we_1$} (-1.2,-.7) node{$\we'_1$};
\draw[red] (1.2,-1) node{$\wg_2$};
\draw[thick] (-.5,-.3)node{$p_k$};
\end{scope}
\begin{scope}[shift={(10,0)}]
\fill[green!10] (0,0) circle[radius=2cm];
\draw[thick,fill = gray!10](-1.5,0) circle (.2);
\draw[thick,fill = gray!10](1.5,0) circle (.2);
  \draw[thick] (225:2cm) arc[start angle=225, end angle=135, radius=2cm];

  \draw[thick,dashed] (225:2cm) arc[start angle=225, end angle=495, radius=2cm];

  \fill[white]
    plot[domain=-0.5:0.5, smooth] (\x, { -0.2*(1-(2*\x)^2) -.8 })  
    -- plot[domain=0.3:-0.3, smooth] (\x, { -0.1 + 0.05*(1-(\x/0.3)^2) -.8 }) 
    -- cycle;

  \draw[thick]
    plot[domain=-0.44:0.44, smooth] (\x, { -0.2*(1-(2*\x)^2) -.8 });  

  \draw[thick]
    plot[domain=-0.3:0.3, smooth] (\x, { -0.1 + 0.05*(1-(\x/0.3)^2) -.8 }); 
  \coordinate (A) at (1.3,.8);
  \coordinate (B) at (160:2);
  \coordinate (C) at (1.3,-.5);
  \coordinate (D) at (190:2);

  \draw[thick]  (B)
    .. controls (-.2,-.2) .. (C);

  \draw[thick] (C) .. controls (0,-.6) .. (D);

 \draw  (C)\ww node[red,right,font=\scriptsize]{$z'$};
\draw[thick] (-1,.5) node{$\we_2$} (-1.2,-.7) node{$\we'_2$};

\end{scope}
\draw[thick,->] (2.2,0) -- (2.8,0);
\draw[thick,->] (7.2,0) -- (7.8,0);
\end{tikzpicture}}
    \caption{The process of solving the intersections in $\subsur_i$ between $\we$ and $\we'$}
    \label{fig:3.8.3}
\end{figure}
\end{proof}
\subsection{The exchange graphs of hearts of quotient type and its principal part} \label{sec:EGquotheart}
For $\AS$ a mixed-angulation of $\sow$, we denote by $\AS|_{\subsur}$ the mixed-angulation of $\subsur$ induced by $\AS$ by regarding it as a subset of $\AS$ up to homotopy. We denote $\dAS|_{\subsur}$ by the dual set of $\AS|_{\subsur}$. By forgetting the open arcs in $\AS|_{\subsur}$, we obtain a mixed-angulation of $\colwsur$, which is denoted by $\uA$. Similarly, we denote its dual set as $\overline{\dAS}$.

Let $\pvs$ be the subcategory of $\D(\subsur)$ generated by those $X_{\we}$, where $\we$ is a graded closed arc in $\dAS|_{\subsur}$. Then \eqref{eq:ses.cats} induces the following Verdier quotient of triangulated categories:
\begin{equation}\label{eq:ses.pvds}
\begin{tikzcd}
    0 \ar[r] & \pvs \ar[r] & \pv \ar[r,"\pi"] &\pvq \ar[r] & 0,
\end{tikzcd}
\end{equation}
where $\pvq$ is the subcategory of $\D(\colwsur)$ generated by $\{X_{\we}\}$ with $\we$ being the graded closed arc in $\overline{\dAS}$.
We consider a special class of bounded t-structure hearts in the quotient category $\pvq$, which we refer to as “of quotient type”. More precisely, our attention is on those hearts of $\pvq$ that arise as the quotient of a heart in $\pv$ by a heart in $\pvs$. For more details, we refer to Proposition 2.20 in \cite{AGH} and Definition 6.3 in \cite{BMQS}. Let $\EG(\pvq)$ be the exchange graph of the quotient category. We call its full subgraph $\EGb(\pvq)$ the \emph{principal part} with vertices which are hearts of quotient type.
\begin{proposition}\label{prop:vertex}
Let $\uA$ be a mixed angulation of $\colwsur$ whose dual set of graded closed arcs is $\overline{\SS}=\{\overline{\eta}_1,\overline{\eta}_2,\ldots,\overline{\eta}_k\}$. We have that the set $\{X_{\overline{\eta}_1},X_{\overline{\eta}_2},\ldots,X_{\overline{\eta}_k}\}$ generates a heart of $\pvq$ of quotient type.
\end{proposition}
\begin{proof}
We prove that the heart is of quotient type by selecting a refinement. We take $\AS$ to be any refinement of $\uA$ and $\h_{\AS}$ to be the heart associated to the dual $\dAS=\{\we_1,\we_2,\ldots,\we_n\}$ generated by its simples $X_{\we_1}, X_{\we_2},\ldots, X_{\we_n}$, where $\pi(X_{\we_i})=X_{\overline{\eta}_i}$ for $1\leq i\leq k$ by \Cref{prop:objcor}. Then $\h_{\AS}\cap\pvs$ is a finite heart of $\pvs$ which is generated by the set of simples $X_{\we_{k+1}},,\ldots, X_{\we_n}$. The finiteness of hearts ensure that $\h_{\AS}\cap\pvs$ is a Serre subcategory of $\h$. Thus the set of simples $\{X_{\overline{\eta}_1},X_{\overline{\eta}_2},\ldots,X_{\overline{\eta}_k}\}=\{\pi(X_{\we_1}), \pi(X_{\we_2}),\ldots, \pi(X_{\we_n})\}$ generate a quotient heart $\h_{\uA}=\h_{\AS}/\h_{\AS}\cap\pvs$ of $\pvq$.
\end{proof}
\begin{theorem} \label{thm:EGiso}
There is an isomorphism
\begin{equation}
\EGb(\colwsur)\cong \EGb(\pvq)
\end{equation}
of the principal parts of corresponding exchange graphs, which is compatible with flips and simple tilting. In particular, $\EGb(\pvq)$ is a union of connected components of
$\EG(\pvq)$.
\end{theorem}
\begin{proof}
We construct a map $\phi:\EGb(\colwsur)\to\EGb(\pvq)$ as follows. As for vertices, by \Cref{prop:vertex}, we get a well-defined map $\phi$ sending a mixed angulation $\uA$ in $\EGb(\colwsur)$, whose dual is $\SS=\{\overline{\eta}_1, \overline{\eta}_2,\ldots, \overline{\eta}_k\}$, to the heart in $\pvq$, which is generated by the set of simples $\{X_{\overline{\eta}_1},X_{\overline{\eta}_2},\ldots,X_{\overline{\eta}_k}\}$. The injective part of $\phi$ follows from the injectivity of the arc-to-object correspondence in \Cref{prop:objcor} and the surjectivity part follows from \eqref{EGpvdEG} in \Cref{cor:ass}. 

We now consider the edges. For any forward flip $\uA\stackrel{\overline{\gamma}}{\to}\uAgs=\uA'$ in $\EGb(\colwsur)$, by \Cref{prop:refine_of_flip}, we can lift it to a finite sequence forward flips
\[\AS\xrightarrow{\wg}\A_{\wg}^{\sharp}\to\cdots\to\AS'\] in $\EGp(\sow)$, where $\AS$ and $\AS'$ are refinements of $\uA$ and $\uA'$ respectively, and $\wg$ is a graded open arc in $\sow$ whose image is $\overline{\gamma}$ under the collapse operation. By \eqref{EGpvdEG}, the series gives a series of corresponding simple tiltings
\begin{equation}\label{eq:heart}
\h\xrightarrow{S}\h_{S}^{\sharp}\to\cdots\to\h',
\end{equation}
where $\h$ and $\h'$ are the corresponding hearts of $\A$ and $\A'$ respectively, and $S$ is the object associated to the dual arc $\we$ of $\wg$. We want to show that \eqref{eq:heart} gives a simple tilting of quotient heart in $\pvd(\colwsur)$:
\[\overline{\h}\xrightarrow{\overline{S}}\overline{\h}'.\]
By Remark 3.3 in \cite{KQ1}, it suffices to verify the following three things of for the simple tilting (forward flip) with respect to types I–IV:
\begin{enumerate}
\item $\pvd(\subsur)\cap\h'$ is a Serre subcategory of $\h'$, 
\item $\overline{\h}\leq \overline{\h}' \leq \overline{\h}[1]$,
\item $\langle\overline{S}\rangle = \overline{\h}'[-1] \cap \overline{\h}$.
\end{enumerate}
For type I, it is shown in Theorem 5.9 in \cite{BMQS}. For type II to IV, we consider the given finite sequence of simple tiltings. By arc-to-object correspondence in \Cref{prop:CHQ}, $\pvd(\subsur)\cap\h'$ is a heart of $\pvd(\subsur)$ generated by simples in $\pvd(\subsur)$, which is associated to those graded open arcs inside the monogon bounded by the bigon or $\wg$. Hence, (1) follows. Furthermore, at each step, by \Cref{prop:refine_of_flip} and \Cref{rmk:ind}, the tilting procedure only at the simple whose image in $\pvd(\colwsur)$ is isomorphic to $\overline{S}$. Thus for any object $X\in\h$, if it has a filtration with factors $S$ and other objects in $\pvd(\subsur)$ or in $\pvd(\sow)\backslash\pvd(\subsur)$, then under this process, the factors in $\pvd(\sow)\backslash\pvd(\subsur)$ will not change and $S$ becomes to $Y$, whose image in $\pvd(\colwsur)$ is $\overline{S}[1]$. Thus (2) and (3) follow. Thus we prove that the series in \eqref{eq:heart} consisting of simple tiltings induce a quotient simple tilting $\overline{\h}\xrightarrow{\overline{S}}\overline{\h}'$, and this is indeed an edge in $\EGb(\pvq)$. We also illustrate how the simples of type~II and type~III change under this finite sequence of simple tiltings in \Cref{sec:B}. The case of type~IV is analogous. 

Conversely, every edge in $\EGb(\pvq)$ arises by definition from a series of simple tiltings in $\EGp(\pv)$, which corresponds to a series of flips in between mixed-angulations $\EGp(\sow)$ by \eqref{EGpvdEG}. This gives rise by definition to an edge in $\EGb(\colwsur)$. Thus $\phi$ is indeed an isomorphism.
\end{proof}
\begin{remark}
We generalize Theorem 5.9 in \cite{BMQS}, where decorations of weight $-1$ are not treated. Decorations of weight $-1$ correspond to quadratic differentials with simple poles. In \Cref{sec:5}, we will prove the Bridgeland–Smith correspondence for quadratic differentials with zeros and poles of arbitrary order as a corollary. It can be viewed as a contribution to the compactification of the space of stability conditions, where the boundary strata correspond to degenerations of quadratic differentials with higher-order singularities.
\end{remark}
\section{Categories via perverse schobers}\label{sec:4}
In this section, we consider the case where the category associated to a surface arises via a perverse schober parametrized by the surface’s spanning graph, which is a ribbon graph. We construct the quotient perverse schober such that its corresponding category of global section is the Verdier quotient obtained via the collapse process, providing a variation of \Cref{thm:EGiso}.
\subsection{Background on perverse schobers}\label{sec:4.1}
We first summarize the basic perliminaries on perverse schober in this section. The preprint \cite{KS} by Kaparanov-Schecthman in 2014 gives a categorification of a perverse sheaf on disk with a single singularity and later Merlin Christ generalizes it to oriented marked surfaces with ribbon graphs, cf. \cite{Cginzburg,CSpherical}. We refer to \cite{KL, L} for fundamental concepts on $\infty$-categories and \cite{CHQ,Crelative} for perverse schobers. To begin with, we mainly follow the notations below in \cite{Cginzburg}:
\begin{itemize}
\item $\Cat_{\infty}$ denotes the $\infty$-category of (small) $\infty$-categories.
\item $\St$ denotes the subcategory of $\Cat_{\infty}$ with stable $\infty$-categories and exact functors.
\item $\Prl$ (resp. $\Prr$) denotes the subcategory of $\Cat_{\infty}$ with presentable and stable $\infty$-categories and left (resp. right) adjoint functors.
\end{itemize} 
Notice that there is an equivalence of $\infty$-categories between $\Prl$ and $(\Prr)^{op}$ via taking the adjoint.
\paragraph{\textbf{Perverse schobers}}
For simplicity, we let $\k$ be a field. For an integer $n$, an \emph{$n$-spider} is the ribbon graph with one vertex $v$ and $n$ adjacent external edges.
\begin{definition}\label{def:schberngon}
A \emph{perverse schober on the $n$-spider} is defined as follows.
\begin{itemize}
\item If $n=1$, it is given by a $\k$-linear \emph{spherical adjunction} between $\k$-linear stable $\infty$-categories
\[F\colon \hh{V}\rightleftarrows \hh{N}\cocolon G,\]
satisfying that $T_{\hh{V}}\in\Fun(\hh{V},\hh{V})$, the cofiber of the unit $\on{id}_{\hh{V}}\to G\circ F$ and $T_{\hh{N}}\in\Fun(\hh{N},\hh{N})$, the fiber of the counit $F\circ G\to \on{id}_{\hh{N}}$ are equivalences. We call $F$ and $G$ \emph{spherical functors}, $T_{\hh{V}}$ \emph{twist functor} and $T_{\hh{N}}$ \emph{cotwist functor}.
\item If $n\geq2$, it is given by
a collection of $\k$-linear adjunctions
\[ (F_i\colon \hh{V}^n\rightleftarrows \hh{N}_i\cocolon G_i)_{i\in \ZZ/n}\]
between $\k$-linear stable $\infty$-categories, if
\begin{enumerate}
    \item The counit $F_i\circ G_i\to \on{id}_{\hh{N}_i}$ is an equivalence,
    \item $F_{i}\circ G_{i+1}$ is an equivalence of $\infty$-categories,
    \item $F_i\circ G_j\simeq 0$ if $j$ does not equal to $i$ or $i+1$,
    \item $G_i$ has a right adjoint $(G_i)^R$ and $F_i$ has a left adjoint $^LF_i$ and
    \item $\im(G_{i+1})=\im(^LF_i)$ as full subcategories of $\hh{V}^n$.
\end{enumerate}
\end{itemize}
\end{definition}
\begin{definition}
For a ribbon graph $\bG$, the \textit{exit path category} $\on{Exit}(\bG)$ is the $1$-category whose objects are the vertices and edges of $\bG$. For each half-edge $a$ of $\bG$,  which belongs to an edge $e$ incident to a vertex $v$ of $\bG$, there is a morphism $v\xrightarrow{a} e$. All other morphisms act as identities.
\end{definition}
Consider a vertex $v$ of valency $n$ in a ribbon graph $\bG$. The undercategory $\on{Exit}(\bG)_{v/}$ consists of $n+1$ objects, corresponding to $v$ itself and its $n$ adjacent half-edges. The only non-identity morphisms in this category are those mapping $v$ to each of these half-edges.
\begin{definition}\label{def:schober}
Let $\bG$ be a ribbon graph. A functor $\hF\colon \on{Exit}(\bG)\to \on{LinCat}_{\k}$\footnote{Here $\on{LinCat}_{\k}$ is the $\infty$-category consisting of left modules over the symmetric monoidal derived $\infty$-category $\D(\k)$ in $\Prl$.} is said to be a \textit{$\bG$-parametrized perverse schober} if, for each vertex $v$ in $\bG$, its restriction to $\on{Exit}(\bG)_{v/}$ gives rise to a perverse schober on the $n$-spider. 
The \emph{generic stalk} of the functor $\F$ parametrized by $\bG$ is the stable $\infty$-category (up to equivalence) assigned to each edge of $\bG$. A vertex $v$ is called a \emph{singularity} of $\hF$ if the spherical functor is non-zero.
\end{definition}

\begin{definition}\label{def:sections}
Let $\hF$ be a $\bG$-parametrized perverse schober. We define \[\glo(\bG,\hF)\colon=\on{lim}(\hF)\] to be the limit of $\hF$ in the $\infty$-category $\on{LinCat}_k$ and call it the $\infty$-\emph{category of global sections} of $\hF$. 
\end{definition}
We have that $\glo(\bG,\hF)$ can be identified with the $\infty$-category of coCartesian sections of the covariant Grothendieck construction or Cartesian sections of the contravariant Grothendieck construction of $\hF$. The relevant concepts can be found in Definition 3.2.5.2 in \cite{L}, \cite{KL} and Section 1.2 in \cite{CSpherical}. We define $\losec(\bG,\hF)$ to be the $\infty$-category of all sections of the contravariant Grothendieck construction of $\hF$ and call it the $\infty$-\emph{category of local sections} of $\hF$. Moreover, the inclusion functor $\glo(\bG,\hF)\subset \losec(\bG,\hF)$ is fully faithful.
\paragraph{\textbf{Semiorthogonal decomposition}}
Let $\C$ be a stable $\infty$-category. Then its \emph{homotopy category} $\ho\C$ is a triangulated category. Two stable subcategories $\hh{A},\hh{B} \subseteq \C$ are called \emph{orthogonal} if the mapping space $\Map_{\hh{C}}(a, b)$ is contractible for all $a\in\hh{A}$ and $b\in\hh{B}$. For a stable subcategory $\hh{A} \subseteq \C$, its \emph{left} and \emph{right orthogonals} are defined as the full $\infty$-subcategories $^{\perp}\hh{A}$ and $\hh{A}^{\perp}$ consisting of objects given by:
\begin{align}
\begin{split}
^{\perp}\hh{A}&\coloneq\{c\in\C\mid\Map_{\hh{C}}(c, a)~\mbox{is contractible for all}~a~\mbox{in}~\hh{A}\},\\
\hh{A}^{\perp}&\coloneq\{c\in\C\mid\Map_{\hh{C}}(a, c)~\mbox{is contractible for all}~a~\mbox{in}~\hh{A}\}.
\end{split}
\end{align}
\begin{definition}\label{def:SOD}
Consider two full stable subcategories $\hh{A}$ and $\hh{B}$ of $\C$. The pair $(\hh{A},\hh{B})$ is said to form a \emph{semiorthogonal decomposition} (SOD for short) if
\begin{enumerate}  
\item For every $a \in \hh{A}$ and $b \in \hh{B}$, the mapping space $\Map_{\C}(b, a)$ is contractible.  
\item Each object $c \in \C$ fits into a exact sequence $b \to c \to a$ with $a \in \hh{A}$ and $b \in \hh{B}$.  
\end{enumerate}
\end{definition}
\begin{proposition}\label{prop:SOD}
We collect several results from Sections 2.2 and 2.3 in \cite{DKSS}.  
\begin{enumerate}  
\item Let $\hh{A}$ and $\hh{B}$ be strictly full stable $\infty$-subcategories of $\C$.  
If $(\hh{A},\hh{B})$ is an SOD of $\C$, then  
$$ \hh{B} = {}^{\perp} \hh{A} \quad\mbox{and}\quad \hh{A} = \hh{B}^{\perp}. $$  
\hfill (Prop. 2.2.4)  

\item A pair $(\hh{A}, \hh{B})$ forms an SOD of $\C$ if and only if the associated pair  
of full triangulated subcategories $(\ho\hh{A}, \ho\hh{B})$ forms an SOD of the triangulated  
category $\ho\C$.  
\hfill (Cor. 2.2.5)  

\item Let $\hh{A}$ be a full stable subcategory of $\C$.  
The pair $(\hh{A}, {}^{\perp} \hh{A})$ (resp. $(\hh{A}^{\perp}, \hh{A})$) forms an SOD of $\C$  
if and only if the inclusion $\hh{A} \to \C$ has a left (resp. right) adjoint.  
\hfill (Prop. 2.3.2)  
\end{enumerate}  
\end{proposition}
\begin{definition}\label{def:SODschobers}
Let ${\bf G}$ be a ribbon graph. A \emph{semiorthogonal decomposition} $\{\F_1,\F_2\}$ of a ${\bf G}$-parametrized perverse schober $\hh{G}$ consists of the following data:  
\begin{enumerate}  
\item an inclusion $\alpha_1\colon\F_1\to \hh{G}$ in $\Fun(\Exit(\bG),\Prr)$,
\item an inclusion $\alpha_2\colon\F_2\to \hh{G}$ in $\Fun(\Exit(\bG),\Prl)$,
\item a condition that $\alpha_2$ is right admissible, and
\item a requirement that the cofiber of $\alpha_2$ in $\on{Fun}(\on{Exit}({\bf G}),\Prl)$ is $\F_1$, with cofiber morphism $\beta\colon \hh{G}\to \F_1$ being pointwise left adjoint to $\alpha_1$, ensuring that $\alpha_1$ is left admissible.  
\end{enumerate}  
For the definitions of inclusion and right/left admissible morphisms for perverse schobers, we refer to Definition 3.14 and Remark 3.15 in \cite{Ccluster}.
\end{definition}
\subsection{Categorical realization of collapsing via quotient perverse schobers}\label{sec:P2}
Let $\bG$ be a ribbon graph and $\Gs$ be a subgraph of $\bG$, where the edges connecting the vertices in $\Gs$ are completely contained in $\Gs$. Assume further that the cyclic orders at each vertex are compatible under the inclusion $\Gs\hookrightarrow\bG$. Similarly, we have the collapse action on the graph, see \Cref{fig:collapsedgraph} as an example.
\begin{definition}\label{def:collapseofgraph}
The collapsed graph $\Gq$ is obtained by contracting all the vertices of $\Gs$ into a single vertex $\vq$, along the edges in $\Gs$.
\end{definition}

Let $\F$ be a $\bG$-parametrized perverse schober, and let $\Fs$ be a $\Gs$-parametrized perverse schober obtained via restriction of $\F$. Specifically, locally at each spider of $\Gs$, the spherical adjunction of $\Fs$ coincides with the spherical adjunction of $\F$. Our aim is to construct a natural perverse schober $\Fquo$ that is parametrized by $\Gq$. This construction is carried out inductively.
\paragraph{\textbf{Base case: when $\Gs$ has two vertices}}
The first step of induction is to prove the following base case:
\begin{itemize}
\item The subgraph $\Gs$ consists of two vertices, namely $v$ and $v'$, which are also vertices of $\bG$.
\item There are several edges in $\Gs$ connecting $v$ and $v'$.
\end{itemize}
At $v$, let the set of half-edges in $\bG$ be $E\coloneq\{e_i \mid i\in I\}$ and that in $\Gs$ be $\Es\coloneq\{e_i\mid i\in\Is\}$, where we have $\Is\subseteq I$ and $\Es\subseteq E$. Similarly, at $v'$, let the set of half-edges in $\bG$ be $\{e'_i\mid i\in I'\}$ and that in $\Gs$ be $\{e'_i\mid i\in\Is'\}$, where we have $\Is'\subseteq I'$ and $\Es'\subseteq E'$. The bijection between $\Es$ and $\Es'$ induces the following bijection:
\[\psi:\Is\to\Is',\]
where $e_{k}$ and $e'_{\psi(k)}$ are connected by an edge in $\Gs$.

By the construction of the collapse, we denote the new vertex in $\Gq$ by $\vq$ and its half-edges by $\Eq\coloneq\{\eq_i\mid i\in\Iq\}$, where there is a natural cyclic order on it. Thus, the bijection between $\Eq$ and $E\cup E'-(\Es\cup\Es')$ gives a map as follows:
\begin{equation}
\phi:\Iq\to I\backslash\Is\cup I'\backslash\Is',
\end{equation}
where $\eq_i$ corresponds to $e_{\phi(i)}$ if $\phi(i)\in I\backslash\Is$, or corresponds to $e'_{\phi(i)}$ if $\phi(i)\in I'\backslash\Is'$ under the collapse construction.
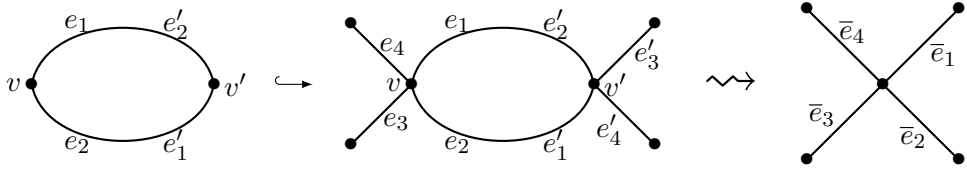
\begin{figure}
\begin{tikzpicture}[scale=1]
  \draw[thick] (-1.2,0) .. controls (-1,1) and (1,1) .. (1.2,0);
  \draw[thick] (-1.2,0) .. controls (-1,-1) and (1,-1) .. (1.2,0);
\draw[thick] (1.2,0)\nn(-1.2,0)\nn;
\draw (-1.2,0)node[left]{$v$} (1.2,0)node[right]{$v'$} (-.6,.8)node{$e_1$}(-.6,-.8)node{$e_2$} (.7,.8)node{$e_2'$} (.7,-.8)node{$e_1'$};
\begin{scope}[shift={(5,0)}]
 \draw[thick] (-1.2,0) .. controls (-1,1) and (1,1) .. (1.2,0);
  \draw[thick] (-1.2,0) .. controls (-1,-1) and (1,-1) .. (1.2,0);
  \draw[thick] (-1.2,0) -- (-2,.8) (-1.2,0) -- (-2,-.8) (1.2,0) -- (2,.8) (1.2,0) -- (2,-.8);
\draw[thick] (1.2,0)\nn(-1.2,0)\nn;
\draw[thick] (-2,.8)\nn(-2,-.8)\nn(2,.8)\nn(2,-.8)\nn;
\draw(-1.45,.5)node{$e_4$}(-1.4,-.5)node{$e_3$}(1.9,.4)node{$e_3'$}(1.4,-.6)node{$e_4'$};
\draw (-1.2,0)node[left]{$v$} (1.2,0)node[right]{$v'$} (-.6,.8)node{$e_1$}(-.6,-.8)node{$e_2$} (.7,.8)node{$e_2'$} (.7,-.8)node{$e_1'$};
\end{scope}
\begin{scope}[shift={(10,0)}]
\draw[thick] (0,0)--(-1,1) (0,0)--(-1,-1) (0,0)--(1,1) (0,0)--(1,-1);
\draw[thick] (0,0)\nn(-1,1)\nn(-1,-1)\nn(1,1)\nn(1,-1)\nn;
\draw (.8,.4)node{$\overline{e}_1$} (.4,-.7)node{$\overline{e}_2$} (-.8,-.4)node{$\overline{e}_3$} (-.4,.7)node{$\overline{e}_4$};
\end{scope}
\draw(8,0)node{\huge{$\rightsquigarrow$}};
\draw[right hook-latex,>=stealth](2,0)to(2.5,0);
\end{tikzpicture}
    \caption{An example of collapsed graph}
    \label{fig:collapsedgraph}
\end{figure}
\paragraph{\textbf{Base case on perverse schobers}}
If we denote $|I|=n$, $|I'|=n'$, $|\Is|=|\Is'|=m$, then we have $|\Iq|=N=n+n'-2m$. Let the perverse schobers on the $n$-spider $v$ and the $m$-spider $v'$ be associated with the collections of $\k$-linear adjunctions between $\k$-linear stable $\infty$-categories as
\[ (F_i\colon \hh{V}\rightleftarrows \hh{N}_i\cocolon G_i)_{i\in \ZZ/n}\]
and
\[ (F_i'\colon \hh{V}'\rightleftarrows \hh{N}_i\cocolon G_i')_{i\in \ZZ/n'}\]
respectively, where all $\hh{N}_i$'s are equivalent, denoted by $\hh{N}$. We consider taking the limit locally in a perverse schober. Specifically, we examine the subdiagram $\F'$ (cf. the part in \Cref{fig:computelimF'}) of $\hF$, which corresponds to the structure of $\Fs$. The local limit of $\F'$ is computed and assigned to $\vq$, while the parts of $\hF$ that do not involve $\F'$ remain unchanged. See the \Cref{fig:computelimF'} for reference.
\begin{figure}
\begin{tikzcd}
       &                                             & \textcolor{blue}{\lim\hh{F}'} \arrow[ld,blue,"\res_{\hh{V}}"'] \arrow[rd,blue, "\res_{\hh{V'}}"] &                                               &        &                      &    &                                                                      &        \\
       & \textcolor{red}{\hh{V} } \arrow[ld, "F_2"'] \arrow[rd,red, "F_1"] &                            & \textcolor{red}{\hh{V}'} \arrow[ld, red,"F_2'"'] \arrow[rd, "F_1'"] &        & {\rightsquigarrow}  & {} &       \textcolor{blue}{\lim\hh{F}'} \arrow[ld, "\overline{F}_2"'] \arrow[rd, "\overline{F}_1"] &        \\
\hh{N} &                                             & \textcolor{red}{\hh{N}}                     &                                               & \hh{N} &                        & \hh{N} &                                                                 & \hh{N}
\end{tikzcd}
    \caption{The computation of $\lim\F'$, where $\F'$ is the diagram marked in red}
    \label{fig:computelimF'}
\end{figure}
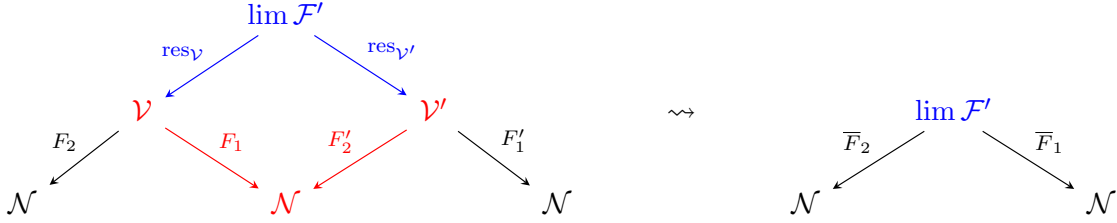
We consider the Cartesian contravariant Grothendieck construction $p:\chi(\F')\to N(\Exit(\Gs))$ of $\F'$ over the nerve $N(\Exit(\Gs))$ whose fibers are equivalent to the values of $\hF$. We denote $\hh{L}$ the $\infty$-category of sections of $p$. Then, $\lim\F'$ can be identified with the full subcategory of $\hh{L}$ spanned by Cartesian sections of $p$. 

\begin{construction}\label{const:collapadj}
We aim to construct the following $\k$-linear adjunctions between $\k$-linear stable $\infty$-categories:
\begin{equation}\label{eq:perv}
(\overline{F}_i\colon \lim\F'\rightleftarrows \hh{N}\cocolon \overline{G}_i)_{i\in \ZZ/N},
\end{equation}
using the strategy similar to Construction 5.5.3 in \cite{CDW}. Since that $\eq_i$'s are obtained from $e_{\phi(i)}$ or $e_{\phi(i)}'$, we define $\overline{F}_i$ as
\[\overline{F}_i=  \begin{cases} F_{\phi(i)}\circ \res_{\hh{V}} & \mbox{if}~\eq_i~\mbox{is from}~e_{\phi(i)},\phi(i)\in I\backslash\Is,\\ F_{\phi(i)}'\circ \res_{\hh{V'}} & \mbox{if}~\eq_i~\mbox{is from}~e'_{\phi(i)}, \phi(i)\in I'\backslash\Is'.\end{cases}\]
We construct $\overline{G}_i$'s as follows. We define $\hh{E}_{\hh{V}}, \hh{E}_k$ and $\hh{E}_{\hh{V'}}$ the full subcategory of $\hh{L}$ spanned by $p$-relative right Kan extension of their restrictions to $\hh{V}, \hh{N}_k$ and $\hh{V'}$ respectively, where $k\in\Is$. By Proposition 4.3.2.15 in \cite{L}, we deduce that the restrictions
\begin{equation}\label{eq:res}
\hh{E}_{\hh{V}}\xrightarrow{\res_{\hh{V}}}\hh{V},\quad \hh{E}_{k}\xrightarrow{\res_{\hh{N}_k}}\hh{N},~~ \mbox{and} ~~\hh{E}_{\hh{V'}}\xrightarrow{\res_{\hh{V'}}}\hh{V'}    
\end{equation}
are trivial fibrations. We take their sections and by passing to opposite $\infty$-category of Proposition 4.3.2.17 in \cite{L}, we obtain that these sections are the corresponding right adjunctions. For each $i\in N$ with $\phi(i)\in I\backslash\Is$, we define $H_{\hh{V}}$ by composing $G_{\phi(i)}$, the section of $\hh{E}_{\hh{V}}\xrightarrow{\res_{\hh{V}}}\hh{V}$ and inclusion as:
\[H_{\hh{V}}: \hh{N}\xrightarrow{G_{\phi(i)}}\hh{V}\to\hh{E}_{\hh{V}}\subseteq\hh{L}.\]
We define
\[H_{\hh{N}_k}:\hh{N}\to\hh{E}_{k}\subseteq\hh{L}\]
by composing the section of $\hh{E}_{k}\xrightarrow{\res_{\hh{N}_k}}\hh{N}$ with inclusion. We define $H_{\hh{V'},k}$ to be the composition of $G_{\psi(k)}'$, the section of $\hh{E}_{\hh{V'}}\xrightarrow{\res_{\hh{V'}}}\hh{V'}$ and inclusion as:
\[H_{\hh{V'},k}: \hh{N}\xrightarrow{G_{\psi(k)}'}\hh{V'}\to\hh{E}_{\hh{V'}}\subseteq\hh{L}.\]
For each $k\in \Is$, we have the following natural transformations
\begin{equation}\label{eq:nat1}
\eta_k: H_{\hh{N}_k}\circ F_k\circ G_{\phi(i)}\to H_{\hh{V}}
\end{equation}
and
\begin{equation}\label{eq:nat2}
\eta'_k:H_{\hh{N}_k}\simeq H_{\hh{N}_k}\circ F_{\psi(k)}'\circ G_{\psi(k)}'\to H_{\hh{V}',k},
\end{equation}
where the equivalence is from the first condition of a perverse schober at $v'$ in \Cref{def:schberngon}. We give the explicit expressions of $\eta_k$ and $\eta_k'$ in the \Cref{sec:A1}.
For those $k\in\Is$ such that $k+1\in\Is$, we have $\psi(k)\in \Is'$ and $\psi(k+1)=\psi(k)-1\in \Is'$. Then we have the equivalences $F_{k}\circ G_{k+1}$ and $F'_{\psi(k+1)}\circ G'_{\psi(k1)}$, which imply the following natural transformation
\begin{align}
\begin{split}
\sigma_k: H_{\hh{V}',k}\circ F_k\circ G_{\phi(i)}&\to H_{\hh{V}',k}\circ F_k\circ G_{k+1}\circ F_{k+1}\circ G_{\phi(i)}\\
&\to H_{\hh{V}',k+1}\circ F'_{\psi(k+1)}\circ G'_{\psi(k)}\circ F_k \circ G_{k+1}\circ F_{k+1}\circ G_{\phi(i)}\\
&\simeq H_{\hh{V}',k+1}\circ F_{k+1}\circ G_{\phi(i)},
\end{split}
\end{align}
where the maps are given by the units of the corresponding adjunctions. Therefore, we define $\overline{G}_i:\hh{N}\to \hh{L}$ as the colimit of the following diagram
\begin{equation}\label{eq:lim}
\begin{tikzpicture}[>=latex,scale=1.2]
    \node (A1) at (-3,2) {$H_{\hh{V}}$};
    \node (A2) at (0,4) {$H_{\hh{N}_1}\circ F_1\circ G_{\phi(i)}$};
    \node (A3) at (0,2) {$H_{\hh{N}_2}\circ F_2\circ G_{\phi(i)}$};
    \node (A4) at (6.5,2) {$H_{\hh{V}',m}\circ F_m\circ G_{\phi(i)}.$};
    
    \node (B1) at (0,1) {$\vdots$};  
    
    \node (C1) at (0,0) {$H_{\hh{N}_m}\circ F_m\circ G_{\phi(i)}$};

    \draw[->,font=\scriptsize] (A2) -- (A1) node[midway,above] {$\eta_1$};
    \draw[->,font=\scriptsize] (A2) -- (A4) node[midway,above,sloped] 
        {$\sigma_{m-1} \circ \ldots \circ \sigma_1 \circ \eta_1' \circ F_1 \circ G_{\phi(i)}$};
    
    \draw[->,font=\scriptsize] (A3) -- (A1) node[midway,above,sloped] {$\eta_2$};
    \draw[->,font=\scriptsize] (A3) -- (A4) node[midway,above,sloped] 
        {$\sigma_{m-1} \circ \ldots \circ \sigma_2 \circ \eta_2' \circ F_2 \circ G_{\phi(i)} \circ \sigma_1$};
    
    \draw[->,font=\scriptsize] (C1) -- (A1) node[midway,left] {$\eta_m$};
    \draw[->,font=\scriptsize] (C1) -- (A4) node[midway,below,sloped] {$\eta_m' \circ F_m \circ G_{\phi(i)}$};

\end{tikzpicture}
\end{equation}
We will verify in \Cref{sec:A2} that $\overline{G}_i$ factors through $\lim\F'\subseteq\hh{L}$, allowing us to regard $\overline{G}_i$ as a functor $\hh{N}\to\lim\F'$. If $i$ satisfies that $\phi(i)\in I'\backslash\Is'$, the construction is similar, where we exchanging the roles of $F$ and $F'$.
\end{construction}
\begin{lemma}\label{prop:limscho}
The $\k$-linear adjunctions
\[(\overline{F}_i\colon \lim\F'\rightleftarrows \hh{N}\cocolon \overline{G}_i)_{i\in \ZZ/N}\]
in \Cref{const:collapadj} is a perverse schober on the $N$-spider with vertex $\vq$.
\end{lemma}
This lemma can also be viewed as providing an alternative proof of (a special case of) Theorem 4.6 in \cite{Cct}. The proof of the lemma will be given in \Cref{sec:A3}. In general, this lemma provides a method to construct a perverse schober on a general collapsed graph, not just by collapsing a connected subgraph with two vertices. Moreover, $\Gs$ can also has external edges which are external edges in $\bG$.
\begin{proposition}\label{propp:limscho}
For a ribbbon graph $\bG$ with a connected subgraph $\Gs$ with more than two vertices, we have a natural perverse schober on the $N$-spider with vertex $\vq$ given by the $\k$-linear adjunctions
\begin{equation}\label{eq:prop}
(\overline{F}_i\colon \lim\F'\rightleftarrows \hh{N}\cocolon \overline{G}_i)_{i\in \ZZ/N},
\end{equation}
where $\vq$ is the vertex in \Cref{def:collapseofgraph} and $N$ denotes the number of half-edges at $\vq$.
\end{proposition}
\begin{proof}
This follows from the fact that limits can be formed locally, allowing us to obtain a perverse schober by combining \Cref{prop:limscho} and iteratively collapsing pairs of two vertices.
\end{proof}
\paragraph{\textbf{Quotient perverse schober parametrized by the collapsed graph}}

We define the following two $\Gq$-parametrized perverse schobers as follows:
\begin{itemize}
\item $\hh{G}$, whose restriction to $\on{Exit}(\bG)_{\vq/}$ is the schober given by the $\k$-linear adjunctions in \eqref{eq:prop}; and whose restriction to $\on{Exit}(\bG)_{v/}$ is same as the one of $\hF$ for other vertices $v$ in $\Gq$.
\item $\hh{F}_0$, whose restriction to $\on{Exit}(\bG)_{\vq/}$ is the schober given by the zero $\k$-linear adjunctions $(0\colon \lim\Fs\rightleftarrows 0\cocolon 0)_{i\in \ZZ/N}$; and whose restriction to $\on{Exit}(\bG)_{v/}$ is given by $(0\colon 0\rightleftarrows 0\cocolon 0)$ for other vertices $v$ in $\Gq$.
\end{itemize}
We construct the quotient perverse schober on $\Gq$ by taking semiorthogonal decomposition. More precisely, the inclusion $\lim\Fs\to\lim\F'$ induces an inclusion of perverse schobers $\hh{F}_0\hookrightarrow\hh{G}$, refer to the schematic diagram below
\begin{equation}\label{eq:sub}
\begin{tikzcd}
\hh{F}_0:\arrow[d,"\alpha_2"]& 0 \arrow[d]      & \pb \arrow[d] \arrow[r] \arrow[l]   & 0 \arrow[d]      \\
\hh{G}:& \hh{N}  & \lim\F' \arrow[l] \arrow[r]  & \hh{N}.
\end{tikzcd}
\end{equation} 

Since we have the pushout
\[\begin{tikzcd}
\lim\Fs \arrow[r, hook] \arrow[d] & \lim\F' \arrow[d] \\
0 \arrow[r]              & \lim\F'/\pb,     
\end{tikzcd}\]
then $\lim\Fs\to\lim\F'$ admits a right adjoint. By \Cref{prop:SOD}, we have that $(\pb^{\perp},\pb)$ is an SOD of $\lim\F'$ and an equivalence $\pb^{\perp}\simeq \lim\F'/\pb$.
\begin{lemma}
The collection of $\k$-linear adjunctions of stable $\infty$-categories
\begin{equation}\label{eq:quoscho}
(\widetilde{F}_i\colon \pb^{\perp}\rightleftarrows \hh{N}\cocolon \widetilde{G}_i)_{i\in \ZZ/N}
\end{equation}
gives a perverse schober on the $N$-spider with vertex $\vq.$ In particular, we have a $\Gq$-parametrized perverse schober $\Fquo$, whose restriction to $\on{Exit}(\bG)_{\vq/}$ is the schober given by \eqref{eq:quoscho}; and whose restriction to $\on{Exit}(\bG)_{v/}$ is same as the one of $\hF$ for other vertices $v$ in $\Gq$.
\end{lemma}
\begin{proof}
That is because that if we identify $\pb^{\perp}$ with $\lim\F'/\pb$, the corresponding $\widetilde{F}_i$ is given by the induced morphism of $\overline{F}_i$ on the quotient and $\widetilde{G}_i$ is given by the composition $q\circ \overline{G}_i$, see the following diagram:
\[\begin{tikzcd}
\lim \arrow[d, "q"'] \arrow[rd, "\overline{F}_i", shift left] &                                      \\
\lim/\pb \arrow[r, "\widetilde{F}_i"', dashed]         & \hh{N}. \arrow[lu, "\overline{G}_i", shift left]
\end{tikzcd}\]
Then the composition $\widetilde{F}_i\circ\widetilde{G}_j$ is equivalent to $\overline{F}_i\circ\overline{G}_j$, which implies the result.
\end{proof}

\begin{proposition}\label{prop:sod}
There is a semiorthogonal decomposition of $\hh{G}$ given by the pair $(\Fquo, \hh{F}_0)$.
\end{proposition}
\begin{proof}
We refer to the schematic diagram below and verify the conditions in \Cref{def:SODschobers}.
\begin{equation}\label{eq:SODper}
\begin{tikzcd}
\hh{F}_0:\arrow[d,hook,"\alpha_2"]& 0 \arrow[d]      & \pb \arrow[d] \arrow[r] \arrow[l]   & 0 \arrow[d]      \\
\hh{G}:& \hh{N}  & \lim\F' \arrow[l] \arrow[r]  & \hh{N}  \\
\Fquo:\arrow[u,hook,"\alpha_1"']& \hh{N} \arrow[u]          &\pb^{\perp}\simeq\lim\F'/\pb\arrow[u] \arrow[l] \arrow[r]        & \hh{N}\arrow[u]          
\end{tikzcd}   
\end{equation}
It is clear that both $\alpha_1$ and $\alpha_2$ are inclusions, as the nontrivial maps in the diagram are inclusions given by the SOD of $\lim\F'$, while the others are either zero maps or the identity map $\id_{\hh{N}}$. Furthermore, $\alpha_2$ is right admissible, since the right adjoint of zero map, from zero category to arbitrary $\infty$-category are zero, and the composition with zero map is also zero. The last condition in \Cref{def:SODschobers} is automatically satisfied by the construction of the SOD of $\lim\F'$. Therefore, we have established the SOD of $\hh{G}$.
\end{proof}
\begin{theorem}\label{thm:exactseq}
We have a cofiber sequence
\begin{equation}\label{eq:fibco}
\glo({\Gs},\Fs)\to\glo({\bG},\F)\to\glo({\Gq},\Fquo)
\end{equation}
in $\St$, which induces a short exact sequence of triangulated categories
\begin{equation}\label{eq:fibcotri}
\Ho\big(\glo({\Gs},\Fs)\big)\to\Ho\big(\glo({\bG},\F)\big)\to\Ho\big(\glo({\Gq},\Fquo)\big)
\end{equation}
when by passing to the homotopy categories.
\end{theorem}
\begin{proof}
By Remark 3.17 in \cite{Ccluster}, we have that there exists a cofiber sequence
\[\glo({\Gq},\F_0)\xrightarrow{i} \glo({\Gq},\hh{G})\xrightarrow{\pi} \glo({\Gq},\Fquo)\,\]
in $\St$. In fact, we have $\glo({\Gq},\hh{G})\simeq \glo({\bG},\F)$, where the limit will not change under local limit action. On the other hand, we have
\[\glo({\Gq},\F_0)\simeq \lim\Fs\simeq \glo({\Gs},\Fs)\]
Thus we obtain the cofiber sequence in \eqref{eq:fibco}, which implies a Verdier quotient by passing to the homotopy categories.
\end{proof}
\begin{remark}
By Proposition 3.6 in \cite{CHQ}, we know that given a perverse schober $\F$ on the $n$-spider with vertex $v$, we can obtain the associated spherical adjunction by taking fiber along all but one half-edges attached to $v$ iteratively. Since limit commutes with limit, then we can obtain the spherical adjunction of the quotient perverse schober by first taking fiber along those half-edges not in $\Gs$ but attached to the vertices in $\Gs$ and then taking the limit.
\end{remark}
\paragraph{\textbf{An example}}
Let $\surf$ be an annulus with one marked point on each boundary component and $\bG$ be a spanning graph, which is drawn in red in the left one of \Cref{fig:kreo}. The green one is the subsurface $\subsur$ of $\surf$, which has one marked point on the inner boundary component while no marked points on the other one. As shown in the figure, the thick red graph $\Gs$ is a spanning graph of $\subsur$.
\begin{figure}[h]\centering
\begin{tikzpicture}
  \fill[cyan!10] (0,0) circle (2.2);
\fill[green!10] (0,0) circle (1.8);
  \fill[gray!10] (0,0) circle (.8);
  \draw[thick] (0,0) circle (2.2);
  \draw[thick] (0,0) circle (1.8);
  \draw[thick] (0,0) circle (.8);
  \draw[very thick, red] (0,0) circle (1.3);
\draw[thick, red] (180:1.3) to (180:2.2);
\draw[black] (180:.8)\nn (0:2.2)\nn; 
\draw[red](0:1.3)\nn (180:1.3)\nn;
\draw(3,0)node{\Huge{$\rightsquigarrow$}};
\begin{scope}[shift={(6.5,0)}]
\fill[cyan!10] (0,0) circle (2.2);
\draw[thick] (0,0) circle (2.2);
\draw[thick, red] (0:0) to (180:2.2);
\draw[black] (0:2.2)\nn; 
\draw[red] (0:0)node[right]{$\vq$}\nn;
\end{scope}
\end{tikzpicture}
\caption{An example of collapse}\label{fig:kreo}
\end{figure}
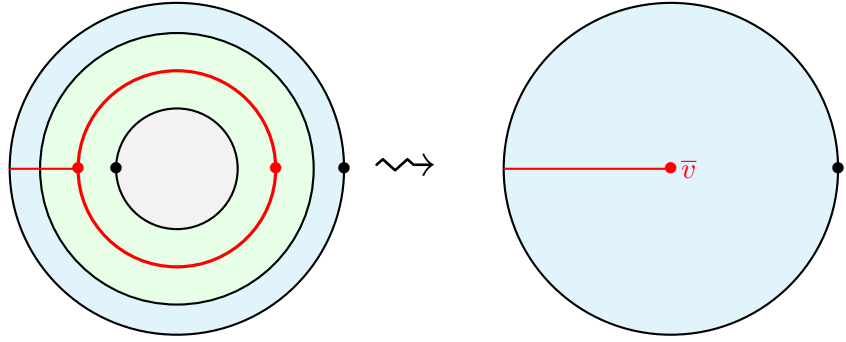
We consider the $\bG$-parametrized perverse schober associated with the spherical adjunction without singularities:
\[0\colon0\rightleftarrows \D(\k)\cocolon 0;\]
and $\Gs$-parametrized perverse schober associated with the same spherical adjunctions. We write the corresponding perverse schober diagrams as follows:
\[\begin{tikzcd}[column sep=15pt]
       &                                                                                        & \D(\k) &                                               &            &                                                   & \D(\k) &                                                   \\
\D(\k) & {\Fun(\Tri^1,\D(\k))} \arrow[l, "\cof"'] \arrow[ru, "\ev_1"] \arrow[rd, "{\ev_0[1]}"'] &        & \D(\k) \arrow[ld, "{[1]}"] \arrow[lu, "\id"'] & \mbox{and} & \D(A_1) \arrow[ru, "\id"] \arrow[rd, "{\id[1]}"'] &        & {\D(A_1).} \arrow[ld, "{[1]}"] \arrow[lu, "\id"'] \\
       &                                                                                        & \D(\k) &                                               &            &                                                   & \D(\k) &                                                  
\end{tikzcd}\]
Similar as shown in section 6.2 in \cite{CHQ}, the limit $\C(\bG,\F)$ of the left diagram is equivalent to the derived category $\D(\k[t])$ of graded polynomial algebra $\k[t]$ with degree 1, and the limit $\C(\Gs,\Fs)$ of the right diagram is equivalent to the derived category $\D(\k[t^{\pm}])$ of graded Laurent algebra $\k[t^{\pm}]$.

Let $\Sq$ be the collapsed surface, which is a disk with one marked point on the boundary, and $\Gq$ be the collapsed graph, which is a 1-spider with new vertex $\vq$. By \Cref{propp:limscho}, the perverse schober $\hh{G}$ parametrized by 1-spider with vertex $\vq$ is given by the spherical adjunctions
\[?\Lten_{\k}\k\colon\D(\k[t])\rightleftarrows \D(\k)\cocolon ?\Lten_{\k}\k[t].\]
By \Cref{prop:sod}, we get a semiorthogonal decomposition of the perverse schober given by the following spherical adjunctions:
\begin{equation}
\begin{tikzcd}
{\D(\k[t^{\pm}])} \arrow[r, shift left] \arrow[d] & 0 \arrow[d] \arrow[l, shift left]         \\
{\D(\k[t])} \arrow[r, shift left]                 & \D(\k) \arrow[l, shift left]              \\
{\D(\k[\varepsilon]/\varepsilon^2=0)} \arrow[u] \arrow[r, shift left]        & {\D(\k)} \arrow[u] \arrow[l, shift left]
\end{tikzcd}
\end{equation}
with $|\varepsilon|=-2$, and the last row gives the perverse schober $\Fquo$ parametrized by 1-spider with vertex $\vq$. By \Cref{thm:exactseq}, we obtain that $\C(\Gq, \Fquo)$ is equivalent to $\D(\k[\varepsilon]/\varepsilon^2=0)$, which is the quotient category assoicated to $(\Sq,\Gq)$.
\subsection{A variation of \Cref{thm:EGiso} using quotient perverse schober}
For the weighted DMS $\sow$ and $\subsur$ be its subsurface in \Cref{sec:sow}. Let $\AS$ be a mixed-angulation of $\sow$ and $\SS$ be its dual S-graph. We take $\bG$ to be $\SS$ and fix a $\SS$-parametrized perverse schober $\hF$ with a positive arc system kit in the sense of \cite{CHQ}. Then the $\infty$-category associated with the surface $\sow$ determined by $\hF$ is $\glo(\SS,\hF)$. By Proposition 4.3 in \cite{CHQ}, we have that $\glo(\SS,\hF)$ only depends on the choice of perverse schober but is independent of the ribbon graph $\SS$. That is, flips of dual graph $\SS$ gives equivalent $\infty$-categories. We know that its homotopy category, which is denoted by $\D(\sow)\coloneqq\Ho\big(\glo(\SS,\hF)\big)$, is a triangulated category. By Construction 4.12, Proposition 4.14, Lemma 4.16 and Corollary 5.1, $\D(\sow)$ satisfies the conditions in \cite{CHQ}.

As indicated in \Cref{sec:P2}, we have the spanning graph $\SS|_{\subsur}$ of $\subsur$, which is a subgraph of $\SS$ obtained by deleting the ones in $\SS$ not in $\subsur$; and $\Squo$ obtained from $\SS$ by replacing the subgraph $\eSsub$ with a single vertex under the collapse operation
\begin{tikzcd}  
\sow \ar[r,rightsquigarrow] & \colwsur.  
\end{tikzcd}  

Categorically, let $\Fsub$ be a $\SS|_{\subsur}$-parametrized perverse schober which is restricted from $\F$ and $\glo(\SS|_{\subsur},\Fsub)$ be the $\infty$-category of global sections. We denote its homotopy category by $\D(\subsur)$. By \Cref{thm:exactseq}, there exists a $\Squo$-parametrized quotient perverse schober $\Fquo$ with limit $\glo(\Squo,\Fquo)$, which fits into the following short exact sequence of triangulated categories: 
\begin{equation}\label{eq:ses.cats1}
\begin{tikzcd}
    0 \ar[r] & \D(\subsur) \ar[r] & \Dw \ar[r,"\pi"] &\D(\colwsur) \ar[r] & 0.
\end{tikzcd}
\end{equation}
Here $\D(\colwsur)$ is the homotopy category of $\glo(\Squo,\Fquo)$. Similarly, let $\pv$ be the subcategory of $\D(\sow)$ generated by $\{X_{\we}:\we~\mbox{is a graded closed arc in}~\SS\}$ and $\pvs$ be the subcategory of $\D(\subsur)$ generated by $\{X_{\we}:\we~\mbox{is a graded closed arc in}~\Ssub\}.$ Then \eqref{eq:ses.cats1} induces the following Verdier quotient of triangulated categories:
\begin{equation}\label{eq:ses.pvds1}
\begin{tikzcd}
    0 \ar[r] & \pvs \ar[r] & \pv \ar[r,"\pi"] &\pvq \ar[r] & 0,
\end{tikzcd}
\end{equation}
where $\pvq$ is the subcategory of $\D(\colwsur)$ generated by $\{X_{\we}\}$ with $\we$ being the graded closed arc in $\Squo.$

With the quotient perverse schober structure, we can easily obtain the \Cref{thm:EGiso} using the same method of Corollary 5.1 in \cite{CHQ}.
\section{Quadratic differentials as stability conditions}\label{sec:5}
\subsection{Moduli spaces of quadratic differentials}
We recall the local coordinates of the moduli space $\on{FQuad}(\sow)$. For any quadratic differential $\varphi$ on a Riemann surface $C$, let $\widehat{C}$ be the branched double cover of $C$, $\widehat{Z}$ and $\widehat{P}$ be the preimage of sets of zeros and poles of $\varphi$ on the double cover $\widehat{C}$. Consider the \emph{hat-homology group}
\[\widehat{H}_1(\varphi)=H_1(\widehat{C}\backslash\widehat{P},\widehat{Z},\mathbb{C})^-,\]
which is the anti-invariant part of the homology group with respect to the involution $\tau$ of the cover $\widehat{C}$. Given a flat surface $(C,\varphi)$, we have the integration map
\begin{equation}
\int : \widehat{H}_1(\varphi)\to\mathbb C, \quad \we\mapsto \int_{\we}\sqrt{\varphi}.
\end{equation}
In a neighborhood $U(\varphi)$ of $\varphi$ in $\on{FQuad}(\sow)$, the combination of all such maps gives a local homeomorphism
\begin{equation}
\label{period map}
\on{Per}: \mathrm{FQuad}(\sow)\to H^1(\widehat{C}\backslash\widehat{P},\widehat{Z},\mathbb{C})^-,
\end{equation}
called the \emph{period map} on $\mathrm{FQuad}(\sow)$. The affine structure on $\mathrm{FQuad}(\sow)$ can be defined by the period map $\on{Per}$.

Note that for the connected component $\EGb(\colwsur)$ of $\EG(\colwsur)$ in \Cref{thm:EGiso}, there is a corresponding generic-finite component $\on{Stab}^{\bullet}(\pvq)$ of $\on{Stab}(\pvq)$ given by
\[\on{Stab}^{\bullet}(\pvq)=\mathbb C\cdot\bigcup_{\mathcal H\in\mathrm{EG}^{\bullet}(\pvq)}\mathrm{U}(\mathcal H)=\mathbb C\cdot\bigcup_{\mathcal H\in\mathrm{EG}^{\bullet}(\colwsur)}\mathrm{U}(\mathcal H),\]
which is called the \emph{principal component}. We now begin to use the graph isomorphism in \Cref{thm:EGiso} to get an isomorphism of complex manifolds to the principal component $\mathrm{Stab}^{\bullet}(\pvq)$, which is a generic-finite component of $\mathrm{Stab}(\pvq)$. The proof closely follows the arguments in \cite{BS,BMQS,CHQ}.

Now we first recall a technical lemma on stratification, which can be used to expand our isomorphism, following \cite{BS,BMQS}.

Let $\mathrm{FQuad}^{\bullet}(\sow)$ be the connected component of $\mathrm{FQuad}(\sow)$ which is corresponding to the principal component $\mathrm{EG}^{\bullet}(\sow)$ of the exchange graph $\mathrm{EG}(\sow)$. Then there is an increasing chain of subspaces of $\mathrm{FQuad}(\sow)$, given by
\[B_0^{\bullet}=B_1^{\bullet}\subset B_2^{\bullet}\subset\cdots\subset\mathrm{FQuad}^{\bullet}(\sow),\]
where the subspace
\[B_p^{\bullet}:=B_p^{\bullet}(\sow)=\{q\in\mathrm{FQuad}^{\bullet}(\sow)\,|\,r_q+2s_q\le p\}\]
is defined by the number $s_q$ of saddle trajectories and $r_q$ of recurrent trajectories. We define $F_p^{\bullet}=B_p^{\bullet}\backslash B_{p-1}^{\bullet}$.

Let us give some properties of the stratification, including ``walls have ends'', referring to Lemma 4.1, Proposition 4.2, Corollary 4.3 in \cite{BMQS}.

\begin{lemma}
\label{lem:wall has ends}
(1) We have $\mathrm{FQuad}^{\bullet}(\sow)=\mathbb C\cdot B_0^{\bullet}(\sow)$.

(2) If $p>2$ and there is not only one negative order singularity of order $-2$, each component of the stratum $F_p^{\bullet}$ contains a point $\varphi$ and a neighborhood $U\subset\mathrm{FQuad}^{\bullet}(\sow)$ of $\varphi$ such that $U\cap B_p^{\bullet}$ is contained in the locus $\on{Per}(\we)\in\mathbb R$ for some $\we\in \widehat{H}_1(\varphi)$. Moreover, this containment is strict in the sense that $U\cap B_{p-1}^{\bullet}$ is connected.

As a result, any path in $\mathrm{FQuad}^{\bullet}(\sow)$ is homotopic relative to its endpoints to a path in $B_2^{\bullet}$.
\end{lemma}
\subsection{Isomorphisms between moduli spaces}
\begin{theorem}\label{thm:quadstab}
There exists an injective morphism $\iota$ between complex manifolds
\[
\iota:\on{FQuad}^{\bullet}(\colwsur)\to\on{Stab}(\pvq).
\]
Moreover, it is an isomorphism of complex manifolds between $\on{FQuad}^{\bullet}(\colwsur)$ and the image $\on{Stab}^{\bullet}(\pvq)$, the principal component corresponding to $\mathrm{EG}^{\bullet}(\pvq)$.
\end{theorem}
\begin{proof}
The sketch proof follows \cite{BS,BMQS,CHQ}. 

At first, we define the map $\iota_0$ on the saddle-free locus $B_0^{\bullet}$. For any $(C,\we)\in \on{FQuad}^{\bullet}(\colwsur)$, the generic trajectories form a mixed-angulation on $\colwsur$. Dually, the saddle connections form an S-graph $\overline{\mathbb{S}}$ on $\colwsur$. By \Cref{thm:EGiso}, there exists a finite heart $\overline{\mathcal{H}}$ in $\pvq$ corresponding to $\overline{\mathbb{S}}$, which is in the principal part $\mathrm{EG}^{\bullet}(\pvq)$. For the stability function, we define $Z_0(S_i)=\on{Per}(\we_i)$, where $\we_i$ is the closed arc in $\colwsur$ corresponding to $S_i$. Now we get a stability condition $(Z_0,\overline{\mathcal{H}})$ in the principal conponent $\on{Stab}^{\bullet}(\pvq)$.

Then we extend $\iota_0$ to $\iota_2$ on the tame locus $B_2^{\bullet}$ using $\mathbb{C}$-action by (1) in \Cref{lem:wall has ends}. 

Moreover, for the non-tame differentials, we extend $\iota_{p-1}$ from $B_{p-1}^{\bullet}$ to $B_p^{\bullet}$ by (2) in \Cref{lem:wall has ends}, as Proposition 5.5 in \cite{BS} or Theorem 7.2 in \cite{BMQS}.

Finally, the map $\iota$ is injective and an isomorphism to the image $\on{Stab}^{\bullet}(\pvq)$ by \Cref{thm:EGiso}.
\end{proof}
\appendix
\section{Some technical proofs in \Cref{sec:P2}}\label{sec:X-cluster}
\subsection{The explicit expressions of natural transformations $\eta_k$ and $\eta_k'$}\label{sec:A1}
For $d\in\hh{N}_i$ and $k\in\Is$, the evaluation of $H_{\hh{V}}$ and $H_{\hh{N}_k}\circ F_k\circ G_{\phi(i)}$ on $d$ are given by the following diagrams respectively:
\[\begin{tikzpicture}[>=latex,scale=1.5]

    \node (F1G) at (3,3) {$F_1\circ G_i(d)$};
    \node (N1) at (3,1.5) {$\vdots$};
    \node (N2) at (3,2.5) {$\vdots$};
    \node (Gphi) at (1,2) {$G_{\phi(i)}(d)$};
    \node (FkG) at (3,2) {$F_k\circ G_i(d)$};
    \node (Zero) at (5,2) {$0$};
    \node (FmG) at (3,1) {$F_m\circ G_i(d)$};

    \draw[->] (F1G) -- (Gphi);
    \draw[->] (F1G) -- (Zero);
    
    \draw[->] (FkG) -- (Gphi);
    \draw[->] (FkG) -- (Zero);
    
    \draw[->] (FmG) -- (Gphi);
    \draw[->] (FmG) -- (Zero);

\end{tikzpicture}\]
and
\[\begin{tikzpicture}[>=latex,scale=1.5]

    \node (F1G) at (3,3) {$0$};
    \node (N1) at (3,1.5) {$\vdots$};
    \node (N2) at (3,2.5) {$\vdots$};
    \node (Gphi) at (1,2) {$0$};
    \node (FkG) at (3,2) {$F_k\circ G_i(d)$};
    \node (Zero) at (5,2) {$0$.};
    \node (FmG) at (3,1) {$0$};

    \draw[->] (F1G) -- (Gphi);
    \draw[->] (F1G) -- (Zero);
    
    \draw[->] (FkG) -- (Gphi);
    \draw[->] (FkG) -- (Zero);
    
    \draw[->] (FmG) -- (Gphi);
    \draw[->] (FmG) -- (Zero);

\end{tikzpicture}\]
Hence we can easily see the natural transformation $\eta_k: H_{\hh{N}_k}\circ F_k\circ G_{\phi(i)}\to H_{\hh{V}}$ upon each position.
Similarly, the evaluation of $H_{\hh{N}_k}\circ F_{\psi(k)}'\circ G_{\psi(k)}'$ and $H_{\hh{V}',k}$ on $d$ are given by the following diagrams respectively
\[\begin{tikzpicture}[>=latex,scale=1.5]

    \node (F1G) at (3,3) {$0$};
    \node (N1) at (3,1.5) {$\vdots$};
    \node (N2) at (3,2.5) {$\vdots$};
    \node (Gphi) at (1,2) {$0$};
    \node (FkG) at (3,2) {$F_{\psi(k)}'\circ G_{\psi(k)}'(d)$};
    \node (Zero) at (5,2) {$0$};
    \node (FmG) at (3,1) {$0$};

    \draw[->] (F1G) -- (Gphi);
    \draw[->] (F1G) -- (Zero);
    
    \draw[->] (FkG) -- (Gphi);
    \draw[->] (FkG) -- (Zero);
    
    \draw[->] (FmG) -- (Gphi);
    \draw[->] (FmG) -- (Zero);

\end{tikzpicture}\]
and
\[\begin{tikzpicture}[>=latex,scale=1.5]

    \node (F1G) at (3,3) {$F_{\psi(1)}'\circ G_{\psi(k)}'(d)$};
    \node (N1) at (3,1.5) {$\vdots$};
    \node (N2) at (3,2.5) {$\vdots$};
    \node (Gphi) at (1,2) {$0$};
    \node (FkG) at (3,2) {$F_{\psi(k)}'\circ G_{\psi(k)}'(d)$};
    \node (Zero) at (5,2) {$G_{\psi(k)}'(d),$};
    \node (FmG) at (3,1) {$F_{\psi(m)}'\circ G_{\psi(k)}'(d) $};

    \draw[->] (F1G) -- (Gphi);
    \draw[->] (F1G) -- (Zero);
    
    \draw[->] (FkG) -- (Gphi);
    \draw[->] (FkG) -- (Zero);
    
    \draw[->] (FmG) -- (Gphi);
    \draw[->] (FmG) -- (Zero);

\end{tikzpicture}\]
which implies the natural transformation $\eta'_k:H_{\hh{N}_k}\simeq H_{\hh{N}_k}\circ F_{\psi(k)}'\circ G_{\psi(k)}'\to H_{\hh{V}',k}$.
\subsection{Checking that $\overline{G}_i$ factors through $\lim\F'$}\label{sec:A2}
To verify that $\overline{G}_i$ factors through $\lim\F'$, it suffices to check that for each $d\in\hh{N}_i$, the evaluation of $\overline{G}_i(d)$ at each edge of $\Gs$ lies in $\lim\F'$, i.e. that the corresponding map is a Cartesian section. For the edge $v\to e_j$ where $j\in\Is$, we obtain that $\overline{G}_i(d)(v\to e_j)$ is the colimit of the following diagram
\begin{equation}\label{eq:check1}
\begin{tikzpicture}[>=latex,scale=1.2]

    \node (Gd) at (3,3) {$G_{\phi(i)}(d)$};
    \node (FGd) at (10,3) {$F_j\circ G_{\phi(i)}(d)$};

    \node (Z1) at (1,2) {$0$};
    \node (Z15) at (2,2) {$\cdots$};
    \node (Z2) at (3,2) {$0$};
    \node (Z25) at (4,2) {$\cdots$};
    \node (Z3) at (5,2) {$0$};
    \node (Z35) at (7,2) {$0$};
    \node (Z36) at (8,2) {$\cdots$};
    \node (Z4) at (10,2) {$F_j\circ G_{\phi(i)}(d)$};
    \node (Z45) at (12,2) {$\cdots$};
    \node (Z5) at (13,2) {$0.$};

    \node (Z7) at (3,1) {$0$};
    \node (FGF) at (10,1) {$F_{\psi(j)}'\circ G_{\psi(m)}'\circ F_m\circ G_{\phi(i)}(d)$};
    
    \draw[<-] (Gd) -- (Z2);
    \draw[<-] (Gd) -- (Z3);
    \draw[->] (FGd) -- (Gd);
    \draw[->] (Z1) -- (Gd);
    \draw[->] (Z1) -- (Z7);
    \draw[->] (Z2) -- (Z7);
    \draw[->] (Z3) -- (Z7);
    \draw[->] (Z4) -- (FGd) node[midway,right] {$\simeq$};
    \draw[->] (Z4) -- (FGF) node[midway,right] {$\alpha$};
    \draw[->] (Z35) -- (FGd);
    \draw[->] (Z5) -- (FGd);
    \draw[->] (Z35) -- (FGF);
    \draw[->] (Z5) -- (FGF);
    \draw[<-] (Z7) -- (FGF);
    \draw[->,red] (Z35)to[bend right=10] (Z1);
    \draw[->,blue] (Z4)to[bend right=10] (Z2);
    \draw[->,green] (Z5)to[bend right=10] (Z3);
\end{tikzpicture}
\end{equation}
Since $\alpha$ is equivalent to identity if $j=m$ and otherwise $\alpha$ is zero, then it
can be simplified as $G_{\phi(i)}(d)\leftarrow F_j\circ G_{\phi(i)}(d)$, which is a Cartesian edge. Similarly, for the edge $v'\to e_{\psi(j)}'$ where $j\in\Is$, we have that $\overline{G}_i(d)(v'\to e_{\psi(j)}')$ is the colimit of the following diagram
\begin{equation}\label{eq:check2}
\begin{tikzpicture}[>=latex,scale=1.2]

    \node (Gd) at (3,3) {$0$};
    \node (FGd) at (10,3) {$F_j\circ G_{\phi(i)}(d)$};

    \node (Z1) at (1,2) {$0$};
    \node (Z15) at (2,2) {$\cdots$};
    \node (Z2) at (3,2) {$0$};
    \node (Z25) at (4,2) {$\cdots$};
    \node (Z3) at (5,2) {$0$};
    \node (Z35) at (7,2) {$0$};
    \node (Z36) at (8,2) {$\cdots$};
    \node (Z4) at (10,2) {$F_j\circ G_{\phi(i)}(d)$};
    \node (Z45) at (12,2) {$\cdots$};
    \node (Z5) at (13,2) {$0.$};

    \node (Z7) at (3,1) {$G_{\psi(m)}'\circ F_{m} \circ G_{\phi(i)}d)$};
    \node (FGF) at (10,1) {$F_{\psi(j)}'\circ G_{\psi(m)}'\circ F_m\circ G_{\phi(i)}(d)$};
    
    \draw[<-] (Gd) -- (Z2);
    \draw[<-] (Gd) -- (Z3);
    \draw[->] (FGd) -- (Gd);
    \draw[->] (Z1) -- (Gd);
    \draw[->] (Z1) -- (Z7);
    \draw[->] (Z2) -- (Z7);
    \draw[->] (Z3) -- (Z7);
    \draw[->] (Z4) -- (FGd) node[midway,right] {$\simeq$};
    \draw[->] (Z4) -- (FGF) node[midway,right] {$\beta$};
    \draw[->] (Z35) -- (FGd);
    \draw[->] (Z5) -- (FGd);
    \draw[->] (Z35) -- (FGF);
    \draw[->] (Z5) -- (FGF);
    \draw[<-] (Z7) -- (FGF);
    \draw[->,red] (Z35)to[bend right=10] (Z1);
    \draw[->,blue] (Z4)to[bend right=10] (Z2);
    \draw[->,green] (Z5)to[bend right=10] (Z3);
\end{tikzpicture}
\end{equation}
Similarly, since $\beta$ is equivalent to identity if $j=m$ and otherwise $\beta$ is zero, then it
can be simplified as $G_{\psi(m)}'\circ F_{m} \circ G_{\phi(i)}(d)\leftarrow F_{\psi(j)}'\circ G_{\psi(m)}'\circ F_{m}\circ G_{\phi(i)}(d)$, which is a Cartesian edge.
\subsection{Proof of \Cref{prop:limscho}}\label{sec:A3}
Before proving this, we mention the following lemma, which gives a way to reduce the computation of (co)limit diagram indexed by a complicated simplicial set to that indexed by its decomposed ingredients. The proof of the lemma can be found in Proposition 7.5.8.12 in \cite{KL}. We consider the following simplicial set
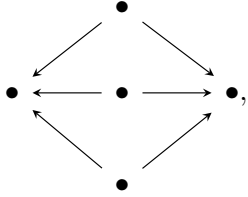
\[\begin{tikzcd}
        & \bullet \arrow[ld] \arrow[rd] &         \\
\bullet & \bullet \arrow[l] \arrow[r]   & \bullet, \\
        & \bullet \arrow[lu] \arrow[ru] &        
\end{tikzcd}\]
which is denoted by $K$. It can be decomposed to m pieces of $(\bullet\longleftarrow\bullet\longrightarrow\bullet)$, which is denoted by $K_i, 1\leq i\leq m$ respectively. Let $\C$ be the small category with objects $\{K_i, 1\leq i\leq m\}$ and all morphisms are generated by: 
\[\begin{tikzcd}
K_i: &\bullet \arrow[d, equal] & \bullet \arrow[r] \arrow[l] & \bullet \arrow[d] \\
K_{i+1}: &\bullet                       & \bullet \arrow[l] \arrow[r] & \bullet,          
\end{tikzcd}\]
where the upper row is $K_i$ and the lower one is $K_{i+1}$ for $1\leq i\leq m-1$.
\begin{lemma}\label{lem:deplim}
Let $\hh{A}$ be an $\infty$-category equipped with a diagram $q: K\to \hh{A}$. Then an object $X\in\hh{A}$ is a colimit of the diagram $q$ if and only if it is a colimit of $Q:N_{\bullet}(\C)\to\hh{A}$. In other words, if the following diagram
\[\begin{tikzcd}
  & Y_1 \arrow[ld] \arrow[rd] &   \\
X & Y_2 \arrow[l] \arrow[r]   & Z \\
  & Y_m \arrow[lu] \arrow[ru] &  
\end{tikzcd}\]
admits a coproduct, then the coproduct in $\hh{A}$ is equivalent to the pushout of 
\[\begin{tikzcd}
Y_1\oplus\cdots\oplus Y_m \arrow[r] \arrow[d] & Z \\
X                                             &  
\end{tikzcd}\]
with $X, Y_i, Z\in\hh{A}, 1\leq i\leq m$.
\end{lemma}

Now we back to the proof. To prove \Cref{prop:limscho}, we will check the five conditions in \Cref{def:schberngon} one by one with $i\in\overline{I}$.

\textbf{1. $\overline{F}_i$ and $\overline{G}_i$ are adjoints to each other. Moreover, the counit $\overline{F}_i\circ \overline{G}_i\to \on{id}_{\hh{N}_i}$ is an equivalence.}

We only need to verify the statement in the case when $i\in\overline{I}$ where $\phi(i)\in I\backslash\Is$ and the other case is similar. Since that exact functors preserve finite colimit, so we can form the left adjoint of $\overline{G}_i$ via the left adjoints at each position of \eqref{eq:lim} and then taking the colimit. The left adjoint of $H_{\hh{V}}$ is given by
\begin{equation}\label{eq:lahv}
\hh{L}\xrightarrow{\res_{\hh{V}}}\hh{V}\xrightarrow{F_{\phi(i)}}\hh{N}. 
\end{equation}
For $k\in\Is$, the left adjoint of $H_{\hh{N}_k}\circ F_k\circ G_{\phi(i)}$ is given by 
\begin{equation}\label{eq:lahn}
\hh{L}\xrightarrow{\res_{\hh{N}_k}}\hh{N}\xrightarrow{^LF_k}\hh{V}\xrightarrow{F_{\phi(i)}}\hh{N}.
\end{equation}
The left adjoint of $H_{\hh{V}',1}\circ F_1\circ G_{\phi(i)}$ is given by 
\begin{equation}\label{eq:lahv'}
\hh{L}\xrightarrow{\res_{\hh{V'}}}\hh{V'}\xrightarrow{F'_{\psi(1)}}\hh{N}\xrightarrow{^LF_1}\hh{V}\xrightarrow{F_{\phi(i)}}\hh{N}.
\end{equation}
By combining \Cref{lem:deplim} with the fact that \eqref{eq:lahn} and \eqref{eq:lahv'} are equivalent when restricted to Cartesian sections, we deduce that the left adjoint of $\overline{G}_i$ is \eqref{eq:lahv}, which is equivalent to $\overline{F}_i$. Moreover, since $\overline{F}_i$ preserves colimit, the composition of $\overline{F}_i$ with $\overline{G}_i$ can be given by composing $\overline{F}_i$ with the functor at each position of \eqref{eq:lim} and then taking the limit. We have that
\[\overline{F}_i\circ H_{\hh{V}}: \hh{N}\xrightarrow{G_{\phi(i)}}\hh{V}\to\hh{E}_{\hh{V}}\subseteq\hh{L}\xrightarrow{\res_{\hh{V}}}\hh{V}\xrightarrow{F_{\phi(i)}}\hh{N}\]
is equivalent to $\id_{\hh{N}}$, since the composition of the middle three maps is equivalent to $\id_{\hh{V}}$ and $F_{\phi(i)}\circ G_{\phi(i)}\simeq \id_{\hh{N}}.$ For any $k\in \Is$, we also have that
\[\overline{F}_i\circ H_{\hh{N}_k}\circ F_k\circ G_{\phi(i)}:\hh{N}\xrightarrow{G_{\phi(i)}}\hh{V}\xrightarrow{F_k}\hh{N}\to\hh{E}_{k}\subseteq\hh{L}\xrightarrow{\res_{\hh{V}}}\hh{V}\xrightarrow{F_{\phi(i)}}\hh{N}\]
and
\[\overline{F}_i\circ H_{\hh{V'},m}\circ F_m\circ G_{\phi(i)}: \hh{N}\xrightarrow{G_{\phi(i)}}\hh{V}\xrightarrow{F_m}\hh{N}\xrightarrow{G_{\psi(m)}'}\hh{V'}\to\hh{E}_{\hh{V'}}\subseteq\hh{L}\xrightarrow{\res_{\hh{V}}}\hh{V}\xrightarrow{F_{\phi(i)}}\hh{N}\]
are equivalent. By \Cref{lem:deplim}, we deduce that
\[\overline{F}_i\circ\overline{G}_i\simeq\id_{\hh{N}}.\] 
\textbf{2. $\overline{F}_i\circ\overline{G}_{i+1}$ is an equivalence of $\infty$-categories.}

If both $\phi(i)$ and $\phi(i+1)$ are in $I\backslash\Is$ or $I'\backslash\Is'$, the proof is the same as in part 1. Otherwise, we can still get $\overline{F}_i\circ\overline{G}_{i+1}$ by each position in the colimit diagram \eqref{eq:lim}. We may assume that $\phi(i)\in I'\backslash\Is'$ and $\phi(i+1)\in I\backslash\Is$. We have that
\[\overline{F}_i\circ H_{\hh{V}}: \hh{N}\xrightarrow{G_{\phi(i+1)}}\hh{V}\to\hh{E}_{\hh{V}}\subseteq\hh{L}\xrightarrow{\res_{\hh{V'}}}\hh{V'}\xrightarrow{F'_{\phi(i)}}\hh{N}\]
and
\[\overline{F}_i\circ H_{\hh{N}_k}\circ F_k\circ G_{\phi(i+1)}:\hh{N}\xrightarrow{G_{\phi(i+1)}}\hh{V}\xrightarrow{F_k}\hh{N}\to\hh{E}_{k}\subseteq\hh{L}\xrightarrow{\res_{\hh{V}'}}\hh{V}'\xrightarrow{F'_{\phi(i)}}\hh{N}\]
are equivalent for any $k\in\Is$. Also, we have that
\[\overline{F}_i\circ H_{\hh{V'},m}\circ F_m\circ G_{\phi(i+1)}: \hh{N}\xrightarrow{G_{\phi(i+1)}}\hh{V}\xrightarrow{F_m}\hh{N}\xrightarrow{G_{\psi(m)}'}\hh{V'}\to\hh{E}_{\hh{V'}}\subseteq\hh{L}\xrightarrow{\res_{\hh{V}'}}\hh{V'}\xrightarrow{F'_{\phi(i)}}\hh{N}\]
is an equivalence since $m=\phi(i+1)-1$ and in this case $\psi(m)=\phi(i)+1$. Since the composition from the second-to-last map to the fourth-to-last map is equivalent to $\id_{\hh{V}'}$ by construction and both $F_{\phi(i+1)-1}\circ G_{\phi(i+1)}$ and $F'_{\phi(i)}\circ G'_{\phi(i)+1}$ are equivalences, we deduce that
$\overline{F}_i\circ\overline{G}_{i+1}$ is an equivalence of $\infty$-categories.   

\textbf{3. $\overline{F}_i\circ\overline{G}_j\simeq 0$ if $j$ does not equal to $i$ or $i+1$.}

The proof is similar to part 1 and part 2. If both $\phi(i)$ and $\phi(j)$ are in $I\backslash\Is$ or $I'\backslash\Is'$, the proof is the same as in part 1, where $\overline{F}_i\circ\overline{G}_j$ is equivalent to $\overline{F}_i\circ H_{\hh{V}}$, which is zero if $j\neq i, i + 1$. Otherwise, we assume that $\phi(i)\in I'\backslash\Is'$ and $\phi(j)\in I\backslash\Is$, then the proof is same as in part 2. We have that $\overline{F}_i\circ\overline{G}_j$ is equivalent to
\[\overline{F}_i\circ H_{\hh{V'},m}\circ F_m\circ G_{\phi(j)}: \hh{N}\xrightarrow{G_{\phi(j)}}\hh{V}\xrightarrow{F_m}\hh{N}\xrightarrow{G_{\psi(m)}'}\hh{V'}\to\hh{E}_{\hh{V'}}\subseteq\hh{L}\xrightarrow{\res_{\hh{V}'}}\hh{V'}\xrightarrow{F'_{\phi(i)}}\hh{N},\]
which is equivalent to zero since that at least one of $F_m\circ G_{\phi(i)}$ and $G_{\psi(m)}'\circ F_{\phi(i)}'$ is equivalent to zero from the conditions. Therefore, $\overline{F}_i\circ\overline{G}_j$ is equivalent to zero if $j$ does not equal to $i$ or $i+1$.

\textbf{4. $\overline{G}_i$ has a right adjoint $(\overline{G}_i)^R$ and $\overline{F}_i$ has a left adjoint $^L\overline{F}_i$.} 

The right adjoint of $\overline{G}_i$ can be obtained by forming the right adjoints at each position of \eqref{eq:lim} and then taking the colimit. To construct the left adjoint $^L\overline{F}_i$ of $\overline{F}_i$, we use the fact that $\lim\F'$ is equivalent to the coCartesian sections of covariant Grothendieck construction. We first refine the definitions of $\hh{E}_{\hh{V}},\hh{E}_k$ and $\hh{E}_{\hh{V'}}$ by taking the corresponding relative left Kan extension and then replace $G_{\phi(i)}$ or $G_{\phi(i)}'$ with the left adjoints of $F_{\phi(i)}$ or $F_{\phi(i)}'$ respectively as described in \Cref{const:collapadj}. Consequently, the left adjoint is defined as the colimit of a diagram analogous to \eqref{eq:lim}. The verification follows similarly to part 1.

\textbf{5. $\im(\overline{G}_{i+1})=\im(^L\overline{F}_i)$ as full subcategories of $\lim\F'$.}

The image can be evaluated pointwise in \eqref{eq:lim} and then taking the colimit. If both $\phi(i)$ and $\phi(i+1)$ are in $I\backslash\Is$ or $I'\backslash\Is'$, it holds because of the definition of perverse schobers on the $n$-spider $v$ and $n'$-spider $v'$ respectively. Otherwise, we may assume that $\phi(i)\in I'\backslash\Is'$ and $\phi(i+1)\in I\backslash\Is$. We consider the image of $\overline{G}_{i+1}$ and $^L\overline{F}_i$ as follows. For $X\in\hh{N}$, we have that $\overline{G}_{i+1}(X)$ is represented by the local diagram:
\[\tiny \begin{tikzcd}[column sep=10pt]
            &                  & \begin{tabular}{c} $F_1\circ G_{\phi(i+1)}(X)$\\$\simeq 0$\end{tabular} \arrow[ld] \arrow[rd] &                                             &                                                                          \\
\begin{tabular}{c} $X\simeq$\\ $F_{\phi(i+1)}\circ G_{\phi(i+1)}(X)$\end{tabular}  \arrow[r] & G_{\phi(i+1)}(X) & \begin{tabular}{c} $F_j\circ G_{\phi(i+1)}(X)$\\$\simeq 0$\end{tabular} \arrow[l] \arrow[r]                                                               & \begin{tabular}{c} $G'_{\psi(m)}\circ F_m $\\$ \circ G_{\phi(i+1)}(X)$\end{tabular} & \begin{tabular}{c} $F_{\phi(i)}'\circ G'_{\psi(m)}\circ$\\$ F_m\circ G_{\phi(i+1)}(X).$\end{tabular} \arrow[l] \\
            &                  & \begin{tabular}{c} $ F_m\circ G_{\phi(i+1)}(X)\simeq $\\ $F_{\psi(m)}'\circ G'_{\psi(m)}\circ F_m\circ G_{\phi(i+1)}(X) $\end{tabular} \arrow[lu] \arrow[ru]                                                             &                                             &                                                                         
\end{tikzcd}\]
For $Y\in\hh{N}$, we have that $^L\overline{F}_i(Y)$ is represented by the local diagram:
\[\tiny\begin{tikzcd}
                                                                &                                                                                        & F'_{\psi(1)}\circ \prescript{L}{}F_{\phi(i)}'(Y)\simeq 0                                                     &                                                             &                                               \\
\begin{tabular}{c} $F_{\phi(i+1)}\circ G_m\circ$\\$ F'_{\psi(m)}\circ \prescript{L}{}F_{\phi(i)}'(Y)$\end{tabular} & \begin{tabular}{c} $G_m\circ F'_{\psi(m)}$\\$\circ \prescript{L}{}F_{\phi(i)}'(Y)$\end{tabular} \arrow[ru] \arrow[l] \arrow[r] \arrow[rd] & F'_{\psi(j)}\circ \prescript{L}{}F_{\phi(i)}'(Y)\simeq 0                                                     & \prescript{L}{}F_{\phi(i)}'(Y) \arrow[lu] \arrow[l] \arrow[r] \arrow[ld] & F'_{\psi(i)}\circ \prescript{L}{}F_{\phi(i)}'(Y) \simeq Y. \\
                                                                &                                                                                        & \begin{tabular}{c} $F_m\circ G_m\circ F'_{\psi(m)}\circ \prescript{L}{}F_{\phi(i)}'(Y)$\\$\simeq F'_{\psi(m)}\circ \prescript{L}{}F_{\phi(i)}'(Y)$\end{tabular} &                                                             &                                              
\end{tikzcd}\]
If we assume that $F_{\phi(i)}'\circ G'_{\psi(m)}\circ F_m\circ G_{\phi(i+1)}(X)\simeq Y$, then we have
\[^LF_{\phi(i)}'\circ F_{\phi(i)}'\circ G'_{\psi(m)}\circ F_m\circ G_{\phi(i+1)}(X)\simeq~ ^LF_{\phi(i)}'(Y).\]
Since $^LF_{\phi(i)}'$ is fully faithful, we have the equivalence
\[^LF_{\phi(i)}'\circ F_{\phi(i)}'\circ~^LF_{\phi(i)}'\simeq~^LF_{\phi(i)}',\]
where implies that the counit map $^LF_{\phi(i)}'\circ F_{\phi(i)}'\to \id$ is an equivalence on $\im(^LF_{\phi(i)}')$. By the assumption that $\im(^LF_{\phi(i)}')=\im(G'_{\phi(i)+1})=\im(G'_{\psi(m)})$, we obtain the following
\[^LF_{\phi(i)}'(Y)\simeq~^LF_{\phi(i)}'\circ F_{\phi(i)}'\circ G'_{\psi(m)}\circ F_m\circ G_{\phi(i+1)}(X)\simeq G'_{\psi(m)}\circ F_m\circ G_{\phi(i+1)}(X).\]
By similar reason, we have the equivalences at other positions. Therefore, we have $\im(\overline{G}_{i+1})$ and $\im(^L\overline{F}_i)$.

Combining 1-5 above, we obtain the \Cref{prop:limscho}.
\section{Flipping and simple tilting of type II, III and IV}\label{sec:B}
In this section, we illustrate the flipping process of \Cref{prop:refine_of_flip} for type III and type IV by an example in the following figures and summarize the step-by-step simple tilting processes of types II and III in the following tables.

In \Cref{fig:3.4.3} and \Cref{tab:1}, the arc obtained by counting clockwise from $\overline{\gamma}$ that forms a bigon together with $\overline{\gamma}$ may be either a boundary segment of $\colwsur$ or an open arc in $\colwsur$. We consider the case where it is an open arc, since the boundary case is simpler and similar. In this case, we denote this open arc by $\overline{\xi}$, whose dual closed arc is $\overline{\delta}$ and the corresponding object is $\overline{T}$. Let $S_i$ be the object associated to the dual arc of $\wg_i$. After the sequence of tilting, we have $S_1,\ldots,S_m$ are in $\pvd(\subsur)$ and the images of $S_m'[1]$ and $T'$ in $\D(\colwsur)$ are isomorphic $\overline{S}[1]$ and $\phi_{\overline{S}}(\overline{T})$ respectively. In \Cref{tab:2} (resp. \Cref{tab:3}), we have that $S_1$ is (resp. $S_1$ and $S_2$ are) in $\pvd(\subsur)$ and the images of $S_2'[1]$ (resp. $S_1'[1]$) and $X''$ in $\D(\colwsur)$ are isomorphic $\overline{S}[1]$ and $\phi_{\overline{S}}(\overline{X})$ respectively.

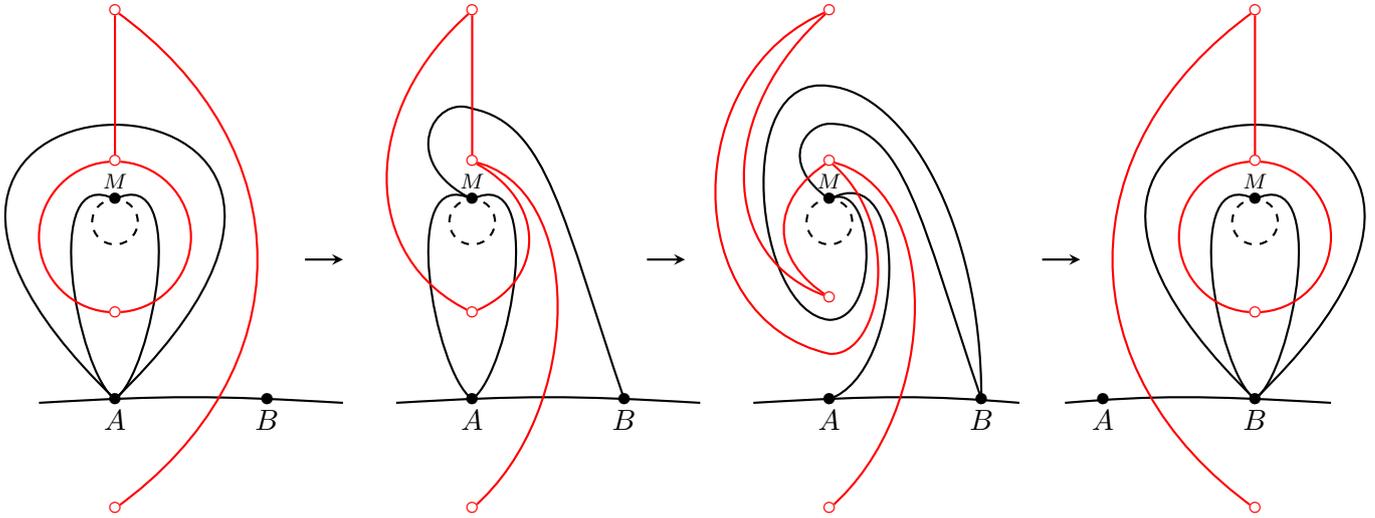
\begin{figure}[ht]\centering
  \makebox[\textwidth][cc]{
\begin{tikzpicture}[scale=1]
\begin{scope}[shift={(-5,0)}]
\draw[thick]
    (-1,-.9) .. controls (1-0.3,-0.8) and (1.3,-0.8) .. (3,-.9);
  \draw[thick]
    (0,-0.85) 
      .. controls (-5,4) and (5,4) .. (0,-0.85);
\draw[thick](0,1.8)\nn node[above,font=\scriptsize]{$M$};
\draw[thick] (0,-.85)\nn node[below]{$A$};
\draw[thick] (2,-.85)\nn node[below]{$B$};
\draw[thick] (0,-.85) .. controls (-.5,-.5) and +(-1,.5) .. (0,1.8);
\draw[thick] (0,-.85) .. controls (.5,-.5) and +(1,.5) .. (0,1.8);
\draw[thick,dashed] (0,1.5) circle (0.3);
\draw[thick,red] (0,1.3) circle (1);
\draw[thick,red] (0,2.3) -- (0,4.3);
\draw[thick,red] (0,-2.3) .. controls (2.5,-.5) and (2.5,2.5) .. (0,4.3);
\draw(0,.3)\ww(0,2.3)\ww (0,4.3)\ww (0,-2.3)\ww;
\draw(2.5,1)edge[thick,->,>=stealth](3,1);
\begin{scope}[shift={(4.7,0)}]
\draw[thick]
    (-1,-.9) .. controls (1-0.3,-0.8) and (1.3,-0.8) .. (3,-.9);
\draw[thick](0,1.8)\nn node[above,font=\scriptsize]{$M$};
\draw[thick] (0,-.85)\nn node[below]{$A$};
\draw[thick] (2,-.85)\nn node[below]{$B$};
\draw[thick]
        (0,1.8)
        .. controls (-1,2.3) and (-0.5,3.2) ..
        (0,3)
        .. controls (1,2.8) and (1.2,1.5) ..
        (2,-0.85);
\draw[thick] (0,-.85) .. controls (-.5,-.5) and +(-1,.5) .. (0,1.8);
\draw[thick] (0,-.85) .. controls (.5,-.5) and +(1,.5) .. (0,1.8);
\draw[thick,dashed] (0,1.5) circle (0.3);
\draw[thick,red] (0,2.3) -- (0,4.3);
\draw[thick,red] (0,.3) .. controls (1,.7) and (1,1.8) .. (0,2.3);
\draw[thick,red] (0,-2.3) .. controls (1.5,-1) and (1.5,2) .. (0,2.3);
\draw[thick,red] (0,.3) .. controls (-1.5,1) and (-1.5,3) .. (0,4.3);
\draw(0,.3)\ww(0,2.3)\ww (0,4.3)\ww (0,-2.3)\ww;
\end{scope}
\draw(7,1)edge[thick,->,>=stealth](7.5,1);
\begin{scope}[shift={(9.4,0)}]
\draw[thick]
    (-1,-.9) .. controls (1-0.3,-0.8) and (1.3,-0.8) .. (2.5,-.9);
\draw[thick](0,1.8)\nn node[above,font=\scriptsize]{$M$};
\draw[thick] (0,-.85)\nn node[below]{$A$};
\draw[thick] (2,-.85)\nn node[below]{$B$};
\draw[thick]
        (0,1.8)
        .. controls (-.7,2.3) and (-0.3,2.8) ..
        (0,2.8)
        .. controls (1,2.8) and (1.2,1.5) ..
        (2,-0.85);
\draw[thick]
        (0,1.8) .. controls +(.6,.3) and (.7,0.2) .. (0,0.2)
        .. controls (-1,0.3) and (-1.3,3.5) .. (0,3.3)
        .. controls (1,3.2) and (2,1.5) .. (2,-0.85);
\draw[thick] (0,-.85) .. controls (1,-.5) and +(1.1,.6) .. (0,1.8);
\draw[thick,dashed] (0,1.5) circle (0.3);
\draw[thick,red]
        (0,2.3) .. controls (1,1.5) and (.7,-0.3) .. (0,-0.25)
        .. controls (-2,.2) and (-2,3.5) .. (0,4.3);
\draw[thick,red] (0,.5) .. controls (-.8,1) and (-.8,1.8) .. (0,2.3);
\draw[thick,red] (0,-2.3) .. controls (1.5,-1) and (1.5,2) .. (0,2.3);
\draw[thick,red] (0,.5) .. controls (-1.5,1) and (-1.5,3) .. (0,4.3);
\draw(0,.5)\ww(0,2.3)\ww (0,4.3)\ww (0,-2.3)\ww;
\end{scope}
\draw(12.2,1)edge[thick,->,>=stealth](12.7,1);
\begin{scope}[shift={(13,0)}]
\draw[thick]
    (-.5,-.9) .. controls (1-0.3,-0.8) and (1.3,-0.8) .. (3,-.9);
\draw[thick] (0,-.85)\nn node[below]{$A$};
\draw[thick] (2,-.85)\nn node[below]{$B$};
\end{scope}
\begin{scope}[shift={(15,0)}]
  \draw[thick]
    (0,-0.85) 
      .. controls (-5,4) and (5,4) .. (0,-0.85);
\draw[thick](0,1.8)\nn node[above,font=\scriptsize]{$M$};
\draw[thick] (0,-.85) .. controls (-.5,-.5) and +(-1,.5) .. (0,1.8);
\draw[thick] (0,-.85) .. controls (.5,-.5) and +(1,.5) .. (0,1.8);
\draw[thick,dashed] (0,1.5) circle (0.3);
\draw[thick,red] (0,1.3) circle (1);
\draw[thick,red] (0,2.3) -- (0,4.3);
\draw[thick,red] (0,-2.3) .. controls (-2.5,-.5) and (-2.5,2.5) .. (0,4.3);
\draw(0,.3)\ww(0,2.3)\ww (0,4.3)\ww (0,-2.3)\ww;
\end{scope}
\end{scope}
\end{tikzpicture}}
    \caption{The forward flip of type III, where there are $\wg_1$ and $\wg_2$}
    \label{fig:1}
\end{figure}
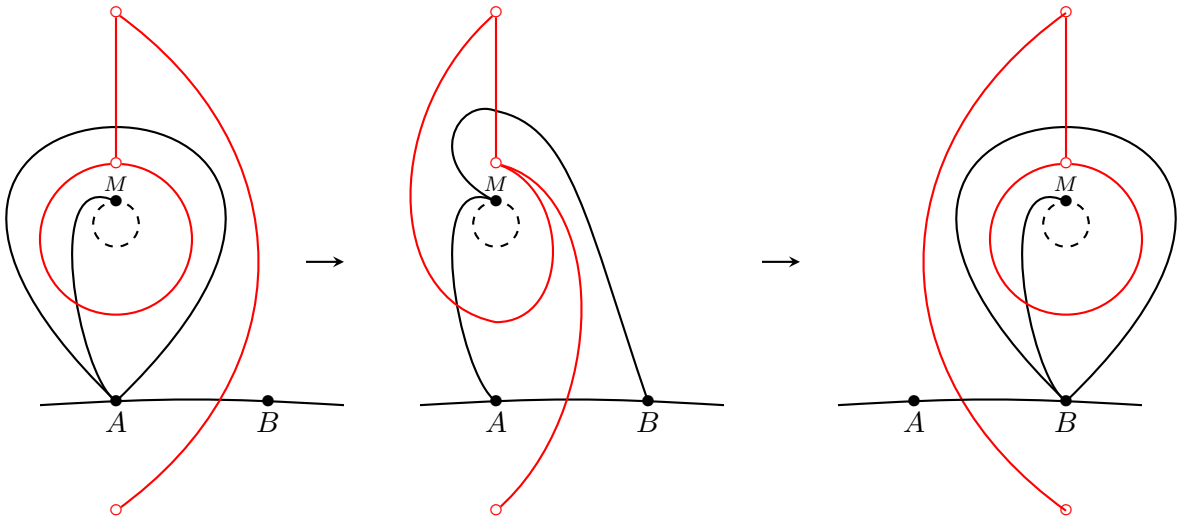
\begin{figure}[ht]\centering
\makebox[\textwidth][c]{
\begin{tikzpicture}[scale=1]
\draw[thick]
    (-1,-.9) .. controls (1-0.3,-0.8) and (1.3,-0.8) .. (3,-.9);
  \draw[thick]
    (0,-0.85) 
      .. controls (-5,4) and (5,4) .. (0,-0.85);
\draw[thick](0,1.8)\nn node[above,font=\scriptsize]{$M$};
\draw[thick] (0,-.85)\nn node[below]{$A$};
\draw[thick] (2,-.85)\nn node[below]{$B$};
\draw[thick] (0,-.85) .. controls (-.5,-.5) and +(-1,.5) .. (0,1.8);
\draw[thick,dashed] (0,1.5) circle (0.3);
\draw[thick,red] (0,1.3) circle (1);
\draw[thick,red] (0,2.3) -- (0,4.3);
\draw[thick,red] (0,-2.3) .. controls (2.5,-.5) and (2.5,2.5) .. (0,4.3);
\draw(0,2.3)\ww (0,4.3)\ww (0,-2.3)\ww;
\draw(2.5,1)edge[thick,->,>=stealth](3,1);
\begin{scope}[shift={(5,0)}]
\draw[thick]
    (-1,-.9) .. controls (1-0.3,-0.8) and (1.3,-0.8) .. (3,-.9);
\draw[thick](0,1.8)\nn node[above,font=\scriptsize]{$M$};
\draw[thick] (0,-.85)\nn node[below]{$A$};
\draw[thick] (2,-.85)\nn node[below]{$B$};
\draw[thick]
        (0,1.8)
        .. controls (-1,2.3) and (-0.5,3.2) ..
        (0,3)
        .. controls (1,2.8) and (1.2,1.5) ..
        (2,-0.85);
\draw[thick] (0,-.85) .. controls (-.5,-.5) and +(-1,.5) .. (0,1.8);
\draw[thick,dashed] (0,1.5) circle (0.3);
\draw[thick,red] (0,2.3) -- (0,4.3);
\draw[thick,red] (0,2.3) .. controls (1,2) and (1,.2) .. (0,.2).. controls (-1.5,.5) and (-1.5,3) .. (0,4.3);
\draw[thick,red] (0,-2.3) .. controls (1.5,-1) and (1.5,2) .. (0,2.3);
\draw(0,2.3)\ww (0,4.3)\ww (0,-2.3)\ww;
\end{scope}
\draw(8.5,1)edge[thick,->,>=stealth](9,1);
\begin{scope}[shift={(12.5,0)}]
  \draw[thick]
    (0,-0.85) 
      .. controls (-5,4) and (5,4) .. (0,-0.85);
\draw[thick] (0,-.85) .. controls (-.5,-.5) and +(-1,.5) .. (0,1.8);
\draw[thick,dashed] (0,1.5) circle (0.3);
\draw[thick,red] (0,1.3) circle (1);
\draw[thick,red] (0,2.3) -- (0,4.3);
\draw(0,2.3)\ww (0,4.3)\ww (0,-2.3)\ww;
\draw[thick](0,1.8)\nn node[above,font=\scriptsize]{$M$};
\draw[thick,red] (0,-2.3) .. controls (-2.5,-.5) and (-2.5,2.5) .. (0,4.3);
\end{scope}
\begin{scope}[shift={(10.5,0)}]
\draw[thick]
    (-1,-.9) .. controls (1-0.3,-0.8) and (1.3,-0.8) .. (3,-.9);
\draw[thick] (0,-.85)\nn node[below]{$A$};
\draw[thick] (2,-.85)\nn node[below]{$B$};
\end{scope}
\end{tikzpicture}}
    \caption{The forward flip of type III, where there is $\wg_1$}
    \label{fig:2}
\end{figure}

\begin{figure}\centering
 \makebox[\textwidth][c]{
\begin{tikzpicture}[xscale=1.1]
\begin{scope}[shift={(-15,0)}]
\draw[thick]
    (-0.5,-.9) .. controls (1-0.3,-0.8) and (1.3,-0.8) .. (2.5,-.9);
  \draw[thick,cyan]
    (0,-0.85) 
      .. controls (-4,4) and (4,4) .. (0,-0.85);
\draw[thick,orange]  (0,-0.85)  .. controls (.55,1.5) and (1.5,1.3) .. (0,-0.85);
\draw[thick,red]  (0,-0.85)  .. controls (-2.2,2.5) and (.5,2.4) .. (0,-0.85);
\filldraw[fill=white,draw=none] (-.2,1.25) circle (0.15);
\draw[thick]  (0,-0.85)  .. controls (-1,2.4) and (1,2.4) .. (0,-0.85);
\draw (0,3.1) node{$\wg=\wg_0$} ;
\draw[thick] (0,-.85)\nn node[below]{$A$};
\draw[thick] (2,-.85)\nn node[below]{$B$};

\draw[-stealth] (3.5,0) -- (4.5,0);
\draw (4,0) node[above]{$\wg_0$};

\begin{scope}[shift={(6.3,0)}]
\draw[thick]
    (-0.5,-.9) .. controls (1-0.3,-0.8) and (1.3,-0.8) .. (2.5,-.9);
  \draw[thick,cyan]
    (0,-0.85) 
      .. controls (.5,3) and (4,3) .. (2,-0.85);
\filldraw[fill=white,draw=none] (.7,1.1) circle (0.15);
\draw[thick,orange]  (0,-0.85)  .. controls (1,1) and (3,1) .. (0,-0.85);
\draw[thick,red]  (0,-0.85)  .. controls (-2.2,2.5) and (.5,2.4) .. (0,-0.85);
\filldraw[fill=white,draw=none] (-.1,1.1) circle (0.15);
\draw[thick]  (0,-0.85)  .. controls (-1,2.4) and (2,2.4) .. (0,-0.85);
\draw[thick] (0,-.85)\nn node[below]{$A$};
\draw[thick] (2,-.85)\nn node[below]{$B$};
\draw (-.5,1.6) node[above]{$\wg_1$};
\end{scope}

\draw[-stealth] (-2,-4.5) -- (-1,-4.5);
\draw (-1.5,-4.5) node[above]{$\wg_1$};

\begin{scope}[shift={(0,-4.5)}]
\draw[thick]
    (-0.5,-.9) .. controls (1-0.3,-0.8) and (1.3,-0.8) .. (2.5,-.9);
  \draw[thick,cyan]
    (0,-0.85) 
      .. controls (-.5,3) and (3,3) .. (2,-0.85);
\filldraw[fill=white,draw=none] (1.9,1.5) circle (0.15);
\draw[thick,orange]  (0,-0.85)  .. controls (1,1) and (3,1) .. (0,-0.85);
\draw[thick,red]  (0,-0.85)  .. controls (1.5,2.5) and (4,2.4) .. (2,-0.85);
\filldraw[fill=white,draw=none] (.2,1.25) circle (0.15);
\draw[thick]  (0,-0.85)  .. controls (-1,2.4) and (1,2.4) .. (0,-0.85);
\draw[thick] (0,-.85)\nn node[below]{$A$};
\draw[thick] (2,-.85)\nn node[below]{$B$};
\draw (0,1.6) node[above]{$\wg_2$};
\end{scope}

\draw[-stealth] (3.5,-4.5) -- (4.5,-4.5);
\draw (4,-4.5) node[above]{$\wg_2$};

\begin{scope}[shift={(6,-4.5)}]
\draw[thick]
    (-0.5,-.9) .. controls (1-0.3,-0.8) and (1.3,-0.8) .. (2.5,-.9);
\draw[thick,orange]  (0,-0.85)  .. controls (0,1) and (2,1) .. (0,-0.85);
\draw[thick]  (2,-0.85)  .. controls (1,2.5) and (3.5,2.4) .. (2,-0.85);
\draw[thick,cyan]  (0,-0.85)  .. controls (-1,2.4) and (1,2.4) .. (2,-0.85);
\filldraw[fill=white,draw=none] (1.9,1.5) circle (0.15);
\filldraw[fill=white,draw=none] (.4,1.5) circle (0.15);
 \draw[thick,red]
    (0,-0.85) 
      .. controls (-.5,3) and (3,3) .. (2,-0.85);
\draw[thick] (0,-.85)\nn node[below]{$A$};
\draw[thick] (2,-.85)\nn node[below]{$B$};
\draw (-0,1.7) node[above]{$\wg_3$};
\end{scope}

\draw[-stealth] (-2,-9) -- (-1,-9);
\draw (-1.5,-9) node[above]{ $\wg_3$};

\begin{scope}[shift={(0,-9)}]
\draw[thick]
    (-0.5,-.9) .. controls (1-0.3,-0.8) and (1.3,-0.8) .. (2.5,-.9);
\draw[thick,orange]  (0,-0.85)  .. controls (0,1) and (2,1) .. (0,-0.85);
\draw[thick]  (2,-0.85)  .. controls (.5,2.5) and (4,2.4) .. (2,-0.85);
\filldraw[fill=white,draw=none] (1.8,1.4) circle (0.15);
\filldraw[fill=white,draw=none] (2.6,.9) circle (0.15);
\draw[thick,cyan]  (2,-0.85)  .. controls (2.5,2.5) and (5,1.5) .. (2,-0.85);
 \draw[thick,red]
    (0,-0.85) 
      .. controls (-.5,3) and (2.5,3) .. (2,-0.85);
\draw[thick] (0,-.85)\nn node[below]{$A$};
\draw[thick] (2,-.85)\nn node[below]{$B$};
\draw (1,2) node[above]{$\wg_4$};
\end{scope}

\draw[-stealth] (3.5,-9) -- (4.5,-9);
\draw (4,-9) node[above]{$\wg_4$};

\begin{scope}[shift={(8,-9)}]
  \draw[thick,red]
    (-2,-0.85) 
      .. controls (-4,3.5) and (2.5,3.5) .. (0,-0.85);
\draw[thick]  (0,-0.85)  .. controls (-2.2,2.5) and (.5,2.4) .. (0,-0.85);
\filldraw[fill=white,draw=none] (-.2,1.25) circle (0.15);
\draw[thick,cyan]  (0,-0.85)  .. controls (-1,2.4) and (1,2.4) .. (0,-0.85);
\end{scope}

\begin{scope}[shift={(6,-9)}]
\draw[thick,orange]  (0,-0.85)  .. controls (-.5,.5) and (-2,.5) .. (0,-0.85);
\draw[thick]
    (-0.5,-.9) .. controls (1-0.3,-0.8) and (1.3,-0.8) .. (2.5,-.9);
\draw[thick] (0,-.85)\nn node[below]{$A$};
\draw[thick] (2,-.85)\nn node[below]{$B$};
\draw (-1,.3) node[above]{$\wg_5$};
\end{scope}

\begin{scope}[shift={(2,-13.5)}]
  \draw[thick,red]
    (-2,-0.85) 
      .. controls (-4,3.5) and (2.5,3.5) .. (0,-0.85);
\draw[thick,orange]  (0,-0.85)  .. controls (1.2,1.2) and (3,.8) .. (0,-0.85);
\draw[thick]  (0,-0.85)  .. controls (-2.2,2.5) and (.5,2.4) .. (0,-0.85);
\filldraw[fill=white,draw=none] (-.2,1.25) circle (0.15);
\draw[thick,cyan]  (0,-0.85)  .. controls (-1,2.4) and (1,2.4) .. (0,-0.85);
\end{scope}

\begin{scope}[shift={(0,-13.5)}]
\draw[thick]
    (-0.5,-.9) .. controls (1-0.3,-0.8) and (1.3,-0.8) .. (2.5,-.9);
\draw[thick] (0,-.85)\nn node[below]{$A$};
\draw[thick] (2,-.85)\nn node[below]{$B$};
\draw (1,2.4) node[above]{$\wg_6$};
\end{scope}

\draw[-stealth] (-2,-13.5) -- (-1,-13.5);
\draw (-1.5,-13.5) node[above]{$\wg_5$};

\begin{scope}[shift={(8,-13.5)}]
  \draw[thick,red]
    (0,-0.85) 
      .. controls (-4,4) and (4,4) .. (0,-0.85);
\draw[thick,orange]  (0,-0.85)  .. controls (.55,1.5) and (1.5,1.3) .. (0,-0.85);
\draw[thick]  (0,-0.85)  .. controls (-2.2,2.5) and (.5,2.4) .. (0,-0.85);
\filldraw[fill=white,draw=none] (-.2,1.25) circle (0.15);
\draw[thick,cyan]  (0,-0.85)  .. controls (-1,2.4) and (1,2.4) .. (0,-0.85);
\end{scope}

\begin{scope}[shift={(6,-13.5)}]
\draw[thick]
    (-0.5,-.9) .. controls (1-0.3,-0.8) and (1.3,-0.8) .. (2.5,-.9);
\draw[thick] (0,-.85)\nn node[below]{$A$};
\draw[thick] (2,-.85)\nn node[below]{$B$};
\end{scope}

\draw[-stealth] (3.5,-13.5) -- (4.5,-13.5);
\draw (4,-13.5) node[above]{$\wg_6$};

\end{scope}
\end{tikzpicture}}
    \caption{The forward flip of type IV, an example}
    \label{fig:3}
\end{figure}
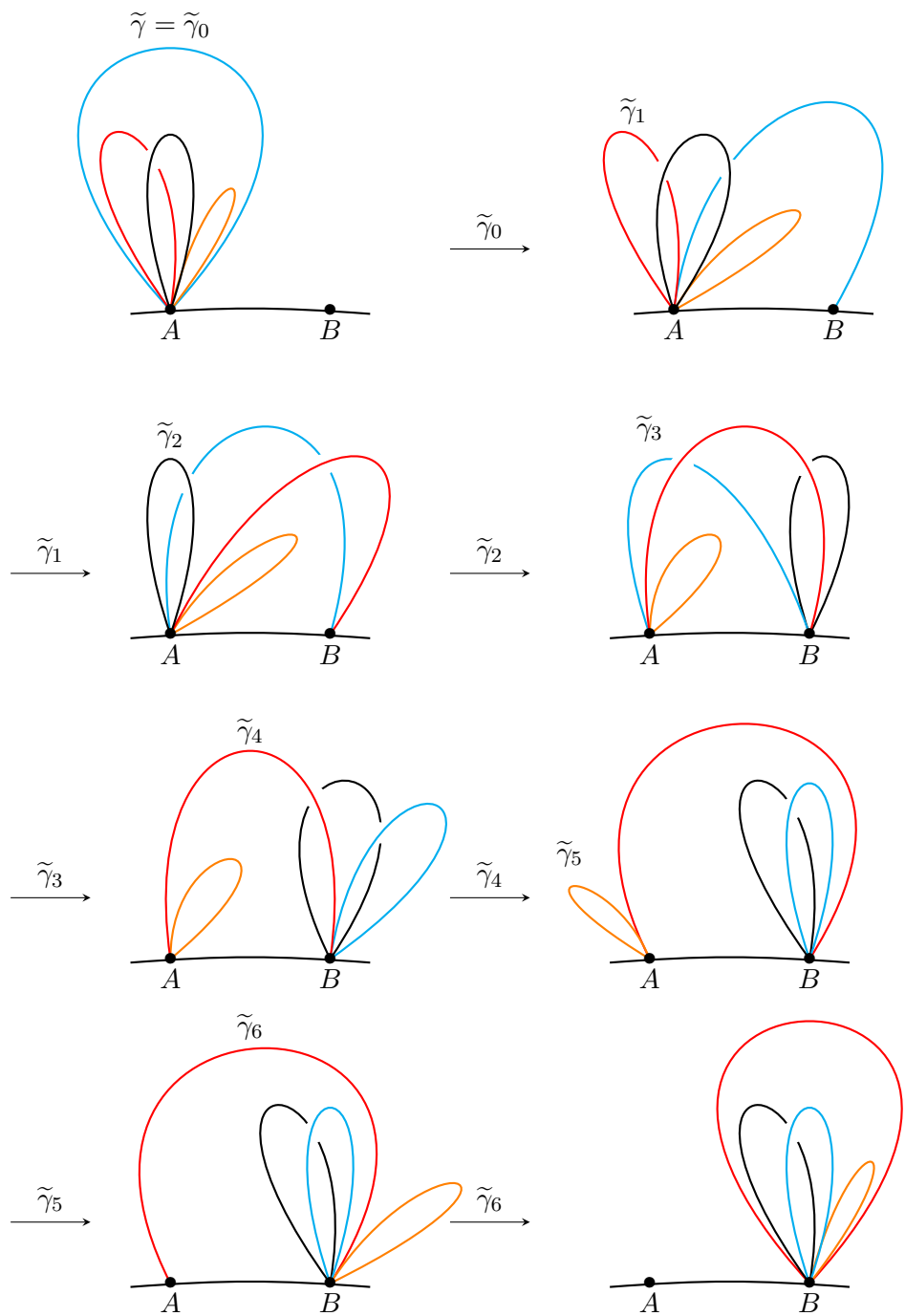

\begin{table}[ht]
    \renewcommand{\arraystretch}{1.5}
    \centering
    \begin{tabular}{|c|c|c|c|c|c|c|}
    \hline
        initial & after flip via $S$ & after flip via $S_1'$ & \multirow{11}{*}{$\cdots$} & after flip via $S_k'$ & \multirow{11}{*}{$\cdots$} & after flip via $S_m'$ \\ \cline{1-3}\cline{5-5}\cline{7-7}
        $X$ & $S\to X'\to X$ & $X'$ &  & $X'$ &  & $X'$ \\ \cline{1-3}\cline{5-5}\cline{7-7}
        $S$ & $S[1]$ & $S_1'\to S_1\to S[1]$  &  & $S_1$ &  & $S_1$ \\ \cline{1-3}\cline{5-5}\cline{7-7}
        $S_1$ & $S\to S_1'\to S_1$ & $S_1'[1]$  &  & $S_2$ &  & $S_2$ \\ \cline{1-3}\cline{5-5}\cline{7-7}
        $S_2$ & $S_2$ & $S_1'\to S_2'\to S_2$  &  & $S_3$ &  & $S_3$ \\ \cline{1-3}\cline{5-5}\cline{7-7}
        $\vdots$ & $\vdots$ & $\vdots$  &  & $\vdots$ &  & $\vdots$ \\ \cline{1-3}\cline{5-5}\cline{7-7}
        $S_k$ & $S_k$ & $S_k$  &  & $S_k'[1]$ &  & $S_{k+1}$ \\ \cline{1-3}\cline{5-5}\cline{7-7}
        $S_{k+1}$ & $S_{k+1}$ & $S_{k+1}$  &  & $S_k'\to S_{k+1}'\to S_{k+1}$ &  & $S_{k+2}$ \\ \cline{1-3}\cline{5-5}\cline{7-7}
        $\vdots$ & $\vdots$ & $\vdots$  &  & $\vdots$ & & $\vdots$ \\ \cline{1-3}\cline{5-5}\cline{7-7}
        $S_m$ & $S_m$ & $S_m$ &  & $S_m$ &  & $S_m'[1]$ \\ \cline{1-3}\cline{5-5}\cline{7-7}
        $T$ & $T$ & $T$ &  & $T$ &  & $S_m'\to T'\to T$ \\ \hline
    \end{tabular}
    \vspace{3mm}
    \caption{The simple tilting process in type II}
    \label{tab:1}
\end{table}

\begin{table}[ht]
    \renewcommand{\arraystretch}{1.5}
    \centering
    \begin{tabular}{|c|c|c|c|}
    \hline
        initial & after flip via $S$ & after flip via $S_1'$ & after flip via $S_2'$ \\ \hline
        $S$ & $S[1]$ & $S_1'\to S_1\to S[1]$  & $S_1$\\ \hline
        $S_1$ & $S\to S_1'\to S_1$ & $S_1'[1]$  & $S_2$\\ \hline
        $S_2$ & $S_2$ & $S_1'\to S_2'\to S_2$  & $S_2'[1]$ \\ \hline
        $X$ & $S\to X'\to X$ & $X'$ & $S_2'\to X''\to X'$ \\ \hline
    \end{tabular}
    \vspace{3mm}
    \caption{The simple tilting process in type III}
    \label{tab:2}
\end{table}

\begin{table}[ht]
    \renewcommand{\arraystretch}{1.5}
    \centering
    \begin{tabular}{|c|c|c|}
    \hline
        initial & after flip via $S$ & after flip via $S_1'$ \\ \hline
        $S$ & $S[1]$ & $S_1'\to S_1\to S[1]$  \\ \hline
        $S_1$ & $S\to S_1'\to S_1$ & $S_1'[1]$ \\ \hline
        $X$ & $S\to X'\to X$ & $S_1'\to X''\to X$ \\ \hline
    \end{tabular}
    \vspace{3mm}
    \caption{The simple tilting process in type III, continued}
    \label{tab:3}
\end{table}
\clearpage

\end{document}